\documentclass[10pt]{amsart}
\usepackage{amsmath}
\usepackage{array}
\usepackage{caption}
\usepackage[urlcolor=blue,citecolor=blue,linkcolor=blue,colorlinks=true,pdfencoding=auto]{hyperref}

\usepackage[pdf]{pstricks}
\usepackage{auto-pst-pdf}
\usepackage{color}
\usepackage{pst-plot}
\usepackage{float,lscape}
\input{multido.tex}

\usepackage{amsrefs,amssymb}
\usepackage{bmpsize}
\usepackage{enumerate,calc,lscape}
\usepackage[matrix,arrow,curve,frame]{xy}   

\usepackage{longtable,booktabs}

\newcommand{\cA}{\ensuremath{\mathcal{A}}}
\renewcommand{\C}{\ensuremath{\mathbb{C}}}

\newcommand{\cE}{\ensuremath{\mathcal{E}}}

\newcommand{\F}{\ensuremath{\mathbb{F}}}

\newcommand{\rH}{\ensuremath{\mathrm{H}}}
\newcommand{\bK}{\ensuremath{\mathbb{K}}}

\newcommand{\bM}{\ensuremath{\mathbb{M}}}

\newcommand{\bP}{\ensuremath{\mathbb{P}}}

\newcommand{\R}{\ensuremath{\mathbb{R}}}

\newcommand{\Z}{\ensuremath{\mathbb{Z}}}

\newcommand{\smsh}{\wedge}
\newcommand{\into}{\hookrightarrow}

\newcommand{\rtarr}{\longrightarrow}
\newcommand{\ltarr}{\longleftarrow}
\newcommand{\xrtarr}{\xrightarrow}
\newcommand{\iso}{\cong}

\newcommand{\bracket}[1]{ \left\langle #1 \right\rangle}

\newcommand{\cla}{\mathrm{cl}}

\newcommand{\Sq}{\ensuremath{\mathrm{Sq}}}
\newcommand{\gr}{\ensuremath{\mathrm{gr}}}
\newcommand{\coB}{\mathrm{coB}}

\newcommand{\SpRmot}{\ensuremath{\mathbf{Sp}^\R}}
\newcommand{\SpCmot}{\ensuremath{\mathbf{Sp}^\C}}
\newcommand{\SpC}{\ensuremath{\mathbf{Sp}^{C_2}}}
\newcommand{\Sp}{\ensuremath{\mathbf{Sp}}}

\newcommand{\arho}{\rho}

\DeclareMathOperator{\Hom}{Hom}

\DeclareMathOperator{\Ext}{Ext}
\DeclareMathOperator{\Tor}{Tor}
\DeclareMathOperator*{\colim}{colim}
\DeclareMathOperator*{\holim}{holim}

\newcounter{themyfigure}

\newenvironment{myfigure}
	{\refstepcounter{themyfigure}}
    {}

\newcommand{\MTwocrad}{0.11}
\newcommand{\SqOneLine}{\psline{->}(0.6,0.5)(1.42,0.5)}
\newcommand{\SqTwoLine}{\psline[linecolor=blue]{->}(0.6,0.55)(2.42,1.46)}
\newcommand{\SqFourLine}{\psbezier[linecolor=orange]{->}(0.5,0.5)(2,0.9)(3,1.4)(4.47,2.445)}
\newcommand{\smallcrad}{0.1}
\newcommand{\crad}{0.12}
\newcommand{\bigcrad}{0.2}
\newcommand{\Bigcrad}{0.27}

\newrgbcolor{pcolor}{0.4 0.4 0.4}
\newrgbcolor{hzerocolor}{0.4 0.4 0.4}
\newrgbcolor{honecolor}{0.4 0.4 0.4}

\newrgbcolor{tauzerocolor}{0.4 0.4 0.4}
\newrgbcolor{tauonecolor}{0.3 0.7 0.3}

\newrgbcolor{hiddentaucolor}{0.3 0.7 0.3}
\newrgbcolor{hiddentwotaucolor}{0 0 1}
\newrgbcolor{taufourextncolor}{1 0.5 0.2}

\newrgbcolor{MWetacolor}{0 0 1}
\newrgbcolor{MWrhocolor}{1 0 0}
\newrgbcolor{MWalphacolor}{0 0.5 0}

\newgray{gridline}{0.8}

\newcommand{\cirrad}{0.11}

\newcommand{\labelrad}{0.1}

\newcolumntype{C}{>{$}c<{$}}
\newcolumntype{L}{>{$}l<{$}}

\begin{document}
\title{
The cohomology of $C_2$-equivariant $\cA(1)$ and the homotopy of 
$\MakeLowercase{ko}_{C_2}$}
\author{B. J. Guillou}
\email{bertguillou@uky.edu}
\address{Department of Mathematics, The University of Kentucky, Lexington, KY}
\author{M. A. Hill}
\email{mikehill@math.ucla.edu}
\address{Department of Mathematics, University of California, Los Angeles, Los Angelas, CA}
\author{D. C. Isaksen}
\email{isaksen@wayne.edu}
\address{Department of Mathematics, Wayne State University, Dertoit, MI}
\author{D. C. Ravenel}
\email{doug@math.rochester.edu}
\address{Department of Mathematics, University of Rochester, Rochester, NY}


\thanks{The authors were supported by NSF grants DMS-1710379, DMS-1509652, DMS-1606290, and DMS-1606623 and by Simons Collaboration Grant No. 282316.}
\

\begin{abstract}
We compute the cohomology of the subalgebra $\cA^{C_2}(1)$
of the $C_2$-equivariant Steenrod algebra $\cA^{C_2}$.
This serves as the input to 
the $C_2$-equivariant Adams spectral sequence 
converging to the completed $RO(C_2)$-graded homotopy groups of 
an equivariant spectrum $ko_{C_2}$. 
Our approach is to use
simpler $\C$-motivic and $\R$-motivic calculations
as stepping stones.
\end{abstract}

\numberwithin{equation}{section}

\newtheorem{thm}{Theorem}[section]
\newtheorem{prop}[thm]{Proposition}
\newtheorem{lemma}[thm]{Lemma}
\newtheorem{cor}[thm]{Corollary}
\newtheorem{conj}[thm]{Conjecture}
\newtheorem{quest}{Question}
\theoremstyle{definition}
\newtheorem{rmk}[thm]{Remark}
\newtheorem{eg}[thm]{Example}
\newtheorem{defn}[thm]{Definition}
\newtheorem{notn}[thm]{Notation}

\newenvironment{pf}{\begin{proof}}{\end{proof}}

\newcounter{marker}

\maketitle

\tableofcontents

\section{Introduction}

The $RO(G)$-graded homotopy groups are among the most fundamental
invariants of the stable $G$-equivariant homotopy category.
This article is a first step towards systematic application of
the equivariant Adams spectral sequence to calculate these groups.

Araki and Iriye \cite{AI} \cite{Iriye} computed much information about
the $C_2$-equivariant stable homotopy groups using EHP-style
techniques in the spirit of Toda \cite{Toda}.  Our approach is
entirely independent from theirs.

We work only with the two-element group $C_2$ because it is the most
elementary non-trivial case.  In order to compute $C_2$-equivariant
stable homotopy groups of the $C_2$-equivariant sphere spectrum
using the Adams spectral sequence, one
needs to work with the full $C_2$-equivariant Steenrod algebra
$\cA^{C_2}$ for the constant Mackey functor $\underline{\F}_2$.  
As the $C_2$-equivariant Eilenberg--Mac~Lane
spectrum for $\underline{\F}_2$ is flat (\cite[Corollary~6.45]{HK})
 the $E_2$-term of the Adams spectral sequence 
is given by the cohomology of the equivariant Steenrod algebra. 
In this article, we tackle a computationally simpler situation by working
over the subalgebra $\cA^{C_2}(1)$.  This means that we are computing
the $C_2$-equivariant stable homotopy groups not of the sphere but of
the $C_2$-equivariant analogue of connective real $K$-theory $ko$.  We
will explicitly construct this $C_2$-equivariant spectrum $ko_{C_2}$ in Section
\ref{sctn:koC2}.

Our calculational program is carried out for $\cA^{C_2}(1)$ in this
article as a warmup for the full Steenrod algebra $\cA$ to be studied
in future work.  Roughly speaking, $\cA$ contains Steenrod squaring
operations $\Sq^{i}$ with the expected properties, and $\cA^{C_2}(1)$
is the subalgebra generated by $\Sq^{1}$ and $\Sq^{2}$. A key point is
that our program works just as well in theory for $\cA^{C_2}$ as for
$\cA^{C_2}(1)$, except that the details are even more complicated.
It remains to be seen how far this can be carried out in practice.

Our strategy is to build up to the complexity of the $C_2$-equivariant
situation by first studying the $\C$-motivic and $\R$-motivic
situations.  
The relevant stable homotopy categories
are related by functors as in the diagram
\[ \xymatrix{
\mathrm{Ho}(\SpRmot) \ar[r]^{-\otimes_\R \C} \ar[d]_{\mathbf{Re}} &  \mathrm{Ho}(\SpCmot) \ar[d]^{\mathbf{Re}} \\
\mathrm{Ho}(\SpC) \ar[r]^{\iota^*} & \mathrm{Ho}(\Sp).
}\]
The vertical functors are Betti realization (see \cite[Section~4.4]{HO}). 
The functor $\iota^*$ restricts an equivariant spectrum to the 
trivial subgroup, yielding the underlying spectrum.

The $\C$-motivic cohomology of a point is equal to
$\F_2[\tau]$ \cite{VoevHF2} (see also \cite[Section~2.1]{DImass}).
The $\C$-motivic Steenrod algebra
$\cA^{\C}$ is very similar to the classical Steenrod algebra, but
there are some small complications related to $\tau$.  In particular,
these complications allow the element $h_1$ in the cohomology of
$\cA^\C$ to be non-nilpotent,
detecting the non-nilpotence of the motivic Hopf map $\eta_\C$ 
(\cite[Corollary~6.4.5]{Mo}). 
In the cohomology of $\cA^\C(1)$, the
non-nilpotence of $h_1$ is essentially the only difference to the
classical case.

The $\R$-motivic cohomology of a point is equal to $\F_2[\tau, \rho]$
\cite{VoevHF2}  (again, see the discussion in \cite[Section~2.1]{DImass}). 
 Now an additional complication enters because
$\Sq^1(\tau) = \rho$.  The computation of the cohomology of the
$\R$-motivic Steenrod algebra $\cA^\R$ becomes more difficult because
the cohomology of a point is a non-trivial $\cA^\R$-module.  In
addition, the $\R$-motivic Steenrod algebra $\cA^\R$ has additional
complications associated with terms involving higher powers of $\rho$
\cite[Theorem~12.6]{Voev}.

A natural way to avoid this problem is to filter by powers of $\rho$.
In the associated graded object, $\Sq^1(\tau)$ becomes zero
and the associated graded Hopf algebroid is simply
the $\C$-motivic Hopf algebra with an adjoined polynomial generator $\rho$.
Therefore, the $\rho$-Bockstein spectral sequence starts from the
cohomology of $\cA^\C$ and converges to the cohomology of $\cA^\R$.

This $\rho$-Bockstein spectral sequence has lots of differentials and
hidden extensions.  Nevertheless, a complete calculation for $\cA^\R(1)$
is reasonable.  A key point is to first carry out the $\rho$-inverted
calculation.  This turns out to be much simpler.  With a priori
knowledge of the $\rho$-inverted calculation in hand, there is just
one possible pattern of $\rho$-Bockstein differentials.

Relying on our experience from the $\R$-motivic situation, we are now
ready to tackle the $C_2$-equivariant situation.  The
$C_2$-equivariant cohomology of a point contains $\F_2[\tau, \rho]$,
but there is an additional ``negative cone" that is infinitely
divisible by both $\tau$ and $\rho$ \cite[Prop~6.2]{HK}.  Except for the
complications in the cohomology of a point, the $C_2$-equivariant
Steenrod algebra $\cA^{C_2}$ is no more complicated than the
$\R$-motivic one \cite[pp. 386--387]{HK}.

Again, a $\rho$-Bockstein spectral sequence allows us to compute the
cohomology of $\cA^{C_2}(1)$.  Because of infinite
$\tau$-divisibility, the starting point of the spectral sequence is
more complicated than just the cohomology of $\cA^\C(1)$.  Once
identified, this issue presents only a minor difficulty.

The $\rho$-inverted calculation determines the part of 
the cohomology of $\cA^{C_2}(1)$
that supports infinitely many $\rho$ multiplications.
Dually, it is also helpful to determine in advance 
the part of 
the cohomology of $\cA^{C_2}(1)$
that is infinitely $\rho$-divisible,
i.e., the inverse limit of an infinite tower of $\rho$-multiplications.
We anticipate that 
this approach via infinitely $\rho$-divisible classes will
be essential in the more complicated calculation over the
full Steenrod algebra $\cA^{C_2}$, to be studied in future work.

As for the $\R$-motivic case, the $\rho$-Bockstein spectral sequence
is manageable, even though it does have lots of differentials
and hidden extensions.

All of these calculations lead to a thorough understanding of
the cohomology of $\cA^{C_2}(1)$.
The charts in Section \ref{sctn:chart} display the
calculation graphically.

The next step is to consider the $C_2$-equivariant 
Adams spectral sequence.  For degree reasons, there are no
non-zero Adams differentials.  The same simple situation occurs 
in the classical, $\C$-motivic, and $\R$-motivic cases.  

However, it turns out that
there are many hidden extensions to be analyzed.
The presence of so many hidden extensions suggests that
the Adams filtration may not be optimal for 
equivariant purposes.
Unfortunately, we do not have an alternative to propose.

The final description of the homotopy groups is complicated.
Nevertheless, our computation establishes
that the homotopy of $ko_{C_2}$ is nearly periodic
(see Theorem~\ref{thm:PhiPerd}). 
We refer to Section \ref{sec:homotopy} and the charts in
Section \ref{sctn:chart} for details.

We are very grateful to the referees for close readings of our paper and detailed commentary.

\subsection{Organization}

In Section \ref{sctn:Ext}, we provide the basic algebraic input to our 
calculation by thoroughly describing the $C_2$-equivariant
cohomology of a point and the $C_2$-equivariant Steenrod algebra
$\cA^{C_2}$.
In Section \ref{sctn:Bockstein}, we set up the $\rho$-Bockstein spectral
sequence, which is our main tool for
computing 
the cohomology of $\cA^{C_2}(1)$.
In Sections \ref{sctn:rho-inverted} and \ref{sec:RhoComp}, 
we 
carry out the $\rho$-inverted and the infinitely
$\rho$-divisible calculations.
In Section \ref{sec:ExtAR1}, we carry out the $\R$-motivic
$\rho$-Bockstein spectral sequence as a warmup for the
$C_2$-equivariant $\rho$-Bockstein spectral sequence
in Section \ref{sec:BockNegCone}.
Section \ref{sctn:Massey} provides some information about
Massey products in the $C_2$-equivariant cohomology
of $\cA(1)$, which is used in Section \ref{sec:hidden}
to determine multiplicative structure that is hidden
by the $\rho$-Bockstein spectral sequence.
Section \ref{sctn:koC2} gives the construction of the
$C_2$-equivariant spectrum whose homotopy groups are
computed by the cohomology of $\cA^{C_2}(1)$,
and Section \ref{sec:homotopy} analyzes multiplicative structure
in these homotopy groups that is hidden by the
Adams spectral sequence.
Finally, Section \ref{sctn:chart} includes a series 
of charts that graphically describe our calculation.

\subsection{Notation}\label{sec:notn}

We employ notation as follows:
\begin{enumerate}
\item
$\bM_2^\C=\F_2[\tau]$ 
is the motivic cohomology of $\C$ with $\F_2$ coefficients, where $\tau$ has bidegree $(0,1)$.
\item
$\bM^\R_2=\F_2[\tau,\rho]$ 
is the motivic cohomology of $\R$ with $\F_2$ coefficients, where $\tau$ and $\rho$ have bidegrees $(0,1)$ and $(1,1)$, respectively.
\item
$\bM^{C_2}_2$ is the bigraded equivariant cohomology of a point with coefficients in the constant Mackey functor $\underline{\F}_2$. See Section \ref{subsctn:cohomology-point} for a description of this algebra. 
\item
$NC$ is the ``negative cone" part of $\bM^{C_2}_2$.  
See Section \ref{subsctn:cohomology-point} for a precise description.
\item
$H^{*,*}_{C_2}(X)$ is the $C_2$-equivariant cohomology of $X$,
with coefficients in the constant Mackey functor $\underline{\F}_2$.
\item
$\cA^\cla$, $\cA^\C$, $\cA^\R$, and $\cA^{C_2}$ are the classical,
$\C$-motivic, $\R$-motivic, and $C_2$-equivariant mod 2 Steenrod algebras.
\item
$\cA^\cla(n)$, $\cA^\C(n)$, $\cA^\R(n)$, and
$\cA^{C_2}(n)$ are the classical, $\C$-motivic, $\R$-motivic,
and $C_2$-equivariant subalgebras generated by
$\Sq^1, \Sq^2, \Sq^4, \ldots, \Sq^{2^n}$.
\item
$\cE^{C_2}(1)$ is the subalgebra of $\cA^{C_2}$ generated by $Q_0 = \Sq^1$
and $Q_1 = \Sq^1 \Sq^2 + \Sq^2 \Sq^1$.
\item
$\Ext_\cla$ is the bigraded ring $\Ext_{\cA^{\cla}} (\F_2, \F_2)$,
i.e., the cohomology of $\cA^{\cla}$.
\item
$\Ext_\C$ is the trigraded ring $\Ext_{\cA^\C}(\bM_2^\C,\bM_2^\C)$,
i.e., the cohomology of $\cA^{\C}$.
\item
$\Ext_\R$ is the trigraded ring $\Ext_{\cA^\R}(\bM^\R_2,\bM^\R_2)$,
i.e., the cohomology of $\cA^{\R}$.
\item
$\Ext_{C_2}$ is the trigraded ring $\Ext_{\cA^{C_2}}(\bM^{C_2}_2,\bM^{C_2}_2)$,
i.e., the cohomology of $\cA^{C_2}$.
\item
$\Ext_{NC}$ is the $\Ext_{\cA^{\R}}$-module
$\Ext_{\cA^{\R}}(NC, \bM_2^\R)$.
\item
$\Ext_\cla(n)$ is the bigraded ring $\Ext_{\cA^{\cla}(n)} (\F_2, \F_2)$,
i.e., the cohomology of $\cA^{\cla}(n)$.
\item
$\Ext_\C(n)$ is the trigraded ring 
$\Ext_{\cA^\C(n)}(\bM^\C_2,\bM^\C_2)$,
i.e., the cohomology of $\cA^{\C}(n)$.
\item 
$\Ext_\R(n)$ is the trigraded ring 
$\Ext_{\cA^\R(n)}(\bM^\R_2,\bM^\R_2)$,
i.e., the cohomology of $\cA^{\R}(n)$.
\item 
$\Ext_{C_2}(n)$ is the trigraded ring 
$\Ext_{\cA^{C_2}(n)}(\bM^{C_2}_2,\bM^{C_2}_2)$,
i.e., the cohomology of $\cA^{C_2}(n)$.
\item 
$\Ext_{NC}(n)$ is the
$\Ext_{\R}(n)$-module
$\Ext_{\cA^\R(n)}(NC,\bM^\R_2)$.
\item
$E^+$ is the $\rho$-Bockstein spectral sequence
\[
\Ext_\C(1)[\rho] \Rightarrow \Ext_\R(1).
\]
See Section~\ref{sctn:Bockstein}.
\item
$E^-$ is the $\rho$-Bockstein spectral sequence
that converges to $\Ext_{NC}(1)$. See Section~\ref{sctn:Bockstein}.
\item
$\frac{\F_2[x]}{x^\infty} \{ y \}$ is the infinitely $x$-divisible module
$\colim_n \F_2[x]/x^n$, consisting of elements of the form
$\frac{y}{x^k}$ for $k \geq 1$.
See Remark \ref{rmk:theta-notation}.
\item
$ko_{C_2}$ is a $C_2$-equivariant spectrum such that
$H^{*,*}_{C_2}(ko_{C_2}) \iso
\cA^{C_2} \!/\!/ \cA^{C_2}(1)$. 
See Section~\ref{sctn:koC2}.
\item
${\pi}_{*,*}(X)$ are the bigraded 
$C_2$-equivariant stable homotopy groups of $X$, completed at $2$ 
so that the equivariant Adams spectral sequence converges.
\item
$\Pi_n(X)$ is the Milnor-Witt $n$-stem 
$\bigoplus\limits_p \pi_{p+n,p}$.

\end{enumerate}

We use grading conventions that are
common in motivic homotopy theory
but less common in equivariant homotopy theory.
In equivariant homotopy theory,
$RO(C_2)\iso \Z[\sigma]/(\sigma^2 - 1)$ 
is the real representation ring of $C_2$, 
where $\sigma$ is the $1$-dimensional sign representation.
The main points of translation are:
\begin{enumerate}
\item
Equivariant degree $p + q \sigma$ will be expressed, according to
the motivic convention, as
$(p+q, q)$, where $p+q$ is the total degree and $q$ is the weight.
\item
The element $\tau$ in $\bM^\R_2$ maps to 
$u$ \cite[Definition 3.12]{HHR} under the
realization map from $\R$-motivic to $C_2$-equivariant
homotopy theory.  We use the symbol $\tau$ in both cases.
\item 
Similarly, realization takes the $\R$-motivic homotopy class
$\rho: S^{-1,-1} \rightarrow S^{0,0}$ to 
$a$ in $\pi_{-1,-1}$ \cite[Definition 3.11]{HHR}.  We use
the symbol $\rho$ for both of these homotopy classes, and
also for the corresponding elements of
$\bM^\R_2$ and $\bM^{C_2}_2$.
\end{enumerate}

We grade $\Ext$ groups in the form $(s,f,w)$, where $s$ is the stem,
i.e., the total degree minus the homological degree; $f$ is the Adams
 filtration, i.e., the homological degree; and $w$ is the weight.
We will also refer to the Milnor-Witt degree, which equals $s-w$.

\section{Ext groups}
\label{sctn:Ext}

\subsection{The equivariant cohomology of a point}
\label{subsctn:cohomology-point}

The purpose of this section is to carefully describe the structure
of the equivariant cohomology ring 
$\bM^{C_2}_2$
of a point from a perspective
that will be useful for our calculations.  This section is a
reinterpretation of results from \cite{HK}*{Proposition~6.2}.

Additively, $\bM^{C_2}_2$ equals
\begin{enumerate}
\item
$\F_2$ in degree $(s,w)$ if $s \geq 0$ and $w \geq s$.
\item
$\F_2$ in degree $(s,w)$ if $s \leq 0$ and $w \leq s - 2$.
\item
$0$ otherwise.
\end{enumerate}
This additive structure is represented by the dots in 
Figure \ref{fig:M2}.  
The non-zero element in degree $(0,1)$ is called $\tau$,
and the non-zero element in degree $(1,1)$ is called $\rho$.
We remind the reader that we are here employing 
cohomological grading. Thus the class $\rho$ has degree $(-1,-1)$
when considered as an element of the homology ring 
$\pi_{*,*}H\underline{\F}_2$.

\psset{linewidth=0.25mm}    
\psset{unit=5mm}
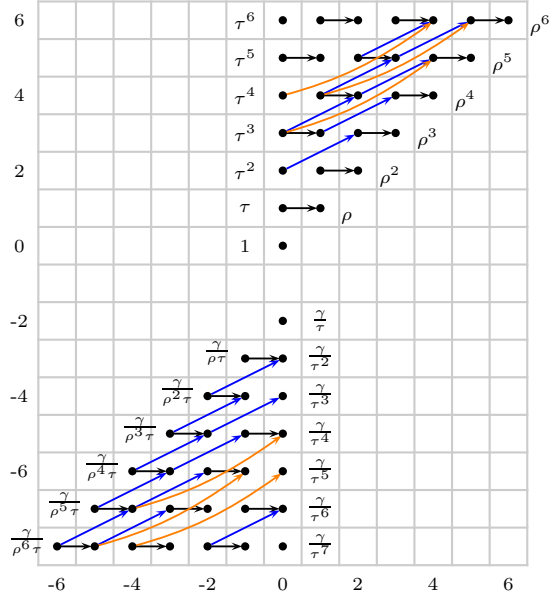
\begin{figure}[h]
\caption{
\label{fig:M2}
$\bM_2^{C_2}$, with action by $\Sq^1$, $\Sq^2$, and $\Sq^4$}
\begin{center}
\begin{pspicture}(-8,-9)(8,8)

\psgrid[unit=1,gridcolor=gridline,subgriddiv=0,gridlabelcolor=white](-0.5,-0.5)(-5.5,-7.5)(7,7)

\scriptsize 
\rput(0,1){\SqOneLine}
\rput(1,2){\SqOneLine}
\rput(0,3){\SqOneLine}
\rput(2,3){\SqOneLine}
\rput(1,4){\SqOneLine}
\rput(3,4){\SqOneLine}
\rput(0,5){\SqOneLine}
\rput(2,5){\SqOneLine}
\rput(4,5){\SqOneLine}
\rput(1,6){\SqOneLine}
\rput(3,6){\SqOneLine}
\rput(5,6){\SqOneLine}
\rput(0,2){\SqTwoLine}
\rput(0,3){\SqTwoLine}
\rput(1,3){\SqTwoLine}
\rput(1,4){\SqTwoLine}
\rput(2,4){\SqTwoLine}
\rput(2,5){\SqTwoLine}
\rput(3,5){\SqTwoLine}
\rput(0,3){\SqFourLine}
\rput(0,4){\SqFourLine}
\rput(1,4){\SqFourLine}

\rput(-1,-3){\SqOneLine}
\rput(-2,-4){\SqOneLine}
\rput(-3,-5){\SqOneLine}
\rput(-1,-5){\SqOneLine}
\rput(-4,-6){\SqOneLine}
\rput(-2,-6){\SqOneLine}
\rput(-5,-7){\SqOneLine}
\rput(-3,-7){\SqOneLine}
\rput(-1,-7){\SqOneLine}
\rput(-6,-8){\SqOneLine}
\rput(-4,-8){\SqOneLine}
\rput(-2,-8){\SqOneLine}
\rput(-2,-4){\SqTwoLine}
\rput(-3,-5){\SqTwoLine}
\rput(-2,-5){\SqTwoLine}
\rput(-4,-6){\SqTwoLine}
\rput(-3,-6){\SqTwoLine}
\rput(-5,-7){\SqTwoLine}
\rput(-4,-7){\SqTwoLine}
\rput(-6,-8){\SqTwoLine}
\rput(-5,-8){\SqTwoLine}
\rput(-2,-8){\SqTwoLine}
\rput(-4,-7){\SqFourLine}
\rput(-5,-8){\SqFourLine}
\rput(-4,-8){\SqFourLine}

\multido{\ii=0+1}{7}
{\rput(\ii,\ii){\pscircle*(0.5,0.5){\MTwocrad}}}
\multido{\ii=0+1}{6}
{\rput(\ii,\ii){\pscircle*(0.5,1.5){\MTwocrad}}}
\multido{\ii=0+1}{5}
{\rput(\ii,\ii){\pscircle*(0.5,2.5){\MTwocrad}}}
\multido{\ii=0+1}{4}
{\rput(\ii,\ii){\pscircle*(0.5,3.5){\MTwocrad}}}
\multido{\ii=0+1}{3}
{\rput(\ii,\ii){\pscircle*(0.5,4.5){\MTwocrad}}}
\multido{\ii=0+1}{2}
{\rput(\ii,\ii){\pscircle*(0.5,5.5){\MTwocrad}}}
\multido{\ii=0+1}{1}
{\rput(\ii,\ii){\pscircle*(0.5,6.5){\MTwocrad}}}

\multido{\ii=0+1}{7}
{\rput(-\ii,-\ii){\pscircle*(0.5,-1.5){\MTwocrad}}}
\multido{\ii=0+1}{6}
{\rput(-\ii,-\ii){\pscircle*(0.5,-2.5){\MTwocrad}}}
\multido{\ii=0+1}{5}
{\rput(-\ii,-\ii){\pscircle*(0.5,-3.5){\MTwocrad}}}
\multido{\ii=0+1}{4}
{\rput(-\ii,-\ii){\pscircle*(0.5,-4.5){\MTwocrad}}}
\multido{\ii=0+1}{3}
{\rput(-\ii,-\ii){\pscircle*(0.5,-5.5){\MTwocrad}}}
\multido{\ii=0+1}{2}
{\rput(-\ii,-\ii){\pscircle*(0.5,-6.5){\MTwocrad}}}
\multido{\ii=0+1}{1}
{\rput(-\ii,-\ii){\pscircle*(0.5,-7.5){\MTwocrad}}}

\multido{\ii=-6+2}{7}
{\rput(\ii,-8.5){\rput(0.5,0){\ii}}}
\multido{\ii=-6+2}{7}
{\rput(-6.5,\ii){\rput(0,0.5){\ii}}}
\rput(-0.5,0.5){$1$}
\rput(-0.5,1.5){$\tau$}
\rput(-0.5,2.5){$\tau^2$}
\rput(-0.5,3.5){$\tau^3$}
\rput(-0.5,4.5){$\tau^4$}
\rput(-0.5,5.5){$\tau^5$}
\rput(-0.5,6.5){$\tau^6$}
\rput(2.2,1.25){$\rho$}
\rput(3.35,2.35){$\rho^2$}
\rput(4.35,3.35){$\rho^3$}
\rput(5.35,4.35){$\rho^4$}
\rput(6.35,5.35){$\rho^5$}
\rput(7.35,6.35){$\rho^6$}
\multido{\ii=2+1}{6}
{\rput(0,-\ii){\rput(1.5,-0.5){$\frac{\gamma}{\tau^\ii}$}}}
\rput(1.5,-1.5){$\frac{\gamma}{\tau}$}
\multido{\ii=2+1}{5}
{\rput(-\ii,-\ii){\rput(-0.3,-1.4){$\frac{\gamma}{\rho^\ii \tau}$}}}
\rput(-1.25,-2.35){$\frac{\gamma}{\rho \tau}$}

\end{pspicture}
\end{center}
\end{figure}

The ``positive cone" refers to the part of $\bM^{C_2}_2$ in degrees 
$(s,w)$ with $w \geq 0$.  The positive cone is isomorphic
to the $\R$-motivic cohomology ring $\bM^\R_2$ of a point.
Multiplicatively, the positive cone is just a polynomial ring
on two variables, $\rho$ and $\tau$.

The ``negative cone" $NC$ refers to the part of $\bM^{C_2}_2$
in degrees $(s,w)$ with $w \leq -2$.  Multiplicatively,
the product of any two elements of $NC$ is zero,
so $\bM^{C_2}_2$ is a square-zero extension
of $\bM^\R_2$.
Also, multiplications by $\rho$ and $\tau$ are non-zero in
$NC$ whenever they make sense.  Thus, the elements of $NC$
are infinitely divisible by both $\rho$ and $\tau$.

We use the notation
$\frac{\gamma}{\rho^j \tau^{k}}$ for the non-zero element
in degree $(-j,-1-j-k)$.  
This is consistent
with the multiplicative properties described in the previous paragraph.
So $\tau \cdot \frac{\gamma}{\rho^j \tau^k}$ equals 
$\frac{\gamma}{\rho^j \tau^{k-1}}$ when $k \geq 2$,
and $\rho \cdot \frac{\gamma}{\rho^j \tau^k}$ equals
$\frac{\gamma}{\rho^{j-1} \tau^k}$ when $j \geq 2$.

The symbol $\gamma$, which does not correspond to
an actual element of $\bM^{C_2}_2$, has degree $(0,-1)$.

The $\F_2[\tau]$-module structure on $\bM^{C_2}_2$ is essential
for later calculations, since we will filter by powers of $\rho$.
Therefore, we explore further the $\F_2[\tau]$-module structure 
on $NC$.

\begin{rmk}
\label{rmk:theta-notation}
Recall that 
$\F_2[\tau]/\tau^\infty$ is the $\F_2[\tau]$-module
$\colim \F_2[\tau]/\tau^k$, which consists entirely of elements
that are divisible by $\tau$.
We write
$\frac{\F_2[\tau]}{\tau^\infty} \left\{ x \right\}$
for the infinitely divisible $\F_2[\tau]$-module consisting
of elements of the form $\frac{x}{\tau^k}$ for $k \geq 1$.
Note that $x$ itself is not an element of
$\frac{\F_2[\tau]}{\tau^\infty} \left\{ x \right\}$.
The idea is that $x$ represents the 
infinitely many relations $\tau^k \cdot \frac{x}{\tau^k} = 0$
that define
$\frac{\F_2[\tau]}{\tau^\infty} \left\{ x \right\}$.
\end{rmk}

With this notation in place,
$\bM^{C_2}_2$ is equal to 
\begin{equation} 
\bM_2^\R \oplus NC = 
\bM_2^\R \oplus \bigoplus_{s \geq 0} \frac{\F_2[\tau]}{\tau^\infty}
\left\{ \frac{\gamma}{\rho^s}\right\} 
\label{eq:MC2}
\end{equation}
as an $\F_2[\tau]$-module.

\subsection{The equivariant Steenrod algebra}\label{sec:C2Steenrod}

As a Hopf algebroid, the equivariant dual Steenrod algebra can be described \cite{Ricka}*{Proposition~6.10(2)} as
\begin{equation} \cA^{C_2}_* \iso \bM_2^{C_2}\otimes_{\bM_2^{\R}} \cA^\R_*. \label{eq:AC2-split}
\end{equation}
Recall \cite{Voev} that 
\[ \cA^\R_* \iso \bM_2^\R[\tau_0,\tau_1,\dots,\xi_1,\xi_2,\dots]/(\tau_i^2 = \rho \tau_{i+1} + \tau \xi_{i+1} + \rho\tau_0\xi_{i+1}),\]
with $\eta_R(\rho)=\rho$ and $\eta_R(\tau) = \tau + \rho \tau_0$.
The formula for the right unit $\eta_R$ on the negative cone given in 
\cite{HK}*{Theorem~6.41} appears in our notation as
\begin{equation} 
\eta_R\left( \frac{\gamma}{\rho^j \tau^k }\right) = 
\frac{\gamma}{\rho^j \tau^k } 
\left[ \sum_{i\geq 0} \left(\frac{\rho}{\tau}\tau_0 \right)^i \right]^k.
\label{eq:etaR}
\end{equation}
Note that the sum is finite because $ \frac{\gamma}{\rho^j \tau^k } \cdot \rho^n=0 $ if $n > j$.

We have quotient Hopf algebroids
\[ \cA^\R_*(n) := 
\bM_2^\R[\tau_0,\dots, \tau_n,\xi_1,\dots,\xi_n]/(\xi_i^{2^{n-i+1}},
\tau_i^2 = \rho \tau_{i+1} + \tau \xi_{i+1} + \rho\tau_0\xi_{i+1}).
\]
and
\[ \cE^\R_*(n) :=
\bM_2^\R [\tau_0,\dots,\tau_n]/( \tau_i^2 = \rho \tau_{i+1} , \tau_n^2)\]
and their equivariant analogues 
\begin{equation}
\cA^{C_2}_*(n) := \bM^{C_2} \otimes_{\bM^\R} \cA^\R_*(n), \qquad \cE^{C_2}_*(n) := \bM^{C_2}\otimes_{\bM^R}E^\R_*(n)
\label{eq:AnC2-split}
\end{equation}
Their duals are the subalgebras $\cA^{C_2}(n)\subseteq \cA^{C_2}$ and $\cE^{C_2}(n)\subseteq \cA^{C_2}$.

The relationship between the equivariant and $\R$-motivic
Steenrod algebras leads to an analogous relationship between
$\Ext$ groups.

\begin{prop}\label{ExtSplits} 
Suppose that $\Gamma$ is a Hopf algebroid over $A$ and that $B\iso A \oplus M$ is a $\Gamma$-comodule which is a square-zero extension of $A$,
meaning that the product of any two elements in $M$ is zero. Then the $A$-module splitting of $B$ induces a splitting
\[ \Ext_{B\otimes_A \Gamma}(B,B) \iso \Ext_\Gamma(A,A) \oplus \Ext_\Gamma(M,A)\]
of $\Ext_\Gamma(A,A)$-modules. Furthermore, this is an isomorphism of $\Ext_\Gamma(A,A)$-algebras, if the right-hand side
is taken to be a square-zero extension of $\Ext_\Gamma(A,A)$.
\end{prop}

\begin{pf}
We may express the cobar complex as:

\[\begin{split}
 \coB^s(B,B\otimes_A \Gamma) = B \otimes_{B} (\Gamma)^{\otimes s} &\iso  B \otimes_{B} (B\otimes_{A} \Gamma)^{\otimes s} \\
 &\iso B \otimes_{A} (\Gamma)^{\otimes s}
\end{split}
\]
As the splitting of $B$ is a splitting as $\Gamma$-comodules, there results a splitting 
\[
\coB^s(A,\Gamma) \oplus \coB^s(M,\Gamma)
\]
of the cobar complex.  
This splitting is square-zero, in the sense that the product
of any two elements in the second factor is equal to zero.
This observation follows from the fact that the product of 
any two elements of $M$ is zero.

In $\Ext_{B\otimes_A \Gamma}$, this yields
\[ \Ext_{B\otimes_A \Gamma} \iso \Ext_{\Gamma}(A,A) \oplus \Ext_{\Gamma}(M,A) .\]
The multiplication on $\Ext_{\Gamma}(M,A)$ is zero since this is already
true in the cobar complex $\coB^s(M,\Gamma)$.
\end{pf}

Employing notation given in section~\ref{sec:notn}, Proposition~\ref{ExtSplits} applies to give isomorphisms
\[ \Ext_{C_2} \iso \Ext_{\R} \oplus \Ext_{NC} \]
and
\[ \Ext_{C_2(n)} \iso \Ext_{\R(n)} \oplus \Ext_{NC(n)}.\]

Thus from the point of view of $\R$-motivic homotopy theory, the cohomology of the negative cone is the only new feature in $\Ext_{\cA^{C_2}}$ or $\Ext_{\cA^{C_2}(n)}$.

\section{The $\rho$-Bockstein spectral sequence}
\label{sctn:Bockstein}

Our tool for computing $\R$-motivic or $C_2$-equivariant $\Ext$ is the $\rho$-Bockstein spectral sequence \cites{H, MWLow}. The $\rho$-Bockstein spectral sequence arises by filtering the cobar complex by powers of $\rho$.
More precisely, 
we can define an $\cA^\R$-module filtration on $\bM_2^{C_2}$, where $F_p(\bM_2^{C_2})$ is the part of $\bM_2^{C_2}$
concentrated in degrees $(s,w)$ with $s \geq p$.
Dualizing, we get a filtration of comodules over the dual Steenrod algebra, which induces a filtration on the cobar complex that computes
$\Ext_{C_2}$. 

Recall that the $\C$-motivic cohomology of a point
is $\bM_2^\C = \F_2[\tau]$, and
the $\C$-motivic Steenrod algebra is $\cA^\C = \cA^\R / \rho$
\cite{VoevHF2} \cite{Voev}.  For convenience,
we write
$\Ext_\C$ for $\Ext_{\cA^\C}(\bM^\C_2, \bM^\C_2)$.

\begin{prop} 
\label{prop:Bockstein-E1}
There is a $\rho$-Bockstein spectral sequence 
\[ E_1 = \Ext_{\gr_{\!\rho} \cA^{C_2}}(\gr_{\!\rho} \, \bM_2^{C_2},\gr_{\!\rho} \, \bM_2^{C_2}) \Rightarrow \Ext_{C_2}\]
such that a Bockstein differential $d_r$ takes a class $x$ of degree $(s,f,w)$ to a class $d_r(x)$ of degree $(s-1,f+1,w)$.
Under the splitting of Proposition~\ref{ExtSplits}, this decomposes as
\[ 
E_1^+ = \Ext_{\C}[\rho] \Rightarrow \Ext_{\R}
\]
and
\[
E_1^- \Rightarrow \Ext_{NC},
\]
where $E_1^-$ belongs to a split short exact sequence
\[
\bigoplus_{s\geq 0} \frac{\bM_2^\C}{\tau^\infty} 
\left\{\frac\gamma{\rho^s}\right\} \otimes_{\bM_2^\C} \Ext_{\C}
\rightarrow E_1^- \rightarrow
\bigoplus_{s\geq 0} \Tor_{\bM_2^\C}
\left( \frac{\bM_2^\C}{\tau^\infty} \left\{\frac\gamma{\rho^s}\right\},
\Ext_{\C}\right).
\]
\end{prop}

\begin{rmk}
Beware that the short exact sequence for $E_1^-$
does not split canonically.
\end{rmk}

\begin{rmk}
\label{rmk:E1-Bockstein}
The same spectral sequences occur in the same form
when $\cA^{C_2}$, $\cA^\R$, and $\cA^\C$ are replaced
by $\cA^{C_2}(n)$, $\cA^\R(n)$, and $\cA^\C(n)$.
\end{rmk}

\begin{pf}
See \cite{H}*{Proposition~2.3} 
(or \cite{MWLow}*{Section~3}) for the description
of $E_1^+$.

For $E_1^-$, the associated graded of $NC$ is
\[ \gr_{\!\rho} NC \iso 
\bigoplus_{s\geq 0}  \frac{\bM_2^\C}{\tau^\infty}\left\{\frac\gamma{\rho^s}\right\},\]
as described in Section \ref{subsctn:cohomology-point}.
It follows that
the Bockstein spectral sequence begins with
\[ E_0 \iso 
\bigoplus_{s\geq 0} \frac{\bM_2^\C}{\tau^\infty}\left\{\frac\gamma{\rho^s}\right\}\otimes_{\bM_2^\C} \coB(\bM_2^\C,\cA_*^\C).\]

The ring $\bM_2^\C\iso \F_2[\tau]$ is a graded principal ideal domain
(in fact, it is a graded local ring with maximal ideal generated by
$\tau$).  Therefore, the Kunneth split exact sequence gives
\[ 
\left(\bigoplus_{s\geq 0} \frac{\bM_2^\C}{\tau^\infty}\left\{\frac\gamma{\rho^s}\right\}\right) \otimes_{\bM_2^\C} \Ext_{\C}
\rightarrow E_1^- \rightarrow
 \Tor_{\bM_2^\C}\left( \bigoplus_{s\geq 0}\frac{\bM_2^\C}{\tau^\infty}\left\{\frac\gamma{\rho^s}\right\},\Ext_{\C}\right).
\]
The first and third terms of the short exact sequence
may be rewritten as in the statement of the proposition
because the direct sum in each case is a splitting of $\bM_2^\C$-modules.
\end{pf}

We shall completely analyze the spectral
sequence
\[ 
E_1^+ = \Ext_{\C}(1)[\rho] \Rightarrow \Ext_{\R}(1)
\]
in Section \ref{sec:ExtAR1}.  While non-trivial, this part
of our calculation is comparatively straightforward.

On the other hand, analysis of the spectral
sequence $E_1^- \Rightarrow \Ext_{NC}(1)$
requires significantly more work.  
The first step is to compute $E_1^-$ more explicitly.
In particular, we must
describe the $\Tor$ groups that arise.

\begin{lemma} \mbox{}
\label{lem:compute-Tor-combined}
\begin{enumerate}
\item
$\Tor_{\bM_2^\C}^* \left( \frac{\bM_2^\C}{\tau^\infty}, 
\bM_2^\C \right)$ equals
$\frac{\bM_2^\C}{\tau^\infty}$, concentrated in homological degree zero.
\item 
$\Tor_{\bM_2^\C}^* \left( \frac{\bM_2^\C}{\tau^\infty}, 
\frac{\bM_2^\C}{\tau^k} \right)$ equals
$\frac{\bM_2^\C}{\tau^k}$, concentrated in homological degree one.
\end{enumerate}
\end{lemma}

\begin{proof} \mbox{}
\begin{enumerate}
\item
This is a standard fact about the vanishing of higher $\Tor$ groups
for free modules.
\item
This follows from direct computation, using the resolution
\[
\xymatrix@1{
\frac{\bM_2^\C}{\tau^k} & \ar[l] \bM_2^\C & \ar[l]_{\tau^k} \bM_2^\C
& \ar[l] 0.
}
\]
After tensoring with $\frac{\bM_2^\C}{\tau^\infty}$, this gives
the map
\[
\xymatrix@1{ 
\frac{\bM_2^\C}{\tau^\infty} \left\{ x \right\} & \ar[l]
\frac{\bM_2^\C}{\tau^\infty} \left\{ y \right\} }
\]
that takes $\frac{y}{\tau^a}$ to $\frac{x}{\tau^{a-k}}$
if $a > k$, and takes $\frac{y}{\tau^a}$ to zero if
$a \leq k$.
This map is onto, and its kernel is isomorphic
to $\bM_2^\C / \tau^k$.
\end{enumerate}
\end{proof}

\begin{rmk}
\label{rmk:Tor-practical}
Lemma \ref{lem:compute-Tor-combined} 
gives a practical method for identifying $E_1^-$
in Proposition \ref{prop:Bockstein-E1}.
Copies of $\bM_2^\C$ in $\Ext_{\C}(1)$ lead to copies of the negative cone in $E_1^-$. On the other hand, 
copies of $\bM_2^\C/\tau$, such as the submodule generated by $h_1^3$,
lead to copies of
$\bM_2^\C/\tau$ in $E_1^-$ that are infinitely 
divisible by $\rho$.  These copies of 
$\bM_2^\C/\tau$ occur with a degree shift because they arise
from $\Tor^1$.
\end{rmk}

\section{$\rho$-inverted $\Ext_{\R}(1)$}
\label{sctn:rho-inverted}

As a first step towards computing $\Ext_{C_2}(1)$,
we will invert $\rho$ in the $\R$-motivic setting and study
$\Ext_{\R}(1)[\rho^{-1}]$.
This gives partial information about $\Ext_{\R}(1)$
and also about $\Ext_{C_2}(1)$.
Afterwards, it remains to compute $\rho^k$ torsion, including
infinitely $\rho$-divisible elements.

We write $\cA^\cla$ for the classical Steenrod algebra.  For
convenience, we write
$\Ext_\cla$ and $\Ext_\cla(n)$ for
$\Ext_{\cA^{\cla}}(\F_2, \F_2)$ and
$\Ext_{\cA^{\cla}(n)}(\F_2, \F_2)$ respectively.

\begin{prop}\label{prop:RhoLocAn} There is an injection 
$\Ext_{\cla}(n-1)[\rho^{\pm 1}] \into 
\Ext_\R(n)[\rho^{-1}]$ such that:
\begin{enumerate}
\item The map is highly structured, i.e., preserves products, Massey products, and algebraic squaring operations.
\item The element $h_i$ of $\Ext_{\cla}(n-1)$ corresponds to 
$h_{i+1}$ of $\Ext_\R(n)$.
\item The map induces an isomorphism
\[ 
\Ext_{\R}(n)[\rho^{-1}] \iso 
\Ext_{\cla}(n-1)[\rho^{\pm 1}] \otimes \F_2 [ \tau^{2^{n+1}} ].
\]
\item An element in $\Ext_{\cla}(n-1)$ of degree $(s,f)$ corresponds to an element in $\Ext_\R(n)$ of degree $(2s+f,f,s+f)$.
\end{enumerate}
\end{prop}

\begin{pf}
The proof is similar to the proof 
 of \cite[Theorem~4.1]{MWLow}. Since localization is an exact functor, we may compute the cohomology of the Hopf algebroid 
$(\bM^\R_2[\rho^{-1}], \cA^\R(n+1)_*[\rho^{-1}])$ to obtain 
$\Ext_\R(n+1)[\rho^{-1}]$. 
After inverting $\rho$, we have 
\[ \tau_{k+1} = \rho^{-1} \tau_k^2 + \rho^{-1}\tau\xi_{k+1} + \tau_0 \xi_{k+1}, 
\]
and it follows that 
\[ \cA^\R(n)_*[\rho^{-1}] \iso \bM_2^\R[\rho^{-1}][\tau_0,\xi_1,\dots,\xi_n]/(\tau_0^{2^{n+1}}, \xi_1^{2^n},\dots,\xi_n^2).\]
This splits as
\[ ( \bM_2^\R[\rho^{-1}],\cA(n)_*[\rho^{-1}]) \iso \big(\bM_2^\R[\rho^{-1}],\cA'(n)\big) \otimes_{\F_2} (\F_2,\cA''(n)),\]
where
\[ \cA'(n) = \bM_2^\R[\rho^{-1}][\tau_0]/\tau_0^{2^{n+1}}\]
and
\[\cA''(n) = \F_2[\xi_1,\dots,\xi_n]/(\xi_1^{2^n},\dots,\xi_n^2).\]
The Hopf algebra  $(\F_2,\cA''(n))$ is isomorphic to the classical Hopf algebra \mbox{$(\F_2,\cA(n-1))$} with altered degrees, 
so its cohomology is $\Ext_{\cla}(n-1)$.

For the Hopf algebroid $\big(\bM_2^\R[\rho^{-1}],\cA'(n)\big) $, we have 
an isomorphism
\[ \big( \bM_2^\R[\rho^{-1}],\cA'(n)\big) \iso \F_2[\rho^{\pm 1}] \otimes_{\F_2} (\F_2[\tau],\F_2[\tau][x]/x^{2^{n+1}})
\]
with 
\[ \eta_L(\tau) = \tau, \qquad \eta_R(\tau) = \tau+x.\]
An argument like that of \cite{MWLow}*{Lemma~4.2} shows that the cohomology of this Hopf algebroid is $\F_2[\tau^{2^{n+1}}]$.
\end{pf}

\begin{cor}\label{ExtA1RRho}
\[ \Ext_{C_2}(1)[\rho^{-1}] \iso  \Ext_\R(1)[\rho^{-1}] \iso \F_2[\rho^{\pm 1}, \tau^4, h_1].\]
\end{cor}

\begin{pf}
The first isomorphism follows from Proposition~\ref{ExtSplits}, as $\Ext_{NC}$ is $\rho$-torsion.
The second isomorphism follows immediately from Proposition~\ref{prop:RhoLocAn}, given that $\Ext_{\cla}(0) \iso \F_2[h_0]$.
\end{pf}

\begin{rmk}
Corollary \ref{ExtA1RRho} implies that the products
$\tau^4 \cdot h_1^k$ are non-zero in
$\Ext_{\R}(1)$.  But
$\tau^4 h_1^k = 0$ in $\Ext_\C(1)$ when $k \geq 3$,
so the products
$\tau^4 \cdot h_1^k$ are hidden in the $\rho$-Bockstein
spectral sequence for $k \geq 3$.
We will sort this out in detail in Section \ref{sec:ExtAR1}.
\end{rmk}

\section{Infinitely $\rho$-divisible elements of $\Ext_{\cA^{C_2}(1)}$}\label{sec:RhoComp}

Having computed the effect of inverting $\rho$ 
in Section \ref{sctn:rho-inverted}, we now consider the dual
question and study infinitely $\rho$-divisible elements.
This gives additional partial information about
$\Ext_{C_2}(1)$.  Afterwards, it remains only to
compute the $\rho^k$ torsion classes that are not infinitely
$\rho$-divisible.

In fact, this section is not strictly necessary to carry out
the computation of $\Ext_{C_2}(1)$.  Nevertheless,
the infinitely $\rho$-divisible calculation works out rather nicely
and provides some useful insight into the main computation.
We also anticipate that 
this approach via infinitely $\rho$-divisible classes will
be essential in the much more complicated calculation
of $\Ext_{C_2}$, to be studied in further work.

For a $\F_2[\rho]$-module $M$, the $\rho$-colocalization, or $\rho$-cellularization, is
the limit $\lim_{\rho} M$ of the inverse system
\[ \cdots \xrtarr{\rho} M \xrtarr{\rho} M.\]
While $\rho$-localization detects $\rho$-torsion-free elements, the $\rho$-colocalization detects infinitely $\rho$-divisible elements.

An alternative description is given by the isomorphism
\[
\lim_{\rho} M \iso \Hom_{\F_2[\rho]}(\F_2[\rho^{\pm 1}],M)
\]
because $\F_2[\rho^{\pm 1}]$ is isomorphic to $\colim_\rho \F_2[\rho]$.
It follows that $\lim_\rho M$ is an
$\F_2[\rho^{\pm 1}]$-module, and
the functor $M \mapsto \lim_\rho M$ is right adjoint to the restriction 
\[\mathrm{Mod}_{\F_2[\rho^{\pm1}]}\rtarr \mathrm{Mod}_{\F_2[\rho]}.
\]

\begin{lemma}\label{RhoColoc} 
\mbox{}
\begin{enumerate}
\item
Let $M$ be a cyclic $\F_2[\rho]$-module $\F_2[\rho]$ or 
$\F_2[\rho]/\rho^k$.
Then $\lim_\rho M$ is zero.
\item
Let $M$ be the infinitely divisible $\F_2[\rho]$-module
$\F_2[\rho]/\rho^\infty$.  Then
$\lim_\rho M$ is isomorphic to $\F_2[\rho^{\pm 1}]$.
\end{enumerate}
\end{lemma}

\begin{pf}
If $M$ is cyclic, then no nonzero element is infinitely $\rho$-divisible, which implies the first statement.  For the case $M=\F_2[\rho]/\rho^\infty$, a (homogeneous) element of the limit is either of the form
\[ \left(\frac{1}{\rho^k},\frac1{\rho^{k+1}},\ldots\right)\]
or of the form
\[ \left(0,\dots,0,1,\frac{1}{\rho},\frac1{\rho^{2}},\ldots\right).\]
For $k \geq 0$,
the isomorphism $\F_2[\rho^{\pm 1}] \rightarrow \lim_\rho M$ 
sends $\rho^k$ to the tuple $(0,\ldots,0,1,\frac1\rho,\ldots)$ having $k-1$ zeroes and sends $\frac1{\rho^k}$ to $(\frac1{\rho^k},\frac1{\rho^{k+1}},\ldots)$.
\end{pf}

We will now compute the $\rho$-colocalization
of $\Ext_{{C_2}(1)}$.

\begin{prop} 
\label{prop:rho-divisible-Ext}
\[
\lim_{\rho} \Ext_{C_2}(1) \iso 
\bigoplus_{k \geq 1}
\F_2[\rho^{\pm 1},h_1]\left\{ \frac{\gamma}{\tau^{4k}}\right\}
\iso
\F_2[\rho^{\pm 1}, h_1] \otimes
\frac{\F_2[\tau^4]}{\tau^\infty}
\left\{ \gamma \right\}.
\]
\end{prop}

Recall that $\gamma$ itself is not an element of
$\lim_{\rho} \Ext_{C_2}(1)$, as described in
Remark \ref{rmk:theta-notation}.
The main point of Proposition \ref{prop:rho-divisible-Ext}
is that the elements $\frac{\gamma}{\tau^{4k}} h_1^j$
are infinitely $\rho$-divisible classes in
$\Ext_{C_2}(1)$, and there are no other infinitely
$\rho$-divisible families in 
$\Ext_{C_2}(1)$.

\begin{pf}
Since the cobar complex $\coB^*(\bM_2^{C_2},A^{C_2}(1))$ is finite-dimensional in each tridegree, the inverse systems 
\[ \cdots \xrtarr{\rho} \coB^*(\bM_2^{C_2},A^{C_2}(1))\xrtarr{\rho} \coB^*(\bM_2^{C_2},A^{C_2}(1))\]
and
\[
 \cdots \xrtarr{\rho} \Ext_{C_2}(1) \xrtarr{\rho} \Ext_{C_2}(1)
\]
satisfy the Mittag-Leffler condition, so that (\cite{Weib}*{Theorem~3.5.8})

\[ \lim_\rho 
\Ext_{C_2}(1) \iso \rH^*\Big( \lim_\rho \coB^*(\bM_2^{C_2},A^{C_2}(1)) \Big).\]

Now we compute
\[\begin{split}
 \lim_\rho \coB^s(\bM_2^{C_2},A^{C_2}(1)) &= 
 \lim_{\rho}\left( \bM_2^{C_2} \otimes_{\bM_2^{C_2}} A^{C_2}(1)^{\otimes s}\right) \\
 &\iso \lim_\rho \left( \bM_2^{C_2}\otimes_{\bM_2^\R} A^\R(1)^{\otimes s}\right).
\end{split}
\]
The splitting $\bM_2^{C_2} = \bM_2^\R \oplus NC$
yields a splitting
\[
\left(
\bM_2^\R \otimes_{\bM_2^\R} \cA^\R(1)^{\otimes s} \right)
\oplus 
\left( NC \otimes_{\bM_2^\R} \cA^\R(1)^{\otimes s} \right)
\]
of $\bM_2^{C_2}\otimes_{\bM_2^\R} \cA^\R(1)^{\otimes s}$ 
as an $\F_2[\rho]$-module.
The first piece of the splitting contributes
nothing to the $\rho$-colocalization by
Lemma \ref{RhoColoc}(1) because $\bM_2^\R$ is free
as an $\F_2[\rho]$-module.

On the other hand,
the $\F_2[\rho]$-module $NC$ is a direct sum of copies
of $\F_2[\rho]/\rho^\infty$.
By Lemma \ref{RhoColoc}(2), we have that 
$\lim_\rho \left( NC \otimes_{\bM_2^\R} \cA^\R(1)^{\otimes s} \right)$
is isomorphic to
\[
\left( 
\frac{\bM^\R_2[\rho^{-1}]}{\tau^\infty}
 \left\{ \gamma \right\}
 \right) \otimes_{\bM_2^\R} \cA^\R(1)^{\otimes s}.
 \]

Now the argument of Proposition~\ref{prop:RhoLocAn} provides a splitting
\[
\begin{split} 
\coB^*_{\bM_2^\R} \left(
\frac{\bM_2^\R[\rho^{-1}]}{\tau^\infty} \left\{ \gamma \right\}, 
A^\R(1) \right) &\simeq \\
& \hspace{-7em} \coB^*_{\F_2[\tau]}
\left( \frac{\F_2[\tau]}{\tau^\infty} \left\{ \gamma \right\},
\frac{\F_2[\tau,x]}{x^4} \right)
[\rho^{\pm 1}] \otimes_{\F_2} \coB^*_{\F_2}(\F_2,\F_2[\xi_1]/\xi_1^2),\end{split}\]
where $x=\rho\tau_0$. 
The cohomology of the second factor is $\F_2[h_1]$.

It remains to show that the cohomology of
\[
\coB^*_{\F_2[\tau]}
\left( \frac{\F_2[\tau]}{\tau^\infty} \left\{ \gamma \right\},
\frac{\F_2[\tau,x]}{x^4} \right)
\]
is equal to $\frac{\F_2[\tau^4]}{\tau^\infty} \left\{ \gamma \right\}$.
As in Formula (\ref{eq:etaR}),
 the comodule structure on 
 $\frac{\F_2[\tau]}{\tau^\infty} \left\{ \gamma \right\}$ is given by 
\[ \eta_R\left(\frac{\gamma}{\tau^k}\right) = \frac{\gamma}{\tau^k}\left(1+\frac{x}{\tau}+\frac{x^2}{\tau^2}+\frac{x^3}{\tau^3}\right)^k.\]

Now we filter
$\coB^*_{\F_2[\tau]}
\left( \frac{\F_2[\tau]}{\tau^\infty} \left\{ \gamma \right\},
\frac{\F_2[\tau,x]}{x^4} \right)$
by powers of $x$. We then have 
\[ E_1 \iso \frac{\F_2[\tau]}{\tau^\infty} \left\{ \gamma \right\} \otimes_{\F_2} \F_2[v_0,v_1], \]
where $v_0 = [x]$ and $v_1=[x^2]$. 
The differential $d_1\left( \frac{\gamma}{\tau^{2k-1}}\right) = \frac{\gamma}{\tau^{2k}} v_0$ gives
\[ 
E_2 \iso 
\frac{\F_2[\tau^2]}{\tau^\infty} \left\{ \gamma \right\}
\otimes_{\F_2} \F_2[v_1].  
\]
Finally, the differential $d_2\left( \frac{\gamma}{\tau^{4k-2}}\right) = \frac{\gamma}{\tau^{4k}} v_1$ gives
\[ E_3 = E_\infty \iso \frac{\F_2[\tau^4]}{\tau^\infty} 
\left\{ \gamma \right\}.  \]
\end{pf}

\section{The cohomology of $\cA^\R(1)$}
\label{sec:ExtAR1}

Our next step in working towards the calculation
of $\Ext_{C_2}(1)$ is to describe the simpler
$\R$-motivic 
$\Ext_{\R}(1)$.  
The reader is encouraged to consult the charts on pages 
\pageref{E+start}--\pageref{E+end} to follow along with the calculations
described in this section.
This calculation was originally 
carried out in \cite{H}. 
We include the details of the 
$\R$-motivic $\rho$-Bockstein spectral sequence,
but we take the approach of \cite{MWLow}, rather than
\cite{H}, in establishing $\rho$-Bockstein differentials.
The point is that there is only one pattern of differentials
that is consistent with the $\rho$-inverted calculation 
of Corollary \ref{ExtA1RRho}.  This observation
avoids much technical work with Massey products that would
otherwise be required to establish relations that then
imply differentials.

For $\cA^\R(1)$, the $\R$-motivic $\rho$-Bockstein spectral sequence takes the form
\[ \Ext_\C(1) [\rho] \Rightarrow \Ext_\R(1),
\]
where 
\[ 
\Ext_{\C}(1) \iso \bM_2^\C[h_0,h_1,a,b]/
h_0h_1, \tau h_1^3,h_1 a, a^2+h_0^2 b.
\]

\begin{prop}
\label{prop:R-mot-diff}
In the $\R$-motivic $\rho$-Bockstein spectral sequence,
we have differentials
\begin{enumerate}
\item
$d_1(\tau)=\rho h_0$.
\item
$d_2(\tau^2) = \rho^2 \tau h_1$.
\item
$d_3(\tau^3 h_1^2) = \rho^3 a$.
\end{enumerate}
All other differentials on multiplicative generators are zero,
and $E_4$ equals $E_\infty$.
\end{prop}

\begin{proof}
By Corollary~\ref{ExtA1RRho}, the infinite $\rho$-towers that 
survive the $\rho$-Bockstein spectral sequence 
occur in the Milnor-Witt $4k$-stem.  All other
infinite $\rho$-towers are either truncated by a differential
or support a differential.

For example, the permanent cycle $h_0$ must be $\rho$-torsion in 
$\Ext_\R(1)$, which forces the Bockstein differential 
\[ d_1(\tau)=\rho h_0.\]
Next, the $\rho$-tower on $\tau h_1$ cannot survive, and the only possibility is that there is a differential
\[ d_2(\tau^2) = \rho^2 \tau h_1.\]
Note that these differentials also follow easily from the 
right unit formula given in Section~\ref{sec:C2Steenrod}.
The $\rho$-tower on $\tau^3 h_1^2$ cannot survive, and we conclude that it must support a differential
\[ d_3(\tau^3 h_1^2) = \rho^3 a.\]
There is no room for further non-zero differentials, so $E_4=E_\infty$.
\end{proof}

Proposition \ref{prop:R-mot-diff} leads to an explicit description
of the $\R$-motivic $\rho$-Bockstein $E_\infty$-page.
However, there are hidden multiplications 
in passing from $E_\infty$ to $\Ext_\R(1)$.

\begin{thm} 
$\Ext_{\cA^\R(1)}$ is the $\F_2$-algebra on generators given in Table~\ref{tbl:ExtA1Rgen} with relations 
given in Table \ref{tbl:ExtA1Rrel}.
\end{thm}

\begin{table}[h]
\captionof{table}{Generators for $\Ext_\R(1)$
\label{tbl:ExtA1Rgen}}
\begin{center}
\begin{tabular}{LLL} 
\hline
mw & (s,f,w) & \text{generator} \\ \hline
0 & (-1,0,-1) & \rho \\   
0 & (0,1,0) & h_0 \\
0 & (1,1,1) & h_1 \\
1 & (1,1,0) & \tau h_1 \\
2 & (0,1,-2) & \tau^2 h_0 \\
2 & (4,3,2) & a \\
4 & (4,3,0) & \tau^2 a \\
4 & (8,4,4) & b \\
4 & (0,0,-4) & \tau^4 \\
 \hline
\end{tabular}
\end{center}
\end{table}

\begin{table}[h]
\captionof{table}{Relations for $\Ext_\R(1)$
\label{tbl:ExtA1Rrel}}
\begin{center}
\begin{tabular}{LLL} 
\hline
mw & (s,f,w) & \text{relation} \\ \hline
0 & (-1,1,-1) & \rho h_0 \\ 
2 & (-1,1,-3) & \rho\cdot \tau^2 h_0 \\
1 & (-1,1,-2) & \rho^2\cdot \tau h_1 \\
2 & (1,3,-1) & \rho^3 a \\
\hline
4 & (0, 2, -4) & (\tau^2 h_0)^2 + \tau^4 h_0^2 \\
\hline
0 & (1, 2, 1) & h_0 h_1 \\
1 & (1, 2, 0) & h_0 \cdot \tau h_1 + \rho h_1\cdot \tau h_1\\
2 & (1, 2, -1) & \tau^2 h_0 \cdot h_1 + \rho(\tau h_1)^2 \\
3 & (1, 2, -2) & \tau^2 h_0 \cdot \tau h_1 \\
\hline
1 & (3, 3, 2) & h_1^2 \cdot \tau h_1 \\
2 & (3, 3, 1) & h_1(\tau h_1)^2 + \rho a \\
3 & (3, 3, 0) & (\tau h_1)^3 \\
4 & (3, 3, -1) & \tau^4 \cdot h_1^3 + \rho\cdot \tau^2 a \\
\hline
4 & (4, 4, 0) & \tau^2 h_0 \cdot a + h_0\cdot \tau^2 a \\
6 & (4, 4, -2) & \tau^2 h_0 \cdot \tau^2 a + \tau^4 h_0 a \\
\hline
2 & (5, 4, 3) & h_1 a \\
3 & (5, 4, 2) & \tau h_1\cdot a \\
4 & (5, 4, 1) & h_1\cdot \tau^2 a + \rho^3 b \\
5 & (5, 4, 0) & \tau h_1\cdot \tau^2 a \\
\hline
4 & (8, 6, 4) & a^2+ h_0^2 b \\
6 & (8, 6, 2) & a \cdot \tau^2 a + \tau^2 h_0 \cdot h_0 b \\
8 & (8, 6, 0) & (\tau^2 a)^2+ \tau^4 h_0^2 b + \rho^2 \tau^4 h_1^2 b \\
 \hline
\end{tabular}
\end{center}
\end{table}

The horizontal lines in Table \ref{tbl:ExtA1Rrel} group the relations
into families.  The first family describes the $\rho^k$-torsion.
The remaining families are associated to the classical products
$h_0^2$, $h_0 h_1$, $h_1^3$, $h_0 a$, $h_1 a$, and 
$a^2 + h_0^2 b$ respectively.

\begin{proof}
The family of $\rho^k$-torsion relations follows from the
$\rho$-Bockstein differentials of Proposition \ref{prop:R-mot-diff}.

Many relations follow immediately from the
$\rho$-Bockstein $E_\infty$-page because there are no possible
additional terms.  

Corollary~\ref{ExtA1RRho} implies that 
$\tau^4 \cdot h_1^3$, is non-zero in $\Ext_\R(1)$.
It follows that there must be a hidden relation 
\[
\tau^4 \cdot h_1^3 = \rho\cdot \tau^2a.
\]
Similarly, there is a hidden relation
\[
h_1\cdot \tau^2a = \rho^3 b
\]
because $\tau^4 \cdot h_1^4$ is non-zero in $\Ext_\R(1)$.
This last relation then gives rise to the extra term
$\rho^2 \tau^4 h_1^2 b$ in the relation for
$(\tau^2 a)^2 + \tau^4 h_0^2 b$.

Shuffling relations for Massey products imply the remaining three
relations, namely
\[ h_0 \cdot \tau h_1 = h_0 \langle h_1, h_0, \rho\rangle = \langle h_0, h_1, h_0 \rangle \rho  = \rho h_1 \cdot \tau h_1,
\]
\[
\tau^2 h_0 \cdot h_1 = 
\langle \rho \tau h_1, \rho, h_0 \rangle h_1 = 
\rho \tau h_1 \langle \rho, h_0, h_1 \rangle =
\rho (\tau h_1)^2, 
\]
and
\[
\rho a = \rho \langle h_0, h_1, \tau h_1 \cdot h_1 \rangle =
\langle \rho, h_0, h_1 \rangle \tau h_1 \cdot h_1 =
h_1 (\tau h_1)^2.
\]
See Table \ref{tbl:ExtA1Massey} in Section \ref{sctn:Massey}
for more details on these
Massey products, whose indeterminacies are all zero.
\end{proof}

\begin{rmk}\label{ExtE1}
For comparison purposes, we recall from \cite[Theorem~3.1]{H} that
\[ \Ext_{\cE^\R(1)} \iso \F_2[\rho,\tau^4, h_0,\tau^2 h_0, v_1]/(\rho h_0, \rho^3  v_1, (\tau^2 h_0)^2+\tau^4 h_0^2).\]
The $\rho$-torsion is created by the Bockstein differentials $d_1(\tau) = \rho h_0$ and $d_3(\tau^2) = \rho^3 v_1$. The class $v_1$ is in degree $(s,f,w) = (2,1,1)$.
\end{rmk}

\begin{prop}\label{prop:ExtA1toE1}
The ring homomorphism $q_*:\Ext_{\cA^\R(1)} \rtarr \Ext_{\cE^\R(1)}$ induced by the quotient $q:\cA^R(1)_* \rtarr \cE^\R(1)_*$ of Hopf algebroids is given as in Table~\ref{tbl:ExtA1toE1}.
\end{prop}

\begin{table}[h]
\captionof{table}{The homomorphism $\Ext_{\cA^\R(1)} \rtarr \Ext_{\cE^\R(1)}$
\label{tbl:ExtA1toE1}}
\begin{center}
\begin{tabular}{LLLL} 
\hline
mw & (s,f,w) & \text{$x\in\Ext_{\cA^\R(1)}$} & \text{$q_*x\in\Ext_{\cE^\R(1)}$}  \\ \hline
0 & (-1,0,-1) & \rho  & \rho\\   
0 & (0,1,0) & h_0  & h_0\\
0 & (1,1,1) & h_1 & 0\\
1 & (1,1,0) & \tau h_1 & \rho v_1 \\
2 & (0,1,-2) & \tau^2 h_0 & \tau^2 h_0 \\
2 & (4,3,2) & a  & h_0 v_1^2 \\
4 & (4,3,0) & \tau^2 a  & \tau^2 h_0 v_1^2\\
4 & (8,4,4) & b & v_1^4 \\
4 & (0,0,-4) & \tau^4 & \tau^4 \\
 \hline
\end{tabular}
\end{center}
\end{table}

\begin{pf}
Many of the values $q_*(x)$ are already true over $\C$ and follow easily from their descriptions in the May spectral sequence. For instance, $b$ is represented by $h_{2,1}^4$, and $v_1$ is represented by $h_{2,1}$.
On the other hand, the value $q_*(\tau h_1)$ is most easily seen using the cobar complex. The class
$\tau h_1$ in $\Ext_{\cA^\R(1)}$ is represented by $\tau \xi_1 + \rho \tau_1$. This maps to $\rho \tau_1$ in the cobar complex for $\cE^\R(1)$ and represents the class $\rho v_1$ there.
\end{pf}

\section{Bockstein differentials in the negative cone}
\label{sec:BockNegCone}

We finally come to the key step in our calculation 
of $\Ext_{C_2}(1)$.  We are now ready to analyze
the $\rho$-Bockstein differentials associated to the negative
cone, i.e., to the spectral sequence $E^-$ of 
Proposition \ref{prop:Bockstein-E1}.
We already analyzed the spectral sequence $E^+$
in Section \ref{sec:ExtAR1}.

\subsection{The structure of $E_1^-$}

First, we need some additional 
information about the algebraic structure of $E_1^-$.
Since $E_1 = E_1^+ \oplus E_1^-$ is defined in terms of $\Ext$ groups,
it is a ring and has higher structure in the form of Massey products.
The subobject $E_1^-$ is a module 
over $E_1^+$, and it possesses Massey products of the form
$\langle x_1, \ldots, x_n, y \rangle$,
where $x_1, \ldots, x_n$ belong to $E_1^+$ and
$y$ belongs to $E_1^-$.

\begin{defn}
Suppose that $x$ is a non-zero element of $\Ext_{\C}(1)$ 
such that $\tau x$ is zero.
According to Remark \ref{rmk:Tor-practical},
for each $s \geq 0$, the element
$x$ gives rise to a copy of 
$\bM_2^\C/\tau$ in 
$\Tor_{\bM_2^\C}\left(\frac{\bM_2^\C}{\tau^\infty},\Ext_{\C}(1) \right)
\left\{ \frac{\gamma}{\rho^s} \right\}$
that is infinitely divisible by $\rho$.
In particular, it gives a non-zero element of the 
$\Tor$ group.
Let $\frac{Q}{\rho^s} x$ be any lift to $E_1^-$ of this non-zero element.
\end{defn}

\begin{rmk}
There is indeterminacy in the choice of ${Q} x$ which arises
from the first term of the short exact sequence 
for $E_1^-$ in Proposition \ref{prop:Bockstein-E1}.
\end{rmk}

\begin{lemma}
\label{lem:Qx}
The element ${Q} x$ of $E_1^-$ is contained in the Massey product
$\left\langle x, \tau, \frac{\gamma}{\tau} \right\rangle$.
\end{lemma}

\begin{proof}
If $d(u) = \tau\cdot x$ in the cobar complex for $\Ext_{\C}(1)$,
then $\frac{\gamma}{\tau} u$ is a cycle, since 
$\tau \frac{\gamma}{\tau} = 0$.
This cycle $\frac{\gamma}{\tau} u$ represents both
the Massey product as well as ${Q} x$.
\end{proof}

\begin{rmk}
\label{rmk:Qh1^3}
The most important example is the element
${Q} h_1^3$, which is defined because
$\tau h_1^3$ equals zero in $\Ext_\C(1)$.
Another possible name for ${Q} h_1^3$ is
$\frac{\gamma}{\tau} v_1^2$, since $v_1^2$ is the element of the
May spectral sequence that creates the relation $\tau h_1^3$.
\end{rmk}

\begin{rmk}
\label{rmk:Qx}
Beware that the Massey product description for ${Q} x$ holds
in $E_1^-$, not in $\Ext_{C_2}(1)$.  In fact, we have already seen
in Section \ref{sec:ExtAR1} that $\tau$ is not a permanent cycle in the
$\rho$-Bockstein spectral sequence.

Nevertheless, minor variations on these Massey products may exist
in $\Ext_{C_2}(1)$.  For example,
$\langle h_1^2, \tau h_1, \frac{\gamma}{\tau} \rangle$
equals ${Q} h_1^3$ in $\Ext_{C_2}(1)$.
\end{rmk}

We can now deduce a specific computational property of
$E_1^-$ that we will need later.

\begin{lemma}
\label{lem:h0-Qh1^3}
In $E_1^-$, there is a relation
$h_0 \cdot {Q} h_1^3 = \frac{\gamma}{\tau} a$.
\end{lemma}

\begin{proof}
Use Lemma \ref{lem:Qx} and the Massey product shuffle
\[ 
h_0 \cdot {Q} h_1^3 = 
h_0 \left\langle h_1^3, \tau, \frac\gamma{\tau} \right\rangle = 
\langle h_0, h_1^3, \tau \rangle \frac{\gamma}{\tau} = 
\frac{\gamma}{\tau} a.
\]
\end{proof}

Table \ref{tbl:MultE1page} gives
multiplicative generators for the Bockstein
$E_1$-page.  
The elements above the horizontal line are multiplicative generators for
$E_1^+$.  The elements below the horizontal generate
$E_1^-$ in the following sense.  Every element of $E_1^-$
can be formed by starting with one of the these listed elements, multiplying
by elements of $E_1^+$, and then dividing by $\rho$.
The elements in Table \ref{tbl:MultEinfpage} are not multiplicative
generators for $\Ext_{C_2}(1)$ in the usual sense, because we allow
for division by $\rho$.  
The point of this notational approach is that the elements of
$E_1^-$ and of $\Ext_{NC}$ are most easily understood as families
of $\rho$-divisible elements.

\begin{table}[h]
\captionof{table}{Generators for the Bockstein $E_1$-page}
\label{tbl:MultE1page}
\begin{center}
\begin{tabular}{LLL} 
\hline
mw & (s,f,w) &  \text{element} \\ \hline
0 & (-1,0,-1) & \rho \\
0 & (0,1,0) & h_0  \\
0 & (1,1,1) & h_1  \\
1 & (0,0,-1) & \tau  \\
2 & (4,3,2) & a  \\
4 & (8,4,4) & b  \\
\hline
0 & (4,2,4) & {Q} h_1^3 \\
-k-1 & (0,0,k+1) & \frac{\gamma}{\tau^k} \\
\hline
\end{tabular}
\end{center}
\end{table}

\subsection{$\rho$-Bockstein differentials in $E^-$}
\label{subsctn:Bock-diff-E^-}

Our next goal is to analyze the
$\rho$-Bockstein differentials in $E^-$.
We will rely heavily on the $\rho$-Bockstein differentials
for $E^+$ established in Section \ref{sec:ExtAR1}, using that
$E^-$ is an $E^+$-module.

As an $E_1^+$-module,
$E_1^-$ is generated by the elements
$\frac{\gamma}{\rho^j \tau^k}$ and
$\frac{Q}{\rho^j} h_1^3$.
This arises from the observation that the
$\tau$ torsion in $\Ext_{\C}(1)$ is generated
as an $\Ext_{\C}(1)$-module by $h_1^3$.

Proposition \ref{d1ThetaRho} gives the values of the 
$\rho$-Bockstein $d_1$ differential on these generators
of $E_1^-$.  All other $d_1$ differentials can then be deduced
from the Leibniz rule and the $E_1^+$-module structure.

All of the differentials in $E^-$ are infinitely divisible by $\rho$,
in the following sense.
When we claim that $d_r(x) = y$, we also have differentials
$d_r \left( \frac{x}{\rho^j} \right) = \frac{y}{\rho^j}$
for all $j \geq 0$.
For example,
in Proposition \ref{d1ThetaRho},
the formula
$d_1\left(\frac{\gamma}{\rho \tau} \right) = 
\frac{\gamma}{\tau^{2}}h_0$ implies that
$d_1\left(\frac{\gamma}{\rho^{j+1} \tau} \right) = 
\frac{\gamma}{\rho^j \tau^{2}}h_0$ for all $j \geq 0$.

\begin{prop}\label{d1ThetaRho}
For all $k \geq 0$,
\begin{enumerate}
\item
$d_1\left(\frac{\gamma}{\rho \tau^{2k+1}} \right) = 
\frac{\gamma}{\tau^{2k+2}}h_0$.
\item
$d_1\left( \frac{Q}{\rho}h_1^3\right) = 
\frac{\gamma}{\tau^2}a$.
\end{enumerate}
These differentials are infinitely divisible by $\rho$.
\end{prop}

\begin{pf}
We give three proofs for the first formula.
First, it follows from
$\Sq^1 \left( \frac{\gamma}{\rho \tau^{2k+1}} \right) = 
\frac{\gamma}{\tau^{2k+2}}$,
using the relationship between $d_1$ and the left and right units
of the Hopf algebroid.
Second, we have 
\[\begin{split}
 0 =d_1\left( \tau^{2k+1} \frac{\gamma}{\rho \tau^{2k+1}} \right) 
 &= \tau^{2k+1} d_1\left( \frac{\gamma}{\rho \tau^{2k+1}} \right)
   + \frac{\gamma}{\rho \tau^{2k+1}} \rho \tau^{2k} h_0 \\
 & = \tau^{2k+1} d_1\left( \frac{\gamma}{\rho \tau^{2k+1}}\right)
  + \frac{\gamma}{\tau} h_0. 
\end{split}\] 
Third, we can use Proposition \ref{prop:rho-divisible-Ext}
to conclude that 
the infinitely $\rho$-divisible elements
$\frac{\gamma}{ \tau^{2k+1}}$ cannot survive the 
$\rho$-Bockstein spectral sequence. The only possibility
is that they support a $d_1$ differential.

For the second formula, use the first formula to determine that
$d_1 \left( \frac{\gamma}{\rho \tau} a \right) = 
\frac{\gamma}{\tau^2}h_0 a$.
Then use the relation of Lemma \ref{lem:h0-Qh1^3}.
Alternatively, this differential is also forced by
Proposition \ref{prop:rho-divisible-Ext}.
\end{pf}

It is now straightforward to compute $E_2^-$, since the 
$\rho$-Bockstein $d_1$ differential is completely known.
The charts in Section \ref{sctn:chart}
depict $E_2^-$ graphically.

Next, Proposition \ref{d2ThetaRho2Tau} gives a
$\rho$-Bockstein $d_2$ differential in $E_2^-$.
This is the essential calculation, in the sense that
the $d_2$ differential is zero on all other $E_2^+$-module generators
of $E_2^-$.

\begin{prop}\label{d2ThetaRho2Tau}
$d_2\left(\frac{\gamma}{\rho^2 \tau^{4k+2}}\right) = 
\frac{\gamma}{\tau^{4k+3}} h_1$
 for all $k \geq 0$.
This differential is infinitely divisible by $\rho$.
\end{prop}

\begin{pf}
As for Proposition \ref{d1ThetaRho}, we give three proofs.
First, $\Sq^2\left(\frac{\gamma}{\rho^2 \tau^{4k+2}}\right) = 
\frac{\gamma}{\tau^{4k+3}}$. Second, we have
\[
\begin{split}
0 =
d_2\left( \tau^{4k+2} \frac{\gamma}{\rho^2 \tau^{4k+2}} \right) &=
\tau^{4k+2} d_2\left( \frac{\gamma}{\rho^2 \tau^{4k+2}}\right) +
\rho^2 \tau^{4k+1} \frac{\gamma}{\rho^2 \tau^{4k+2}} h_1 \\
& = \tau^{4k+2} d_2\left( \frac{\gamma}{\rho^2 \tau^{4k+2}}\right) + \frac{\gamma}{\tau} h_1.
\end{split} \]

Third, use Proposition \ref{prop:rho-divisible-Ext}
to conclude that 
the infinitely $\rho$-divisible elements
$\frac{\gamma}{ \tau^{4k+1}}$ cannot survive the 
$\rho$-Bockstein spectral sequence. The only possibility
is that they support a $d_2$ differential.
\end{pf}

At this point, the behavior of $E^-$ becomes qualitatively different
from $E^+$.  For $E^+$, there are non-zero $d_3$ differentials, and then
the $E_4^+$-page equals the $E_\infty^+$-page.  

For $E^-$, it turns out that the $d_r$ differential is non-zero
for infinitely many values of $r$.  This does not present a 
convergence problem, because there are only finitely many
non-zero differentials in any given degree.
One consequence is that the orders of the $\rho$-torsion
in $\Ext_{C_2}(1)$ are unbounded.  In other words,
for every $s$, there exists an element
$x$ of  such that $\rho^s x$ is non-zero
but $\rho^{s+t} x$ is zero for some $t > 0$.
This is fundamentally different from
$\Ext_\R(1)$, where $\rho^3 x$ is zero 
if $x$ is not $\rho$-torsion free.

Proposition \ref{d4kQh14} makes explicit these higher differentials.

\begin{prop}\label{d4kQh14} 
For all $k \geq 1$,
\begin{enumerate}
\item
$d_{4k}(\frac{Q}{\rho^{4k}}h_1^{4k}) = 
\frac{\gamma}{ \tau^{4k}}b^k$. 
\item
$d_{4k+1}(\frac{Q}{\rho^{4k+1}}h_1^{4k+3}) = 
\frac{\gamma}{\tau^{4k+2}}ab^k$. 
\end{enumerate}
These differentials are infinitely divisible by $\rho$.
\end{prop}

\begin{pf}
We know that $\frac{\gamma}{\tau^{4k}}$ and $b$ are permanent cycles. On the other hand,
in $\Ext_{C_2}(1)$ the relation $\tau^4 h_1^4 = \rho^4 b$ gives
\[
\frac{\gamma}{\tau^{4k}} b^k = 
\rho^4 \frac{\gamma}{\rho^4\tau^{4k}} b^k  = 
\tau^4 \frac{\gamma}{\rho^4\tau^{4k}} h_1^4 b^{k-1}. \]
Thus $\frac{\gamma}{\tau^{4k}} b^k$ is $h_1$-divisible, which implies that it must be zero in $\Ext_{C_2}(1)$, 
as there is no surviving class in the appropriate degree to support the $h_1$-multiplication. The only Bockstein differential that could hit $\frac{\gamma}{\tau^{4k}} b^k$ is the claimed one.

For the second formula,
the classes $\frac{\gamma}{\tau^{4k+2}}a$ and $b$ are permanent cycles, yet 
\[ \frac{\gamma}{\tau^{4k+2}}a b^k = 
\rho^4 \frac{\gamma}{\rho^4\tau^{4k+2}}a b^{k}  = 
\tau^4 \frac{\gamma}{\rho^4\tau^{4k+2}}a h_1^4 b^{k-1}  \]
in $\Ext_{C_2}(1)$.
But $h_1 a = 0$, so 
$\frac{\gamma}{\tau^{4k+2}}ab^k$ must be zero in $\Ext_{C_2}(1)$, 
forcing the claimed differential.

Alternatively, one can use Proposition
\ref{prop:rho-divisible-Ext} to obtain both differentials.
\end{pf}

Table \ref{tbl:BockDiff} summarizes the Bockstein differentials
that we computed in Sections \ref{sec:ExtAR1} and 
\ref{subsctn:Bock-diff-E^-}.
The differentials above the horizontal line occur in $E^+$,
while the differentials below the horizontal line occur
in $E^-$ and are infinitely divisible by $\rho$.

\begin{table}[h]
\captionof{table}{Bockstein differentials}
\label{tbl:BockDiff}
\begin{center}
\begin{tabular}{LLLLLl} 
\hline
mw &  (s,f, w) & \text{element} & r & d_r  & \text{proof} \\ \hline
1 & (0, 0, -1) & \tau & 1 & \rho h_0
  & Prop. \ref{prop:R-mot-diff} \\
2 & (0, 0, -2) & \tau^2 & 2 & \rho^2 \tau h_1 
  & Prop. \ref{prop:R-mot-diff} \\
3 & (2, 2, -1) & \tau^3 h_1^2 & 3 & \rho^3 a 
  & Prop. \ref{prop:R-mot-diff} \\
\hline  
-2k-2 & (1, 0, 2k+3) & \frac{\gamma}{\rho \tau^{2k+1}} & 1
  & \frac{\gamma}{\tau^{2k+2}} h_0 
  & Prop. \ref{d1ThetaRho} \\
0 & (5, 2, 5) & \frac{Q}{\rho} h_1^3 & 1 
  & \frac{\gamma}{\tau^2} a 
  & Prop. \ref{d1ThetaRho} \\
-4k-3 & (2, 0, 4k+5) & \frac{\gamma}{\rho^2 \tau^{4k+2}} & 2 
  & \frac{\gamma}{\tau^{4k+3}} h_1 
  & Prop. \ref{d2ThetaRho2Tau} \\
0 & (8k+1, 4k-1, 8k+1) & \frac{Q}{\rho^{4k}} h_1^{4k} & 4k 
  & \frac{\gamma}{\tau^{4k}}b^k 
  & Prop. \ref{d4kQh14} \\
0 & (8k+5, 4k+2, 8k+5) & \frac{Q}{\rho^{4k+1}} h_1^{4k+3} & 4k+1
  & \frac{\gamma}{\tau^{4k+2}}ab^k 
  & Prop. \ref{d4kQh14} \\
\hline
\end{tabular}
\end{center}
\end{table}

The $\rho$-Bockstein differentials of Sections \ref{sec:ExtAR1}
and \ref{sec:BockNegCone} allow us to completely compute
the $E_\infty$-page of the $\rho$-Bockstein spectral sequence
for $\Ext_{C_2}(1)$.

\subsection{$\rho$-Bockstein differentials in $E^{-}$ for $\cE^{C_2}(1)$}\label{sec:BockE1NC}

For comparison, we also carry out the analogous but easier computation over $\cE^{C_2}(1)$ rather than $\cA^{C_2}(1)$. 
Besides $d_1 \left( \frac{\gamma}{\rho \tau^{2k+1}}\right) = \frac{\gamma}{ \tau^{2k+2}} h_0$, the only other Bockstein differential is given in the following result.

\begin{prop}\label{d3ThetaRho2Tau2}
$d_3 \left( \frac{\gamma}{\rho^3 \tau^{4k+2}}\right) = \frac{\gamma}{ \tau^{4k+4}} v_1$
 for all $k \geq 0$.
This differential is infinitely divisible by $\rho$.
\end{prop} 

\begin{pf}
The differential $d_3(\tau^2) = \rho^3 v_1$ of Remark~\ref{ExtE1} gives
\[
\begin{split}
0 =
d_3\left( \tau^{4k+2} \frac{\gamma}{\rho^3 \tau^{4k+2}} \right) &=
\tau^{4k+2} d_3\left( \frac{\gamma}{\rho^3 \tau^{4k+2}}\right) +
\rho^3 \tau^{4k} \frac{\gamma}{\rho^3 \tau^{4k+2}} v_1 \\
& = \tau^{4k+2} d_3\left( \frac{\gamma}{\rho^3 \tau^{4k+2}}\right) + \frac{\gamma}{\tau^2} v_1.
\end{split} \]
\end{pf}

\section{Some Massey products}
\label{sctn:Massey}

The final step in the computation of
$\Ext_{C_2}(1)$ is to determine multiplicative extensions
that are hidden in the $\rho$-Bockstein $E_\infty$-page.
In order to do this, we will need some
Massey products in $\Ext_{C_2}(1)$. 
Table \ref{tbl:ExtA1Massey}
summarizes the information that we will need.  

\begin{thm}
Some Massey products in $\Ext_{C_2}(1)$ are given
in Table \ref{tbl:ExtA1Massey}.  All have zero indeterminacy.
\end{thm}


\begin{table}[h]
\captionof{table}{Some Massey products in  $\Ext_{C_2}(1)$
\label{tbl:ExtA1Massey}}
\begin{center}
\begin{tabular}{LLLLl} 
\hline
mw & (s,f,w) & \text{bracket} & \text{contains} & \text{proof}
\\ \hline
1 & (1, 1, 0) & \bracket{\rho, h_0, h_1} & \tau h_1 
  & $d_1(\tau) = \rho h_0$ \\
1 & (2, 2, 1) & \bracket{ h_0, h_1, h_0 } & \tau h_1^2 & classical \\
2 & (4, 3, 2) & \bracket{ \tau h_1 \cdot h_1, h_1,  h_0} & a 
  & classical  \\
2 & (0, 1, -2) & \bracket{ \rho \tau h_1, \rho, h_0 } & \tau^2 h_0 &
$d_2(\tau^2) = \rho^2 \tau h_1$ \\
4 & (8, 5, 4) & \bracket{ a, h_1, \tau h_1^2 } & h_0 b &
  classical \\
\hline
-4 & (0, 0, 4) & 
  \bracket{\tau^2 h_0, \rho, \frac{\gamma}{\tau^6} } &
  \frac{\gamma}{\tau^3} & $d_1 (\tau^3) = \rho \tau^2 h_0$ \\
-4 & (0, 0, 4) & 
  \bracket{h_0, \rho, \frac{\gamma}{\tau^4} } &
  \frac{\gamma}{\tau^3} & $d_1 (\tau) = \rho h_0$ \\
-3 & (1, 0, 4) & 
  \bracket{\rho, \frac{\gamma}{\tau^4}, \tau h_1} & 
   \frac{\gamma}{\rho \tau^2} 
  & $d_2 \left( \frac{\gamma}{\rho^2 \tau^2} \right) = 
   \frac{\gamma}{\tau^3} h_1$  \\
-3 & (0, 0, 3) &
   \bracket{ \rho \tau h_1, \rho, \frac{\gamma}{\tau^4} } &
   \frac{\gamma}{\tau^2}  &
   $d_2(\tau^2) = \rho^2 \tau h_1$ \\
-2 & (4, 2, 6) & 
 \bracket{ \frac{\gamma}{\tau^3}, h_1, \tau h_1 \cdot h_1 }  &
  \frac{\gamma}{\rho^2 \tau} h_1^2  &
  $d_2 \left( \frac{\gamma}{\rho^2 \tau^2} \right) = 
  \frac{\gamma}{\tau^3} h_1$  \\
-2 & (0, 0, 2) & 
  \bracket{ \tau^2 h_0, \rho, \frac{\gamma}{\tau^4} } &
  \frac{\gamma}{\tau}  & 
  $d_1(\tau^3) = \rho \tau^2 h_0$ \\
-2 & (0,0,2) &  \bracket{h_0, \rho, \frac{\gamma}{\tau^2}}  & 
  \frac{\gamma}{\tau}   & $d_1(\tau) = \rho h_0$ \\
-2 & (2, 1, 4) &  
  \bracket{ h_1, h_0, \frac{\gamma}{\tau^2}}  & 
   \frac{\gamma}{\rho\tau} h_1 & 
   $d_1 \left( \frac{\gamma}{\rho \tau} \right) = 
  \frac{\gamma}{ \tau^2} h_0$ \\
\hline
0 & (4, 2, 4) +  & 
   \bracket{ \rho, \frac{\gamma}{\tau^{4k+2}}, a b^k } &
  \frac{Q}{\rho^{4k}} h_1^{4k+3}  & 
  $d_{4k+1} \left( \frac{Q}{\rho^{4k+1}} h_1^{4k+3} \right) =$  \\
  & +(8k, 4k, 8k) & & & $= \frac{\gamma}{\tau^{4k+2}} a b^k$ \\
0 & (8, 3, 8) + &
  \bracket{ \rho, \frac{\gamma}{\tau^{4k+4}}, b^{k+1} } &
  \frac{Q}{\rho^{4k+3}} h_1^{4k+4} & 
  $d_{4k+4} \left( \frac{Q}{\rho^{4k+4}} h_1^{4k+4} \right) =$ \\
  & +(8k, 4k, 8k) & & &  $= \frac{\gamma}{\tau^{4k+4}} b^{k+1}$ \\
 \hline
\end{tabular}
\end{center}
\end{table}

\begin{proof}
For some Massey products in Table \ref{tbl:ExtA1Massey}, a
$\rho$-Bockstein differential is displayed in the last column.
In these cases, May's Convergence Theorem 
\cite{matric} \cite{Isstems}*{Chapter 2.2} applies, and the Massey
product can be computed with the given differential.
Roughly speaking, May's Convergence Theorem
says that Massey products in $\Ext_{C_2}(1)$ can be computed
with any $\rho$-Bockstein differential.  Beware that May's Convergence
Theorem requires technical hypotheses involving ``crossing 
differentials" that are not always satisfied.
Failure to check these conditions can lead to mistaken calculations.

The proofs for other Massey products in Table \ref{tbl:ExtA1Massey} 
are described as ``classical".  In these cases, the Massey product
already occurs in $\Ext_{\cla}$.  
\end{proof}

\begin{rmk}
The eight Massey products in the middle Section of 
Table \ref{tbl:ExtA1Massey} are only the first examples of 
infinite families that are $\tau^4$-periodic.  For example, 
$\bracket{\tau^2 h_0, \rho, \frac{\gamma}{\tau^{4k+6}} }$
equals $\frac{\gamma}{\tau^{4k+3}}$ for all $k \geq 0$,
and
$\bracket{ \rho, \frac{\gamma}{\tau^{4k+4}}, \tau h_1 }$
equals $\frac{\gamma}{\tau^{4k+3}}$ for all $k \geq 0$.
\end{rmk}

\section{Hidden extensions}
\label{sec:hidden}

We now
determine multiplicative extensions
that are hidden in the $\rho$-Bockstein $E_\infty$-page.
We have already determined some of these hidden extensions
in Section \ref{sec:ExtAR1}.
In this section, we establish additional hidden relations
on elements associated with the negative cone.
We have not attempted a completely exhaustive analysis of the
ring structure of $\Ext_{C_2}(1)$.

Recall that $\Ext_{C_2}(1)$ is a square-zero extension of 
$\Ext_\R(1)$. This eliminates many possible hidden extensions.
For example, $({Q} h_1^3)^2$ is zero in $\Ext_{C_2}(1)$.

\begin{prop}\label{prop:Q/rho^4k h1^4k+3 1}
For all $k \geq 0$,
\begin{enumerate}
\item 
$h_0 \cdot \frac{Q}{\rho^{4k}} h_1^{4k+3}= 
\frac{\gamma}{\tau^{4k+1}} ab^k$.
\item 
$a\cdot \frac{Q}{\rho^{4k}} h_1^{4k+3} = 
\frac{\gamma}{\tau^{4k+1}} h_0 b^{k+1}$. 
\end{enumerate}
\end{prop}

\begin{pf}
\mbox{}
\begin{enumerate}
\item\label{h0Qh17}  
\[
h_0 \bracket{ \rho, \frac{\gamma}{\tau^{4k+2}}, a b^k } =
\bracket{ h_0, \rho, \frac{\gamma}{\tau^{4k+2}} } a b^k.  
\]
\item
Using part (1),
we have that
\[
h_0 a \cdot \frac{Q}{\rho^{4k} h_1^{4k+3}} =
a \cdot \frac{\gamma}{\tau^{4k+1}} a b^k = 
\frac{\gamma}{\tau^{4k+1}} h_0^2 b^{k+1}, 
\]
which is non-zero.  Therefore,
$a \cdot \frac{Q}{\rho^{4k}} h_1^{4k+3}$ must also be non-zero,
and $\frac{\gamma}{\tau^{4k+1}} h_0 b^{k+1}$ is the only nonzero 
class in the appropriate tridegree.
\end{enumerate}
\end{pf}

\begin{prop}\label{prop:Q/rho^4k h1^4k+3 2}
For all $k \geq 1$,
\begin{enumerate}
\item
$\tau^2 a \cdot \frac{Q}{\rho^{4k}} h_1^{4k+3} = 
\frac{\gamma}{\tau^{4k-1}} h_0 b^{k+1} + \frac{Q}{\rho^{4k-3}} h_1^{4k+2} b$.
\item
$\tau^4 \cdot \frac{Q}{\rho^{4k}} h_1^{4k+3} = 
\frac{Q}{\rho^{4k-4}} h_1^{4k-1} b$.
\item
$\tau^2 h_0 \cdot \frac{Q}{\rho^{4k}} h_1^{4k+3} = 
\frac{\gamma}{\tau^{4k-1}} ab^k$.
\end{enumerate}
\end{prop}

\begin{pf}
\mbox{}
\begin{enumerate}
\item
Using Proposition \ref{prop:Q/rho^4k h1^4k+3 1}(1),
we have that
\[ 
h_0 \cdot \tau^2 a \cdot \frac{Q}{\rho^{4k}} h_1^{4k+3} =
\tau^2 a \cdot \frac{\gamma}{\tau^{4k+1}} a b^k =
\frac{\gamma}{\tau^{4k-1}} h_0^2 b^{k+1},
\]
which is non-zero.  This shows that
$\tau^2 a \cdot \frac{Q}{\rho^{4k}} h_1^{4k+3}$ is either
$\frac{\gamma}{\tau^{4k-1}} h_0 b^{k+1}$ or 
$\frac{\gamma}{\tau^{4k-1}} h_0 b^{k+1} + 
\frac{Q}{\rho^{4k-3}} h_1^{4k+2} b$. 

On the other hand, 
\[ 
h_1\cdot \tau^2 a \cdot \frac{Q}{\rho^{4k}} h_1^{4k+3} =
\rho^3 b \cdot \frac{Q}{\rho^{4k}} h_1^{4k+3} =
\frac{Q}{\rho^{4k-3}} h_1^{4k+3} b. 
\]
Therefore, 
$\tau^2 a \cdot \frac{Q}{\rho^{4k}} h_1^{4k+3}$ must equal
$\frac{\gamma}{\tau^{4k-1}} h_0 b^{k+1} + 
\frac{Q}{\rho^{4k-3}} h_1^{4k+2} b$.  
\item
Using Proposition \ref{prop:Q/rho^4k h1^4k+3 1}(1),
we have that 
\[
h_0 \cdot \tau^4 \cdot \frac{Q}{\rho^{4k}} h_1^{4k+3} =
\tau^4 \frac{\gamma}{\tau^{4k+1}} a b^k =
\frac{\gamma}{\tau^{4k-3}} a b^k,
\]
which is non-zero.  This shows that
$\tau^4 \cdot \frac{Q}{\rho^{4k}} h_1^{4k+3}$ is also 
non-zero, and there is just one possible value.
\item
\[ 
\tau^2 h_0 \bracket{ \rho, \frac{\gamma}{\tau^{4k+2}}, ab^{k} } =
\bracket{ \tau^2 h_0, \rho, \frac{\gamma}{\tau^{4k+2}} } ab^{k}.  
\]
\end{enumerate}
\end{pf}

\begin{prop}\label{prop:Q/rho^4k+3 h1^4k+4}
For all $k \geq 0$,
\begin{enumerate}
\item $ h_0 \cdot \frac{Q }{\rho^{4k+3}} h_1^{4k+4} = 
\frac{\gamma}{\tau^{4k+3}} b^{k+1}$.
\item 
$\tau h_1 \cdot \frac{Q}{\rho^{4k+3}} h_1^{4k+4} = 
 \frac{\gamma}{\rho\tau^{4k+2}} b^{k+1}$.
\item 
$\tau^2 h_0 \cdot \frac{Q}{\rho^{4k+3}} h_1^{4k+4} = 
 \frac{\gamma}{\tau^{4k+1}} b^{k+1}$.
\item 
$a \cdot \frac{Q}{\rho^{4k+3}} h_1^{4k+4} = 
\frac{\gamma}{\rho^2 \tau^{4k+1}} h_1^2 b^{k+1}$.
\item 
$\tau^2 a \cdot \frac{Q}{\rho^{4k+3}} h_1^{4k+4} = 
\frac{Q}{\rho^{4k}} h_1^{4k+3} b$.
\end{enumerate} 
\end{prop}

\begin{pf}
\mbox{}
\begin{enumerate}
\item\label{h0Qh14} 
\[
h_0 \bracket{ \rho, \frac{\gamma}{\tau^{4k+4}}, b^{k+1} } =
\bracket{ h_0, \rho, \frac{\gamma}{\tau^{4k+4}} } b^{k+1}.
\]
\item 
\[
\rho \tau h_1 \bracket{ \rho, \frac{\gamma}{\tau^{4k+4}}, b^{k+1} } =
\bracket{\rho \tau h_1, \rho, \frac{\gamma}{\tau^{4k+4}} } b^{k+1}.
\]
\item 
\[
\tau^2 h_0 \bracket{ \rho, \frac{\gamma}{\tau^{4k+4}}, b^{k+1} } =
\bracket{ \tau^2 h_0, \rho, \frac{\gamma}{\tau^{4k+4}} } b^{k+1}.
\]
\item 
Using part (1), we have that
\[
h_0 a \cdot \frac{Q}{\rho^{4k+3}} h_1^{4k+4} =
a \cdot \frac{\gamma}{\tau^{4k+3}} b^{k+1} =
\frac{\gamma}{\tau^{4k+3}} a b^{k+1},
\]
which is non-zero.  Therefore,
$a \cdot \frac{Q}{\rho^{4k+3}} h_1^{4k+4}$ must also be non-zero,
and there is just one possibility.
\item
Using part (1), we have that
\[
h_0 \cdot \tau^2 a \cdot \frac{Q}{\rho^{4k+3}} h_1^{4k+4} =
\tau^2 a \cdot \frac{\gamma}{\tau^{4k+3}} b^{k+1} =
\frac{\gamma}{\tau^{4k+1}} a b^{k+1},
\]
which is non-zero.  This shows that
$\tau^2 a \cdot \frac{Q}{\rho^{4k+3}} h_1^{4k+4}$ is also
non-zero, and there is just one possible value.
\end{enumerate} 
\end{pf}

\begin{prop}
\label{prop:theta/rho^2 tau^4k+1 h1^2}
For all $k \geq 0$,
\begin{enumerate}
\item 
$h_0 \cdot \frac{\gamma}{\rho^2 \tau^{4k+1}} h_1^2 = 
\frac{\gamma}{\tau^{4k+3}} a$.
\item 
$a \cdot \frac{\gamma}{\rho^2\tau^{4k+1}} h_1^2 = \frac{\gamma}{\tau^{4k+3}} h_0 b$.
\item
$\tau^2 a \cdot \frac{\gamma}{\rho^2\tau^{4k+1}} h_1^2 =
\frac{\gamma}{\tau^{4k+1}} h_0 b$.
\end{enumerate} 
\end{prop}

\begin{pf}
\mbox{}
\begin{enumerate}
\item
\label{templabel} 
\[
\bracket{ \frac{\gamma}{\tau^{4k+3}}, h_1, \tau h_1 \cdot h_1 } h_0 =
\frac{\gamma}{\tau^{4k+3}} \bracket{ h_1, \tau h_1 \cdot h_1, h_0 }.
\]
\item
Using part (1), we have that
\[
h_0 a \cdot \frac{\gamma}{\rho^2 \tau^{4k+1}} h_1^2 =
a \cdot \frac{\gamma}{\tau^{4k+3}} a,
\]
which equals $\frac{\gamma}{\tau^{4k+3}} h_0^2 b$
modulo a possible error term involving higher powers of $\rho$.
Using that $h_1 a = 0$, we conclude that the error term is zero.
\item
Using part (1), we have that
\[
h_0 \cdot \tau^2 a \cdot \frac{\gamma}{\rho^2 \tau^{4k+1}} h_1^2 =
\tau^2 a \cdot \frac{\gamma}{\tau^{4k+3}} a =
\frac{\gamma}{\tau^{4k+1}} h_0^2 b,
\]
which is non-zero.  This shows that
$\tau^2 a \cdot \frac{\gamma}{\rho^2 \tau^{4k+1}} h_1^2$ is also
non-zero, and there is just one possible value.
\end{enumerate} 
\end{pf}

\begin{prop}\label{prop:h0extn}
For all $k \geq 0$,
\begin{enumerate}
\item 
$h_0 \cdot \frac{\gamma}{\rho \tau^{4k+1}} h_1 = 
\frac{\gamma}{\tau^{4k+1}} h_1^2$.
\item $h_0 \cdot \frac{\gamma}{\rho \tau^{4k+2}} = 
\frac{\gamma}{\tau^{4k+2}} h_1$.
\end{enumerate}
\end{prop}

\begin{pf}
All of these extensions follow from Massey product shuffles:
\begin{enumerate}
\item
\[ 
h_0 \bracket{ h_1, h_0, \frac{\gamma}{\tau^{4k+2}} }
= \bracket{ h_0, h_1, h_0 } \frac{\gamma}{\tau^{4k+2}}.
\]
\item 
\[ 
h_0 \bracket{\rho,\frac{\gamma}{\tau^{4k+4}},\tau h_1} = 
\bracket{h_0, \rho, \frac{\gamma}{\tau^{4k+4}}} \tau h_1.
\]
\end{enumerate}  
\end{pf}

\begin{prop}
For all $k \geq 0$,
\begin{enumerate}
\item $h_1 \cdot \frac{\gamma}{\rho \tau^{4k+4}} h_1^2 = \frac{\gamma}{\tau^{4k+6}} a$.
\item $h_1 \cdot \frac{\gamma}{\rho^3 \tau^{4k+6}}a = \frac{\gamma}{\tau^{4k+8}} b$.
\end{enumerate}
\end{prop}

\begin{pf}
\mbox{}
\begin{enumerate}
\item 
\[
\tau h_1 \cdot h_1 \bracket{h_1, h_0, \frac{\gamma}{\tau^{4k+6}}} =
\bracket{\tau h_1 \cdot h_1, h_1, h_0} \frac{\gamma}{\tau^{4k+6}}.
\]
Alternatively, this $h_1$ extension is forced
by Proposition~\ref{RhoColoc}.
\item 
We have
\[
h_1 \cdot \frac{\gamma}{\rho^3 \tau^{4k+6}} a =
\frac{\gamma}{\rho^3 \tau^{4k+8}} h_1 \cdot \tau^2 a =
\frac{\gamma}{\rho^3 \tau^{4k+8}} \rho^3 b =
\frac{\gamma}{\tau^{4k+8}},
\]
where the second equality follows from Table \ref{tbl:ExtA1Rrel}.
\end{enumerate} 
\end{pf}

Over $\cE^{C_2}(1)$, the only hidden multiplication is

\begin{prop}
In $\Ext_{\cE^{C_2}(1)}$, we have 
$h_0 \cdot \frac{\gamma}{\rho^2\tau^{4k+2}}v_1^n = \frac{\gamma}{ \tau^{4k+3}}v_1^{n+1}$.
for all $k,n \geq 0$.
\end{prop}

\begin{pf}
\[
h_0 \cdot \frac{\gamma}{\rho^2\tau^{2}} = h_0 \bracket{ \rho, \frac{\gamma}{\tau^{4}}, v_1 } =
\bracket{ h_0, \rho, \frac{\gamma}{\tau^{4}} }v_1 = \frac{\gamma}{\tau^3} v_1.
\]
\end{pf}

\subsection{$\Ext_{C_2}(1)$}

The charts in Section \ref{sctn:chart}
depict  $\Ext_{C_2}(1)$ graphically.
Table \ref{tbl:MultEinfpage} gives 
generators for $\Ext_{C_2}(1)$.
The elements above the horizontal line are multiplicative generators for
$\Ext_\R(1)$.  The elements below the horizontal generate
$\Ext_{NC}$ in the following sense.  Every element of $\Ext_{NC}$
can be formed by starting with one of these listed elements, multiplying
by elements of $\Ext_\R(1)$, and then dividing by $\rho$.

The elements in Table \ref{tbl:MultEinfpage} are not multiplicative
generators for $\Ext_{C_2}(1)$ in the usual sense, because we allow
for division by $\rho$.  
For example, $\frac{\gamma}{\rho^2\tau}h_1^2$ is indecomposable
in the usual sense, yet it does not appear in Table \ref{tbl:MultEinfpage}
because 
$\rho^2 \cdot \frac{\gamma}{\rho^2\tau}h_1^2 = \frac{\gamma}{\tau}h_1^2$
is decomposable.

The point of this notational approach is that the elements of
$\Ext_{NC}$ are most easily understood as families
of $\rho$-divisible elements.

\begin{table}[h]
\captionof{table}{Generators for $\Ext_{C_2}(1)$}
\label{tbl:MultEinfpage}
\begin{center}
\begin{tabular}{LLL} 
\hline
mw & (s,f,w) & \text{element}  \\ \hline
0 & (-1,0,-1) & \rho   \\
0 & (0,1,0) & h_0   \\
0 & (1,1,1) & h_1  \\
1 & (1, 1, 0) & \tau h_1  \\
2 & (0, 1, -2) & \tau^2 h_0  \\
2 & (4,3,2) & a   \\
4 & (0,0,-4) & \tau^4  \\
4 & (4,3,0) & \tau^2 a  \\
4 & (8,4,4) & b   \\
\hline
-k-1 & (0, 0, k+1) & \frac{\gamma}{\tau^k}  \\
0 & (4, 2, 4) & {Q} h_1^3 \\
\hline
\end{tabular}
\end{center}
\end{table}

\subsection{The ring homomorphism $q_*:\Ext_{\cA^{C_2}(1)} \rtarr \Ext_{\cE^{C_2}(1)}$}

It is worthwhile to consider the comparison to $\Ext_{\cE^{C_2}(1)}$. We already described the map on the 
summand arising from the positive cone in Proposition~\ref{prop:ExtA1toE1}. The map on the summand for the negative cone is given as follows.

\begin{prop}\label{prop:ExtNCA1toE1}
The homomorphism $q_*:\Ext_{\cA^{\R}(1)}(NC, \bM_2^\R) \rtarr \Ext_{\cE^{\R}(1)}(NC, \bM_2^\R)$ induced by the quotient $q:\cA^R(1)_* \rtarr \cE^\R(1)$ of Hopf algebroids is given as in Table~\ref{tbl:ExtNCA1toE1}.
\end{prop}

\begin{table}[h]
\captionof{table}{The homomorphism $\Ext_{\cA^\R(1)}(NC) \rtarr \Ext_{\cE^\R(1)}(NC)$
\label{tbl:ExtNCA1toE1}}
\begin{center}
\begin{tabular}{LLLL} 
\hline
mw & (s,f,w) & \text{$x\in\Ext_{\cA^\R(1)}NC$} & \text{$q_*x\in\Ext_{\cE^\R(1)}NC$}  \\ \hline
0 & (4,2,4) + k(8,4,4) & \frac{Q}{\rho^{4k}} h_1^{4k+3}  & \frac{\gamma}{\tau}v_1^{4k+2} \\   
0 & (8,3,8) + k(8,4,4) & \frac{Q}{\rho^{4k+4}} h_1^{4k+3}  & \frac{\gamma}{\rho^2\tau^2}v_1^{4k+3} \\   
-2 & (0,0,2) & \frac{\gamma}{\tau}  & \frac{\gamma}{\tau} \\   
-2 & (2,1,4) & \frac{\gamma}{\rho\tau}h_1  & \frac{\gamma}{\tau^2} v_1 \\   
-2 & (4,2,6) & \frac{\gamma}{\rho^2\tau}h_1^2  & \frac{\gamma}{\tau^3} v_1^2 \\   
-3 & (1,0,4) & \frac{\gamma}{\rho\tau^2}  & \frac{\gamma}{\rho\tau^2} \\   
-5 & (0,0,5) & \frac{\gamma}{\tau^4}  & \frac{\gamma}{\tau^4} \\   
 \hline
\end{tabular}
\end{center}
\end{table}

\begin{pf}
For the classes of the form $\frac{\gamma}{\rho^j\tau^k}$, this is true on the cobar complex. For the classes of the form $\frac{Q}{\rho^j}h_1^n$, this follows from the $h_0$-extension given in Proposition~\ref{prop:Q/rho^4k h1^4k+3 1} and the value $q_*(a) = h_0v_1^2$. Similarly, the value on $\frac{\gamma}{\rho^2\tau}h_1^2$ is obtained by combining Proposition~\ref{prop:theta/rho^2 tau^4k+1 h1^2} with the value of $q_*(a)$. Lastly, the value on $\frac{\gamma}{\rho\tau}h_1$ follows from $q_*(\tau h_1) = \rho v_1$.
\end{pf}

\begin{rmk}
Note that, on the other hand, the hidden $h_0$-extensions on classes in $\Ext_{\cA^{C_2}(1)}$, such as ${Q}h_1^3$, can also be deduced from the homomorphism $q_*$ if its values are determined by other means.
\end{rmk}

\section{The spectrum $ko_{C_2}$}
\label{sctn:koC2}

Let $\Sp$ denoted the category of spectra, and let $\SpC$ denote the
category of ``geniune'' $C_2$-spectra \cite{Alaska}*{Chapter~XII},
obtained from the category of based $C_2$-spaces by inverting
suspension with respect to the one-point compactification
$S^{2,1}$ of the regular representation $(\C,z\mapsto \overline{z})$.
There are restriction and fixed-point functors
\[ \iota^*:\mathrm{Ho}(\SpC) \rtarr \mathrm{Ho}(\Sp), 
 \qquad (-)^{C_2}:\mathrm{Ho}(\SpC) \rtarr \mathrm{Ho}(\Sp)
\]
which detect the homotopy theory of $C_2$-spectra, meaning that a map
$f$ in $\mathrm{Ho}(\SpC)$ is an equivalence if and only if
$\iota^*(f)$ and $f^{C_2}$ are equivalences in $\mathrm{Ho}(\Sp)$.
Moreover, a sequence $X\rtarr Y \rtarr Z$ is a cofiber sequence in
$\mathrm{Ho}(\SpC)$ if and only if applying both functors $\iota^*$
and $(-)^{C_2}$ yield cofiber sequences. Both statements follow from
the fact \cite{SS}*{Example~3.4(i)} that the pair of $C_2$-spectra $\{
\Sigma^\infty_{C_2} S^0, \Sigma^\infty_{C_2} C_2\,_+\}$ give a compact
generating set for $\mathrm{Ho}(\SpC)$.  Beware that we are discussing
categorical fixed-point spectra here, not geometric fixed-point spectra.

Recall (cf. \cite{L}*{Proposition~3.3}) that for a $C_2$-spectrum $X$,
the equivariant connective cover $X\langle 0 \rangle\xrtarr{q}X$ is a
$C_2$-spectrum such that:
\begin{enumerate}
\item
$\iota(q)$ is the connective cover of the underlying spectrum $X$, and
\item
$q^{C_2}$ is the connective cover of $X^{C_2}$.
\end{enumerate}

Recall that $KO_{C_2}$ is the $C_2$-spectrum representing the
$\bK$-theory of $C_2$-equivariant real vector bundles
\cite{Alaska}*{Ch.~XIV}.

\begin{defn}\label{def:koC2} Let $ko_{C_2}$ be the equivariant
connective cover $KO_{C_2} \langle 0 \rangle$ of $KO_{C_2}$.
\end{defn}

We also have a description from the point of view of equivariant
infinite loop space theory.

\begin{thm}\cite{Me}*{Theorem 7.1} $ko_{C_2}\simeq \bK_{C_2}(\R)$,
where $\R$ is considered as a topological ring with trivial
$C_2$-action.
\end{thm}

The underlying spectrum of $ko_{C_2}$ is $ko$.

\begin{lemma}
The fixed-point spectrum of $ko_{C_2}$ is $(ko_{C_2})^{C_2} \simeq ko\vee ko$.
\end{lemma}

\begin{pf}
This is a specialization of the statement that, if $X$ is any space equipped with a trivial $G$-action, then $KO_G(X)$ is 
isomorphic to $RO(G) \otimes KO(X)$ \cite{Alaska}*{Section XIV.2}. 
Alternatively, from the point of view of algebraic $\bK$-theory, we have $\bK_{C_2}(\R)^{C_2}\simeq \bK(\R[C_2])$ \cite{Me}*{Theorem 1.2}, and $\R[C_2]\iso \R\times \R$. It follows that 
\[ (ko_{C_2})^{C_2} \simeq \bK_{C_2}(\R)^{C_2} \simeq \bK(\R)\times \bK(\R) \simeq ko\vee ko.\]
\end{pf}

We are working towards a description of
the $C_2$-equivariant cohomology of $ko_{C_2}$
as the quotient $\cA^{C_2} /\!/ \cA^{C_2}(1)$.  
This will allow us to express the $E_2$-page of the Adams spectral
sequence for $ko_{C_2}$ in terms of 
the cohomology of
$\cA^{C_2}(1)$.
The main step will be to establish the cofiber sequence of
Proposition~\ref{EquivWoodProp}. In preparation, we first prove some
auxiliary results.

\begin{defn}\label{DefnRho} Let $\rho$ be the element of 
$\pi_{-1,-1}$
determined by the inclusion
$S^{0,0} \into S^{1,1}$ 
of fixed points. 
\end{defn}

Note that the element 
$\rho\in\pi_{-1,-1}$
induces multiplication by $\rho$ in cohomology
under the Hurewicz homomorphism.

Recall that the real $C_2$-representation ring $RO(C_2)$ is
a rank two free abelian group. Generators are given by the trivial
 one-dimensional representation $1$ and the sign representation
 $\sigma$. 
Let $A(C_2)$ denote the Burnside ring of $C_2$, defined as the 
Grothendieck group associated to the monoid of finite $C_2$-sets. 
This is also a rank two free abelian group, with generators the trivial 
one-point $C_2$-set $1$ and the free $C_2$-set $C_2$. As a ring, 
$A(C_2)$ is isomorphic to $\Z[C_2]/(C_2^2 - 2C_2)$.

The linearization map $A(C_2)\rtarr RO(C_2)$ sending a $C_2$-set
to the induced permutation representation is an isomorphism, sending
the free orbit $C_2$ to the regular representation $1\oplus \sigma$.
Recall that the Euler characteristic moreover gives an isomorphism 
from $A(C_2)$ to $\pi_0(S^{0,0})$ \cite[Corollary to Proposition~1]{Se}.

\begin{lemma}\label{FxdPtsko}
The $C_2$-fixed point spectrum of $\Sigma^{1,1}ko_{C_2}$ is equivalent
to $ko$.
\end{lemma}

\begin{pf}
Recall the cofiber sequence $C_2\,_+ \xrtarr{\pi} S^{0,0}
\xrtarr{\rho} S^{1,1}$ of $C_2$-spaces.  This yields a cofiber
sequence
\[ 
C_2\,_+\smsh ko_{C_2} \xrtarr{\pi} ko_{C_2} \xrtarr{\rho} \Sigma^{1,1} ko_{C_2}
\]
of equivariant spectra.
Passing to fixed point spectra gives the cofiber sequence
\[ ko \xrtarr{\pi^{C_2}} ko \vee ko \xrtarr{\rho^{C_2}} (\Sigma^{1,1} ko_{C_2})^{C_2}.
\]
In the analogous sequence for the sphere $S^{0,0}$, the map $\pi^{C_2}$
is induced by the split inclusion $\Z \rtarr A(C_2)$ sending $1$ to the free orbit $C_2$.
It follows that
the map $\pi^{C_2}$ is induced by the split inclusion $\Z\rtarr RO(C_2)$
that takes $1$ to the regular representation $\rho_{C_2}$, and this induces a splitting of the cofiber sequence.
Therefore, $(\Sigma^{1,1} ko_{C_2})^{C_2}$
is equivalent to $ko$. 
\end{pf}

Recall that $k\R$ denotes the equivariant connective cover
$K\R \langle 0 \rangle$
 of Atiyah's $K$-theory `with reality' spectrum $K\R$ \cite{At}. 
The latter theory classifies complex vector bundles equipped with
a conjugate-linear action of $C_2$.
The underlying spectrum of $k\R$ is $ku$, and its fixed-point spectrum
is $ko$.

\begin{thm}\cite{Me}*{Theorem 7.2} $k\R\simeq \bK_{C_2}(\C)$, where $\C$ is considered as a topological ring with $C_2$-action given by complex conjugation.
\end{thm}

\begin{defn}\label{DefnEta}
The $C_2$-equivariant Hopf map $\eta$ is
\[ 
\C^2 - \{0\} \rightarrow \C\bP^1: 
 (x,y) \mapsto [x:y], 
\]
where both source and target are given the complex conjugation action.
\end{defn}

As $\C\iso \R[C_2]$, the punctured representation $\C^2 - \{0\}$ 
is homotopy equivalent to $S^{3,2}$,
and $\C\bP^1$ 
is homeomorphic to $S^{2,1}$.
It follows that $\eta$ gives rise to a stable homotopy class in
$\pi_{1,1}$.

\begin{rmk} The element $\eta$ only defines a specific element of 
$\pi_{1,1}$ after choosing isomorphisms $\C^2-\{0\}\iso S^{3,2}$ 
and $\C\bP^1\iso S^{2,1}$ in the homotopy category. We follow the 
choices of \cite{DIHopf}*{Example~2.12}. By Proposition~C.5 of 
\cite{DIHopf}, with these choices, the induced map 
$\eta^{C_2}:S^1 \rtarr S^1$ on fixed points
is a map of degree $-2$.
\end{rmk}

\begin{lemma}\label{EtaRhoSphere} The element 
$\rho\eta$ in $\pi_{0,0}$ corresponds to 
the element $C_2-2$ of $A(C_2)$.
\end{lemma}

\begin{pf}
In $\pi_{0,0}$, 
we have $(\eta\rho)^2=-2\eta\rho$ (\cite[Lemma~6.1.2]{Mo}). 
The nonzero solutions to $x^2=-2x$ in $A(C_2)$ are $x=-2$, $x=C_2-2$, and $x=-C_2$. 
The only such solution which restricts to zero at the trivial subgroup is $x=C_2-2$.
\end{pf}

\begin{lemma}\label{EtaFixed} 
The induced map $\eta^{C^2}:(\Sigma^{1,1}ko_{C_2})^{C_2} \rtarr (ko_{C_2})^{C_2}$ is equivalent to $ko\xrtarr{(-1,1)} ko\vee ko$.
\end{lemma}

\begin{pf}
To determine the fixed map $\eta^{C_2}$, we use that a map $X\xrtarr{\varphi} Y$ of $C_2$-spectra induces a commutative diagram 
\[ 
\xymatrix @R=1.5em @C=2em{
X^{C_2} \ar[r]^-{\varphi^{C_2}} \ar[d] & Y^{C_2} \ar[d] \\
X^e \ar[r]_{\varphi^e}  & Y^e, \\
  }\]
in which the vertical maps are the inclusions of fixed points. In the case of $\eta$ on $ko_{C_2}$,
this gives the diagram
\[ 
\xymatrix @R=1.5em @C=2em{
ko\simeq(\Sigma^{1,1}ko_{C_2})^{C_2} \ar[r]^-{\eta^{C_2}} \ar[d]_0 & ko\vee ko \simeq (ko_{C_2})^{C_2} \ar[d]^\nabla \\
\Sigma^{1} ko \ar[r]_{\iota^*\eta}  & ko, \\
  }\]
where $\nabla$ is the fold map,
as both the sign representation $\sigma$ and the trivial 
representation $1$ of $C_2$ restrict to the 1-dimensional trivial 
representation of the trivial group.
This shows that $\eta^{C_2}$ 
factors through the fiber of $\nabla$, so that
$\eta^{C_2}$ must be of the form $(k,-k)$ for some integer $k$. 
On the other hand, we have the commutative diagram
\[ 
\xymatrix{
ko \otimes RO(C_2) \ar[r] & ko \ar[r] & ko\otimes RO(C_2) \\
(ko_{C_2})^{C_2} \ar[r]^{\rho^{C_2}} \ar[u]^\iso & (\Sigma^{1,1} ko_{C_2})^{C_2} \ar[r]^{\eta^{C_2}} \ar[u]^\iso & (ko_{C_2})^{C_2} \ar[u]^\iso \\
(S^{0,0})^{C_2} \ar[r]^{\rho^{C_2}} \ar[u] & (S^{1,1})^{C_2} \ar[r]^{\eta^{C_2}} \ar[u] & (S^{0,0})^{C_2} \ar[u] \\
}\]
According to Lemma~\ref{EtaRhoSphere}, on the sphere $\eta\rho$ induces multiplication by $(C_2-2)$ under the isomorphism $\pi_{0,0} \iso A(C_2)$. 
The outer vertical compositions induce the linearization isomorphism $A(C_2) \iso RO(C_2)$ on $\pi_0$.
It follows that the top row induces multiplication by $(\sigma-1)$ on homotopy. We conclude that $\eta^{C_2}$ is $(-1,1)$.
\end{pf}

\begin{defn}\label{cMap} The complexification map $KO_{C_2}\xrtarr{c}K\R$
assigns to an equivariant real bundle $E\rtarr X$ the associated bundle
$\C\otimes_{\R} E\rtarr X$, where $C_2$ acts on $\C$ via complex conjugation.
We denote by $ko_{C_2}\xrtarr{c}k\R$ the associated map on connective covers.
\end{defn}

\begin{rmk}
Alternatively, from the point of view algebraic $\bK$-theory, the complexification map can be described as $\bK_{C_2}(\iota)$, where $\R\xrtarr{\iota}\C$ is the inclusion of $C_2$-equivariant topological rings.
\end{rmk}

\begin{prop}\label{EquivWoodProp} The Hopf map $\eta$ induces a cofiber sequence
\begin{equation} \Sigma^{1,1} ko_{C_2} \xrtarr{\eta} ko_{C_2} \xrtarr{c} k\R.\label{EquivWood}\end{equation}
\end{prop}

\begin{pf} 
On underlying spectra, this is the classical cofiber sequence 
\[\Sigma ko \xrtarr{\eta} ko \rtarr ku.\]
On fixed points, according to
Lemma~\ref{FxdPtsko} the sequence \eqref{EquivWood} induces a sequence 
\[ ko\xrtarr{\eta^{C_2}} ko\vee ko \xrtarr{c^{C_2}} ko.\]
By Lemma~\ref{EtaFixed}, the map $\eta^{C_2}$ is of the form $(-1,1)$.
For any real $C_2$-representation $V$, the construction $\C\otimes_\R V$ only depends on the dimension of $V$, which implies that $c^{C_2}$ is the fold map.
So the fixed points sequence is also a cofiber sequence.
\end{pf}

\begin{rmk} From the point of view of spectral Mackey functors \cites{GM,Ba}, the cofiber sequence \eqref{EquivWood} is the cofiber sequence of Mackey functors
\[
{\xymatrix{
 ko \ar@/_2ex/[d]_{0}  \ar[r]^-{(1,-1)}  & ko\vee ko \ar@/_2ex/[d]_{\nabla} \ar[r]^-\nabla & ko \ar@/_2ex/[d]_{c}\\
 \Sigma^{1,1} ko  \ar@/_2ex/[u]_{\eta}   \ar[r]^\eta \ar@(dl,dr)_{\mathrm{sign}} & ko \ar[r]^c \ar@/_2ex/[u]_{\Delta} \ar@(dl,dr)_{\mathrm{triv}}& ku \ar@/_2ex/[u]_{r} \ar@(dl,dr)_{\mathrm{conj}}
} }
\]
where $ku\xrtarr{r} ko$ considers a rank $n$ complex bundle as a rank $2n$ real bundle.
\end{rmk}

\begin{thm}\label{Hko} The $C_2$-equivariant cohomology of $ko_{C_2}$, 
as a module over $\cA^{C_2}$, is
\[ H^{*,*}_{C_2}(ko_{C_2}; \F_2) \iso \cA^{C_2}\!/\!/ \cA^{C_2}(1).\]
\end{thm}

\begin{pf}
According to \cite{Ricka}*{Corollary 6.19}, we have 
$H^{*,*}_{C_2}(k\R)\iso \cA^{C_2}\!/\!/ \cE^{C_2}(1)$. 
Since $\eta$ induces the trivial map on 
equivariant cohomology, the sequence (\ref{EquivWood}) induces a short exact sequence
\begin{equation} 
0 \rtarr H^{*-2,*-1}_{C_2}(ko_{C_2}) \xrtarr{i} 
  \cA^{C_2}\!/\!/\cE^{C_2}(1) \xrtarr{j} 
  H^{*,*}_{C_2}(ko_{C_2}) \rtarr 0 \label{CohomWood}
\end{equation}
of $A^{C_2}$-modules. 

The cofiber $C\eta$ is a 2-cell complex that supports a $\Sq^2$
in cohomology.
It follows that
the composition 
\[ 
k\R\simeq ko_{C_2}\smsh C(\eta) \rtarr \Sigma^{2,1} ko_{C_2} \into \Sigma^{2,1} ko_{C_2}\smsh C(\eta)
\]
induces the map
\[ \cA^{C_2}\!/\!/\cE^{C_2}(1) \xrtarr{ij}  \cA^{C_2}\!/\!/\cE^{C_2}(1): 
1 \mapsto \Sq^2.\]
In particular, the composition 
$\cA^{C_2} \rtarr \cA^{C_2}\!/\!/\cE^{C_2}(1) \xrtarr{j} 
H^{*,*}_{C_2}(ko_{C_2})$ 
factors through $\cA^{C_2}\!/\!/\cA^{C_2}(1)$.
Given the right $\cE^{C_2}(1)$-module decomposition 
$\cA^{C_2}(1) \iso \cE^{C_2}(1) \oplus \Sigma^{2,1}\cE^{C_2}(1)$, 
it follows that the sequence (\ref{CohomWood}) sits in a diagram
\[ 
\xymatrix{
0 \ar[r] & H^{*-2,*-1}_{C_2}(ko_{C_2}) \ar[r] &  
  \cA^{C_2}\!/\!/\cE^{C_2}(1) \ar[r] \ar@{=}[d] &   
  H^{*,*}_{C_2}(ko_{C_2}) \ar[r] & 0
\\
0 \ar[r] & \Sigma^{2,1} \cA^{C_2}\!/\!/ \cA^{C_2}(1) 
  \ar[r] \ar@{^{(}->}[u] &  \cA^{C_2}\!/\!/\cE^{C_2}(1) \ar[r]  &
  \cA^{C_2}\!/\!/\cA^{C_2}(1) \ar@{->>}[u]  \ar[r] & 0.
}
\]
The outer two maps agree up to suspension, 
so they are both isomorphisms.
\end{pf}

\begin{cor} The $E_2$-page of the Adams spectral sequence for 
$ko_{C_2}$ is 
\[ 
E_2 \iso \Ext_{\cA^{C_2}}(H^{*,*}_{C_2}(ko_{C_2}),\bM_2^{C_2}) \iso 
\Ext_{C_2}(1).
\]
\end{cor}

\begin{pf} 
This is a standard change of rings isomorphism
\cite[Theorem~A1.3.12]{Rav}, using that $H^{*,*}_{C_2}(ko_{C_2})$ is
isomorphic to $\cA^{C_2} \!/\!/ \cA^{C_2}(1)$.  Note that the
change of rings theorem applies by \cite[Corollary 6.15]{Ricka}.
\end{pf}

\begin{rmk} \label{TateApproach}
Working in the $2$-complete category,
it is also possible to build $ko_{C_2}$ using the ``Tate diagram''
approach. See, for example, \cite{G2} for a nice description of this
approach.  According to this approach, one specifies a $C_2$-spectrum
$X$ by giving three pieces of data:
\begin{enumerate}
\item an underlying spectrum $X^e$ with $C_2$-action, 
\item a geometric fixed points spectrum 
$X^{g{C_2}}$, and 
\item a map $X^{g{C_2}}\rtarr (X^e)^{tC_2}$
from the geometric fixed points to the Tate construction.
\end{enumerate} 

In our case, 
the underlying spectrum is
$ko$ with trivial $C_2$-action.
The rest of the Tate diagram information is given by the following result. 
\end{rmk}

\begin{prop}\label{koTateDiag}
The geometric fixed points of $ko_{C_2}$ is $\bigvee_{k\geq 0} \Sigma^{4k}H\hat{\Z}_2$,
and the map $(ko_{C_2})^{g{C_2}} \rtarr ko^{tC_2}$ is the connective cover.
\end{prop}

\begin{pf}
The Tate construction $ko^{tC_2}$ was computed by Davis-Mahowald to be 
$\bigvee_{n\in \Z} \Sigma^{4n}H\hat\Z_2$ \cite{DM}*{Theorem 1.4}.
For the interpretation of the Davis-Mahowald calculation in terms of the Tate construction, see \cite{Alaska}*{Section XXI.3}.

The geometric fixed points sit in a cofiber sequence
\[ ko \smsh \R\bP^\infty_+ \simeq ko_{hC_2} \rtarr (ko_{C_2})^{C_2} \rtarr (ko_{C_2})^{gC_2},\]
which we can write as
\[ko \vee (ko \smsh \R\bP^\infty) \rtarr ko\vee ko \rtarr (ko_{C_2})^{gC_2}.\]
The left map is a map of $ko$-modules, and we consider the simpler cofiber sequence
\[ ko \smsh \R\bP^\infty \xrtarr{ko\smsh t}  ko \rtarr (ko_{C_2})^{gC_2},\]
where $t:\R\bP^\infty \rtarr S^0$ is the Kahn-Priddy transfer. As in \cite[Section~1.5]{Rav}, we write $R$ for the cofiber of $t$, so that $(ko_{C_2})^{gC_2}\simeq ko\smsh R$. As Adams explained in \cite{NthKind}, the cohomology of $R$ has a filtration as $\cA^\cla(1)$-modules in which the associated graded object is $\bigoplus_{k\geq 0} \Sigma^{4k} \cA^\cla(1)/\!/ \cA^\cla(0)$. It follows that $ko\smsh R\simeq \bigvee_{k\geq 0} \Sigma^{4k} H\hat{\Z}_2$.

Similarly, the associated graded for $\colim_n \rH^*(\Sigma \R\bP^\infty_{-n})$ is $\bigoplus_{k\in \Z} \Sigma^{4k} \cA^\cla(1)/\!/ \cA^\cla(0)$. The map $R \rtarr \holim_n \Sigma \R\bP^\infty_{-n}$ is surjective on cohomlogy, and the same is true for the induced map $R\smsh ko \rtarr  \holim_n  (\R\bP^\infty_{-n}\smsh \Sigma ko)$. We conclude that the map 
\[ \bigvee_{k\geq 0} \Sigma^{4k}H\hat{\Z}_2 \simeq (ko_{C_2})^{gC_2} \rtarr ko^{tC_2} \simeq \holim_n  (\R\bP^\infty_{-n}\smsh \Sigma ko)\]
is a split inclusion in homotopy and therefore a connective cover.
\end{pf}

\begin{rmk}
Note that the description of geometric fixed points given here is confirmed by 
Proposition~\ref{ExtA1RRho}. That is, the geometric fixed points of a $C_2$-spectrum
$X$ are given by the categorical fixed points of $S^{\infty, \infty}\smsh X$, where 
$S^{\infty,\infty} =\colim(S^{n,n}\xrtarr{\rho}S^{n+1,n+1})$. Thus the geometric
fixed points are computed by the $\rho$-inverted Adams spectral sequence. 
As we recall in the next section, the homotopy element $2$ is detected by
the element $h_0+\rho h_1$ in $\Ext$. Thus the element
$\rho^k h_1^k \tau^{4j}$ of Proposition~\ref{ExtA1RRho} detects $2^k$ in
the $4j$-stem of the geometric fixed points.
\end{rmk}

\section{The homotopy ring}
\label{sec:homotopy}

In this section, we will describe the bigraded homotopy ring 
$\pi_{*,*}(ko_{C_2})$ of $ko_{C_2}$.
We are implicitly completing the homotopy groups at $2$ so that
the Adams spectral sequence converges \cite[Corollary 6.47]{HK}.

It turns out that the Adams spectral
sequence collapses, so that $\Ext_{C_2}(1)$ 
is an associated graded object of 
$\pi_{*,*}(ko_{C_2})$.  Nevertheless, the Adams spectral sequence
hides much of the multiplicative structure.  

Recall that the Milnor-Witt stem of $X$ is defined (cf. \cite{MWLow}) as the direct sum
\[ \Pi_n(X) \iso \bigoplus_i \pi_{n+i,i}(X).\]

\begin{prop} There are no non-zero differentials in the
Adams spectral sequence for $ko_{C_2}$.
\end{prop}

\begin{pf}
This follows by inspection of the $E_2$-page,
shown in the charts in Section \ref{sctn:chart}.

Adams $d_r$ differentials decrease the stem by 1, increase the filtration by $r$, and preserve the weight. It follows that
 Adams differentials decrease the Milnor-Witt stem by 1. Every class in Milnor-Witt stem congruent to $3$ modulo $4$ is infinitely $\rho$-divisible. As there are no infinitely $\rho$-divisble classes in Milnor-Witt stem congruent to $2$ modulo $4$, it follows that there are no non-zero differentials supported in the Milnor-Witt $(4k+3)$-stem. 

Every class in Milnor-Witt stem $4k$ supports an infinite tower of either $h_0$-multiples or $h_1$-multiples, while there are no such towers in Milnor-Witt stem $4k+1$.  It follows that there cannot be any non-zero differentials emanating from the $(4k+1)$-Milnor-Witt-stem. 
Finally, direct inspection shows there cannot be any non-zero
differentials starting in the Milnor-Witt $(4k+2)$ or $4k$-stems. 
\end{pf}

The structure of the Milnor-Witt $n$-stem $\Pi_n(ko_{C_2})$
of course depends on $n$.  The description of these Milnor-Witt
stems naturally breaks into cases, depending on the
value of $n \pmod 4$.

Table \ref{tbl:HtpyGen} summarizes the notation that we will
use for specific elements of $\pi_{*,*}(ko_{C_2})$.
The definition of each element is discussed in detail
in the following sections.

\begin{table}
\captionof{table}{Notation for $\pi_{*,*}(ko_{C_2})$}
\label{tbl:HtpyGen}
\begin{center}
\begin{tabular}{LLLLL} 
\hline
mw & (s,w) & \text{element} & \text{detected by} & \text{defining relation} \\ \hline
0 &  (-1,-1) & \rho &  \rho & \\
0 &  (1,1) & \eta & h_1 & \\
0 & (4,4) & \alpha & Q h_1^3 & \rho\alpha = \eta^3\\
0 & (0,0) & \omega & h_0 & \omega = \eta \rho + 2 \\
4 & (0,-4) & \tau^4 &   \tau^4 &  \\
0 & (8,8) & \beta & \frac{Qh_1^4}{\rho^3} & 4\beta = \alpha^2 \\
2 & (0,-2) & \tau^2\omega & \tau^2 h_0 
  & (\tau^2\omega)^2 = 2 \omega \cdot \tau^4 \\
-2 & (0,2) & \tau^{-2}\omega & \frac{\gamma}{\tau}
  & \tau^4 \cdot \tau^{-2}\omega = \tau^2 \omega \\
-4 & (0,4) & \tau^{-4}\omega & \frac{\gamma}{\tau^3} 
  & \tau^4 \cdot \tau^{-4}\omega = \omega \\
-5-4k & (0,5+4k) & \frac{\Gamma}{\tau^{4+4k}} & \frac{\gamma}{\tau^{4+4k}}
  & \tau^4 \cdot \frac{\Gamma}{\tau^{4+4k}} = \frac{\Gamma}{\tau^{4+4(k-1)}} \\
1 & (1,0) & \tau\eta & \tau h_1 & \\
2 & (4,2) & \tau^2 \alpha & a & 2 \tau^2 \alpha = \alpha\cdot \tau^2\omega\\
\hline
\end{tabular}
\end{center}
\end{table}

\subsection{The Milnor-Witt $0$-stem}

Our first task is to describe the Milnor-Witt $0$-stem
$\Pi_0(ko_{C_2})$.
The other Milnor-Witt stems are modules over
$\Pi_0(ko_{C_2})$, and we will use this module structure
heavily in order to understand them.

\begin{prop}
\label{prop:underlying-rho}
Let $X$ be a $C_2$-equivariant spectrum, and 
let $\alpha$ belong to $\pi_{n,k}(X)$.
The element $\alpha$ is divisible by $\rho$ if and only if
its underlying class $\iota^*(\alpha)$ in $\pi_n (\iota^* X)$
is zero.
\end{prop}

\begin{pf}
The $C_2$-equivariant cofiber sequence
\[
C_2\,_+ \rtarr S^{0,0} \xrtarr{\rho} S^{1,1}
\]
induces a long exact sequence
\[
\cdots \rtarr \pi_{n+1,k+1}(X) \xrtarr{\rho} 
\pi_{n,k}(X) \xrtarr{\iota^*} \pi_n(\iota^* X) \rtarr 
\pi_{n+2,k+1}(X) \xrtarr{\rho} \cdots.
\]
\end{pf}

\begin{cor}
\label{prop:rho-Qh1^3}
There is a hidden $\rho$ extension from ${Q} h_1^3$ to $h_1^3$
in the Adams spectral sequence.
\end{cor}

\begin{pf}
Recall that $\eta^3$ is zero in $\pi_3(ko)$.
Proposition \ref{prop:underlying-rho}
implies that $\eta^3$ in $\pi_{3,3}(ko_{C_2})$ is divisible by $\rho$.
The only possibility is that there is a hidden extension from
${Q} h_1^3$ to $h_1^3$.
\end{pf} 

\begin{prop} The element $\eta$ in $\pi_{1,1}(ko_{C_2})$ is detected by $h_1$.
\end{prop}

\begin{pf}
The restriction $\iota^*(\eta)$ of $\eta$ is the classical $\eta$, 
which is detected by the classical element $h_1$.
As all other elements of $\Ext_{\cA^{C_2}(1)}$ in the 1-stem and weight 1 all live in higher filtration, the result follows.
\end{pf}

\begin{defn}
Let $\alpha$ be an element in $\pi_{4,4} (ko_{C_2})$ 
detected by ${Q} h_1^3$
such that
$\rho \alpha = \eta^3$.
\end{defn}

Corollary \ref{prop:rho-Qh1^3} guarantees
that such an element $\alpha$ exists.

There are many elements of $\pi_{4,4}$ 
detected by ${Q} h_1^3$ because of the presence of
elements in higher Adams filtration.  The condition
$\rho \alpha = \eta^3$ narrows the possibilities, but still
does not determine a unique element because of
the elements $\frac{\gamma}{\tau} h_0^k a$ in higher Adams
filtration.  For our purposes, 
this remaining choice makes no difference.

\begin{defn}
Let $\omega$ be the element 
$\eta\rho + 2$
of $\pi^{C_2}_{0,0}(ko_{C_2})$.
\end{defn}

As for $\rho$ and $\eta$,
the element $\omega$ comes from the homotopy groups of the
equivariant sphere spectrum.  Strictly speaking, there is no
need for the notation $\omega$ since it can be written 
in terms of other elements.  Nevertheless, it is convenient
because $\omega$ plays a central role.
According to Lemma \ref{EtaRhoSphere}, 
$\omega$ corresponds to the element $C_2$ of the Burnside
ring $A(C_2)$.

Note that $\omega$ is detected by $h_0$, while
$2$ is detected by $h_0 + \rho h_1$.
For this reason, $\omega$, rather than $2$, plays
the role of the zeroth Hopf map in the equivariant
(and $\R$-motivic) context.
Also note that $\omega$ equals $1 - \epsilon$,
where $\epsilon$ is the twist
\[
S^{1,1} \smsh S^{1,1} \rtarr S^{1,1} \smsh S^{1,1}.
\]

\begin{prop} 
\label{prop:2-eta^5}
The homotopy class $\eta^5$ is divisible by $2$.
\end{prop}

\begin{pf} 
The relation $\omega \eta=0$ was established by Morel \cite{Mo} in the $\R$-motivic stable stems, and the equivariant stems agree with the $\R$-motivic ones in the relevant degrees \cite{C2RStems}*{Theorem 4.1}.
(See also \cite{DIHopf} for a geometric argument for this relation
given in the motivic context.  This geometric argument 
works just as well equivariantly.)

Using the defining relation for $\alpha$,
it follows that
\[
-2 \eta \alpha = \rho \eta^2 \alpha = \eta^5.
\] 
\end{pf}

Proposition \ref{prop:2-eta^5} 
was already known to be true in the homotopy of the 
$C_2$-equivariant sphere spectrum \cite{B}.
The divisibility of the elements $\eta^k$ is very much related
to work of Landweber \cite{Landweber}.

\begin{defn} 
Let $\tau^4$ be an element of $\pi_{0,-4}(ko_{C_2})$ 
that is detected by $\tau^4$.
\end{defn}

The element $\tau^4$ is not uniquely determined because of elements
in higher Adams filtration.  
For our purposes, we may choose an arbitrary such element.

\begin{prop}\label{prop:BasicPhiMult}
\mbox{}
\begin{enumerate}
\item 
\label{part:tau^4-Qh1^3}
There is a hidden $\tau^4$ extension from ${Q} h_1^3$ to $\tau^2 a$.
\item 
\label{part:tau^4-Q/rho^3 h1^4}
There is a hidden $\tau^4$ extension from $\frac{Q}{\rho^3} h_1^4$
to $b$.
\end{enumerate} 
\end{prop}

\begin{pf}
\mbox{}
\begin{enumerate}
\item
The product $\rho \alpha \cdot \tau^4$ equals $\tau^4 \cdot \eta^3$,
which is detected by $\tau^4 \cdot h_1^3$.
This last expression equals $\rho \cdot \tau^2 a$ in $\Ext$.
\item
Part (1) implies that there is a hidden
$\tau^4$ extension from
${Q} h_1^4$ to $\rho^3 b$,
since $h_1 \cdot \tau^2 a$ equals $\rho^3 b$ in $\Ext$.
This means that there is a hidden $\tau^4$ extension
from $\frac{Q}{\rho^3} h_1^4$ to $b$,
since $\rho^3 \cdot \frac{Q}{\rho^3} h_1^4$ 
equals ${Q} h_1^4$ in $\Ext$.
\end{enumerate}
\end{pf}

\begin{lemma} 
\label{lem:alpha^2-div-4}
The class $\alpha^2$ in $\pi_{8,8}(ko_{C_2})$ 
is divisible by $4$.
\end{lemma}

\begin{pf} 
By Proposition~\ref{prop:BasicPhiMult}, the multiplication map
\[ 
\tau^4: \pi_{8,8}(ko_{C_2}) \xrightarrow{\iso}\pi_{8,4}(ko_{C_2})
\]
is an isomorphism. By considering the effect of multiplication by $\tau^4$ in $\Ext$, we see that
\[ 
\tau^4:\pi_{8,4}(ko_{C_2}) \xrightarrow{\iso}\pi_{8,0}(ko_{C_2})
\]
is also an isomorphism.
Thus it suffices to show that $(\tau^4)^2 \alpha^2$ is $4$-divisible
in $\pi_{8,0}(ko_{C_2})$.
But
$(\tau^4)^2 \cdot \alpha^2$
is detected by $(\tau^2 a)^2$ by 
Proposition \ref{prop:BasicPhiMult} (\ref{part:tau^4-Qh1^3}),
which equals $(h_0 + \rho h_1)^2 \tau^4 b$ in $\Ext$.
Finally, observe that $h_0 + \rho h_1$ detects $2$.
\end{pf}

\begin{defn}
Let $\beta$ be the element of $\pi_{8,8}(ko_{C_2})$ 
detected by $\frac{Q}{\rho^3} h_1^4$ and satisfying
$4\beta = \alpha^2$.
\end{defn}

Note that $\beta$ is uniquely determined by $\alpha$,
even though there are elements of higher Adams filtration,
because there is no $2$-torsion in $\pi_{8,8}(ko_{C_2})$.

\begin{prop} 
\label{prop:rho^3beta}
$\rho^3\beta = \eta \alpha$.
\end{prop}

\begin{pf} 
The defining relation for $\beta$ implies that 
$4 \rho^3 \beta$ equals $\rho^3 \alpha^2$, which equals
$\rho^2 \eta^3 \alpha$ by the defining relation for $\alpha$.
Using the relation
$(\eta\rho + 2) \eta=0$, this element equals
$4 \eta \alpha$.
Finally, there is no $2$-torsion in $\pi_{5,5}(ko_{C_2})$.
\end{pf}

\begin{prop} The (2-completed) Milnor-Witt $0$-stem of $ko_{C_2}$ is
\[ \Pi_0(ko_{C_2}) \iso 
\Z_2[\eta,\rho,\alpha,\beta]/
(\rho (\eta \rho + 2),\eta (\eta \rho + 2),
\rho \alpha - \eta^3, \rho^3 \beta-\eta \alpha, \alpha^2-4 \beta),\]
where the generators have degrees
$(1,1)$, $(-1,-1)$, $(4,4)$, and $(8,8)$ respectively.
These homotopy classes are detected by $h_1$, $\rho$, 
${Q} h_1^3$, and 
$\frac{Q h_1^4}{\rho^3}$ in the Adams spectral sequence.
\end{prop}

\begin{pf} 
The relations $\rho (\eta \rho + 2)$ and
$\eta (\eta \rho + 2)$ are already true in the sphere
\cite{Mo} \cite{DIHopf}.
The third and fifth relations are part of the definitions of 
$\alpha$ and $\beta$, while the fourth relation is
Proposition \ref{prop:rho^3beta}.

It remains to show that $\beta^k$ is detected by 
$\frac{Q}{\rho^{4k-1}} h_1^{4k}$ and that $\alpha \beta^k$ 
is detected by $\frac{Q}{\rho^{4k-1}}h_1^{4k+4}$. 

We assume for induction on $k$ that $\beta^k$
is detected by $\frac{Q}{\rho^{4k-1}} h_1^{4k}$.
We have the relation
$h_0 \cdot \frac{Q}{\rho^{4k-1}} h_1^{4k} = 
\frac{\gamma}{\tau^{4k-1}} b^k$ in $\Ext$,
so $\omega \beta^k$ is detected by
$\frac{\gamma}{\tau^{4k-1}} b^k$ in $\Ext$.
Now $b$ detects $\tau^4 \cdot \beta$ 
by Proposition \ref{prop:BasicPhiMult} (\ref{part:tau^4-Q/rho^3 h1^4}),
so $\omega \beta^{k+1}$ is detected by
$\frac{\gamma}{\tau^{4k-1}} b^{k+1}$.
Finally,
$\frac{\gamma}{\tau^{4k-1}} b^{k+1}$ equals
$\tau^4 \cdot \frac{\gamma}{\tau^{4k+3}} b^{k+1}$ in $\Ext$,
which equals
$\tau^4 \cdot h_0 \cdot \frac{Q}{\rho^{4k+3}} h_1^{4k+4}$.

We have now shown that 
$\tau^4 \cdot h_0 \cdot \frac{Q}{\rho^{4k+3}} h_1^{4k+4}$
detects $\tau^4 \cdot \omega \beta^{k+1}$.
It follows that 
$\frac{Q}{\rho^{4k+3}} h_1^{4k+4}$ detects $\beta^{k+1}$.

A similar argument handles the case of $\alpha\beta^k$.
\end{pf}

\subsection{$\tau^4$-periodicity}\label{sec:t4Per}

Before analyzing the other Milnor-Witt stems of $ko_{C_2}$,
we will explore a piece of the global structure involving
the element $\tau^4$ of $\pi_{0,-4}(ko_{C_2})$.

\begin{prop}\label{prop:PhiMult}
There are hidden $\tau^4$ extensions:
\begin{enumerate}
\item from $\frac{\gamma}{\tau}$ to $\tau^2 h_0$.
\item from $\frac{\gamma}{\rho^2\tau} h_1^2$ to $a$.
\item from $\frac{\gamma}{\tau^3}$ to $h_0$.
\item from $\frac{\gamma}{\rho\tau^2}$ to $\tau h_1$.
\end{enumerate}
\end{prop}

\begin{pf}
\mbox{}
\begin{enumerate}
\item\label{PhiTheta1} 
Recall that 
$\frac{\gamma}{\tau} \cdot a$ equals
$h_0 \cdot {Q} h_1^3$ in $\Ext$, so the
hidden $\tau^4$ extension on ${Q} h_1^3$ from 
Proposition \ref{prop:BasicPhiMult} (\ref{part:tau^4-Qh1^3})
implies that there 
is a hidden $\tau^4$ extension
from 
$\frac{\gamma}{\tau} \cdot a$ to $\tau^2 h_0 a$.
It follows that there is a hidden
$\tau^4$ extension from $\frac{\gamma}{\tau}$ to $\tau^2 h_0$. 
\item
\label{PhiEta2Theta1} 
Using that $h_1^2 \cdot \tau^2 h_0$ equals $\rho^2 a$ in $\Ext$,
part (1) implies that there is a hidden
$\tau^4$ extension from
$\frac{\gamma}{\tau} h_1^2$ to $\rho^2 a$.
\item
\label{PhiTheta3} 
Recall that 
$\frac{\gamma}{\tau^3} \cdot b$ equals 
$h_0 \cdot \frac{Q}{\rho^3} h_1^4$ in $\Ext$,
so the hidden $\tau^4$ extension on $\frac{Q}{\rho^3} h_1^4$
from Proposition \ref{prop:BasicPhiMult} (\ref{part:tau^4-Q/rho^3 h1^4})
implies that there
is a hidden $\tau^4$ extension from
$\frac{\gamma}{\tau^3} \cdot b$ to $h_0 b$.
It follows that there is a hidden
$\tau^4$ extension from $\frac{\gamma}{\tau^3}$
to $h_0$.
\item\label{PhiTheta2} 
Using that $\rho a$ equals $h_1 (\tau h_1)^2$ in $\Ext$,
part (2) implies that there is a hidden $\tau^4$ extension from
$\frac{\gamma}{\rho \tau} h_1^2$ to $h_1 (\tau h_1)^2$.
Now $\frac{\gamma}{\rho \tau} h_1^2$ equals
$\frac{\gamma}{\rho \tau^2} h_1 \cdot \tau h_1$, so
there is also a hidden $\tau^4$ extension on
$\frac{\gamma}{\rho \tau^2}$.
\end{enumerate}
\end{pf}

The homotopy of $ko_{C_2}$ is nearly $\tau^4$-periodic, 
in the following sense.

\begin{thm}\label{thm:PhiPerd}
Multiplication by $\tau^4$ gives a homomorphism on Milnor-Witt stems
\[\Pi_n(ko_{C_2})\rtarr \Pi_{n+4}(ko_{C_2})\]
 which is
\begin{enumerate}
\item injective if $n=-4$.
\item surjective (and zero) if $n=-5$.
\item bijective in all other cases.
\end{enumerate}
\end{thm}

\begin{pf}
\mbox{}
\begin{enumerate}
\item This is already true in $\Ext$, except in the $0$-stem. But the $0$-stem is handled by Proposition~\ref{prop:PhiMult}(\ref{PhiTheta3}).
\item There is nothing to prove here, given that $\Pi_{-1}(ko_{C_2})=0$.
\item We give arguments depending on the residue of $n$ modulo 4. 
\begin{itemize}
\item 
$n\equiv 0 \pmod{4}$: If $n < -4$, this is already true in $\Ext$. For $n\geq 0$, this follows from the relation $\rho\alpha=\eta^3$
and the hidden $\tau^4$ extensions on $\alpha$ and $\beta$ given 
in Proposition \ref{prop:BasicPhiMult}.

\item 
$n\equiv 1\pmod{4}$: For $n<-3$, this is already true in $\Ext$. For $n\geq -3$, this follows from Proposition~\ref{prop:PhiMult}(\ref{PhiTheta2}).

\item 
$n\equiv 2\pmod{4}$: For $n<-2$, this is already true in $\Ext$. For $n\geq -2$, this follows from Proposition~\ref{prop:PhiMult}(\ref{PhiTheta1}) and (\ref{PhiEta2Theta1}).

\item 
$n\equiv 3\pmod{4}$: This is already true in $\Ext$. 
\end{itemize}
\end{enumerate}
\end{pf}

\begin{rmk}
Another way to view the $\tau^4$-periodicity is via the Tate diagram (Proposition~\ref{koTateDiag}). We have a cofiber sequence
\[ {EC_2}_+ \smsh ko \rtarr ko_{C_2} \rtarr S^{\infty,\infty}\smsh ko_{C_2}.\]
The homotopy orbit spectrum therefore captures the $\rho$-torsion. If $x\in \pi_{*,*}ko_{C_2}$ is $\rho$-torsion, then so is $\tau^4\cdot x$. But multiplication by $\tau^4$ is an equivalence on underlying spectra and therefore gives an equivalence on homotopy orbits. This implies the $\tau^4$-periodicity in the $\rho$-torsion.
\end{rmk}

\subsection{The Milnor-Witt $n$-stem with $n \equiv 0 \pmod 4$}

Theorem \ref{thm:PhiPerd} indicates that $\tau^4$ multiplications
are useful in describing the structure of the homotopy groups
of $ko_{C_2}$.  Therefore, 
our next task is to build on our understanding of
$\Pi_0 (ko_{C_2})$ and to describe the subring
$\bigoplus_{k \in \Z} \Pi_{4k} (ko_{C_2})$
of $\pi_{*,*} ko_{C_2}$. 

The $\Ext$ charts indicate that the behavior of these groups
differs for $k \geq 0$ and for $k < 0$.

\begin{prop}
$\bigoplus_{k \geq 0} \Pi_{4k}(ko_{C_2})$ is isomorphic to
$\Pi_0(ko_{C_2})[\tau^4]$.
\end{prop}

\begin{proof}
This follows immediately from Theorem \ref{thm:PhiPerd}.
\end{proof}

\begin{defn} 
\label{defn:t^2omega}
Define 
$\tau^2\omega$ to be an element in $\pi_{0,-2}(ko_{C_2})$ 
that is detected by $\tau^2 h_0$ such that
$(\tau^2 \omega)^2 = 2 \omega \cdot \tau^4$.
\end{defn}

An equivalent way to specify a choice of $\tau^2 \omega$ is to require
that the underlying map $\iota^*(\tau^2 \omega)$ equals $2$ in
$\pi_0(ko)$.

\begin{defn}
For $k \geq 1$,  
let $\frac{\Gamma}{\tau^k}$ be an element of $\pi_{0,k+1}$ detected by
$\frac{\gamma}{\tau^k}$ such that:
\begin{enumerate}
\item
$\tau^4 \cdot \frac{\Gamma}{\tau} = \tau^2 \omega$.
\item
$\tau^4 \cdot \frac{\Gamma}{\tau^3} = \omega$.
\item
$\tau^4 \cdot \frac{\Gamma}{\tau^k} = \frac{\Gamma}{\tau^{k-4}}$
when $k \geq 5$.
\end{enumerate}
\end{defn}

According to Theorem~\ref{thm:PhiPerd}, the elements $\frac{\Gamma}{\tau^k}$ are uniquely determined by
the stated conditions.
Proposition \ref{prop:PhiMult} (\ref{PhiTheta1}) and (\ref{PhiTheta3}) allow
us to choose $\frac{\Gamma}{\tau}$ and 
$\frac{\Gamma}{\tau^3}$ with the desired properties.
As suggested by the defining relations for these elements, we will often write $\tau^{-2-4k}\omega$ for $\frac{\Gamma}{\tau^{1+4k}}$ and $\tau^{-4-4k}\omega$ for $\frac{\Gamma}{\tau^{3+4k}}$.

\begin{prop}
As a $\pi_0(ko_{C_2})[\tau^4]$-module, 
$\bigoplus_{k \in \Z} \Pi_{4k}(ko_{C_2})$ is isomorphic to the
$\pi_0(ko_{C_2})[\tau^4]$-module generated by $1$ 
and the elements  $\tau^{-4-4k}\omega$ for $k\geq 0$,
subject to the relations
\begin{enumerate}
\item
$\tau^4\cdot \tau^{-4-4k}\omega = \tau^{-4k}\omega$.
\item
$\rho \cdot \tau^{-4-4k}\omega = 0$.
\item
$\eta \cdot \tau^{-4-4k}\omega = 0$.
\item
$\tau^4 \cdot \tau^{-4}\omega = \omega$.
\end{enumerate} 
\end{prop}

\begin{proof}
This follows by inspection of the $\Ext$ charts,
together with the defining relations
for 
$\tau^{-4-4k}\omega$.
\end{proof}

\subsection{The Milnor-Witt $n$-stem with $n \equiv 1 \pmod 4$}

\begin{defn} Denote by $\tau\eta$ an element of $\pi_{1,0}(ko_{C_2})$ 
that is detected by $\tau h_1$.
\end{defn}

Note that $\tau \eta$ is not uniquely determined because of elements
in higher Adams filtration, but the choice makes no practical
difference.  One way to specify a choice of $\tau \eta$ is to 
use the composition
\[
S^{1,0} \rtarr S^{0,0} \rtarr ko_{C_2},
\]
where the first map is the image of the classical Hopf map
$\eta: S^1 \rtarr S^0$, and the second map is the unit.

\begin{prop} 
As a $\Pi_0(ko_{C_2})[\tau^4]$-module, 
there is an isomorphism
\[ 
\bigoplus_{k\in \Z} \Pi_{1+4k}(ko_{C_2}) \iso 
\left(\Pi_0(ko_{C_2})[(\tau^4)^{\pm 1}]/
(2,\rho^2,\eta^2,\alpha)\right)\{\tau \eta\}.\]
\end{prop}

\begin{pf}
This follows from inspection of the 
$\Ext$ charts, together with Theorem \ref{thm:PhiPerd}.
\end{pf}

\subsection{The Milnor-Witt $n$-stem with $n \equiv 2 \pmod 4$}

Recall from Definition \ref{defn:t^2omega}
that $\tau^2 \omega$ is an element of $\pi_{0,-2}(ko_{C_2})$
that is detected by $\tau^2 h_0$.

\begin{lemma} The product $\alpha \cdot \tau^2\omega$ in 
$\pi_{4,2}(ko_{C_2})$ is detected by $h_0 a$.
\end{lemma}

\begin{pf} 
The product $\tau^4 \cdot \alpha \cdot \tau^2 \omega$
is detected by $\tau^4 h_0 a$ by 
Proposition \ref{prop:BasicPhiMult} (\ref{part:tau^4-Qh1^3}).
\end{pf}

\begin{defn}
Define $\tau^2\alpha$ to be an element of $\pi_{4,2}(ko_{C_2})$ 
that is detected by $a$ such that 
$2 \cdot \tau^2\alpha$ equals $\alpha\cdot \tau^2\omega$.
\end{defn}

\begin{prop} 
As a $\Pi_0(ko_{C_2})[\tau^4]$-module, 
$\bigoplus_{k\in \Z} \Pi_{2+4k}(ko_{C_2})$
is isomorphic to the free
$\Pi_0(ko_{C_2})[(\tau^4)^{\pm 1}]$-module generated by
$\tau^2\omega$, 
$(\tau\eta)^2$, and 
$\tau^2\alpha$, subject to the relations
\begin{enumerate}
\item $\rho\cdot \tau^2\omega = 0$.
\item $\alpha \cdot \tau^2\omega = 2\cdot \tau^2\alpha$.
\item $\rho (\tau\eta)^2 = \eta\cdot\tau^2\omega$.
\item $2 (\tau \eta)^2 = 0$.
\item $\eta (\tau \eta)^2 = \rho \cdot \tau^2\alpha$.
\item $ \alpha (\tau\eta)^2 = 0$.
\item $\eta \cdot \tau^2\alpha = 0$.
\item\label{LastReln} $\alpha \cdot \tau^2\alpha = 2 \beta \cdot \tau^2\omega$.
\end{enumerate}
\end{prop}

\begin{pf}
Except for the last relation,
this follows from inspection of the $\Ext$ charts,
together with Theorem \ref{thm:PhiPerd}.

For the last relation, use that
$2 \alpha \cdot \tau^2 \alpha$ equals
$\tau^2 \omega \cdot \alpha^2$ by the definition
of $\tau^2 \alpha$,
and that $\tau^2 \omega \cdot \alpha^2$ 
equals $4 \beta \cdot \tau^2 \omega$
by the defining relation for $\beta$.
As there is no $2$-torsion in this degree, 
relation~(\ref{LastReln}) follows.
\end{pf}

\subsection{The Milnor-Witt $n$-stem with $n \equiv 3 \pmod 4$}

The structure of
$\bigoplus_{k\in \Z} \Pi_{4k+3}(ko_{C_2})$
is qualitatively different than the other cases because
it contains elements that are infinitely divisible by $\rho$.
The $\Ext$ charts show that 
$\bigoplus_{k \in \Z} \Pi_{4k+3}(ko_{C_2})$
is concentrated in the range $k \leq -2$.

The elements $\frac{\Gamma}{\tau^{4k}}$ 
are infinitely divisible by both $\rho$ and $\tau^4$.
We write $\frac{\Gamma}{\rho^{j}\tau^{4k} }$ for an element such that
$\rho^j \cdot \frac{\Gamma}{\rho^{j}\tau^{4k}}$ equals
$\frac{\Gamma}{\tau^{4k}}$.  

By inspection of the $\Ext$ charts, we see that 
$\bigoplus_{k\leq 0} \Pi_{4k-5}(ko_{C_2})$ is generated 
as an abelian group by the elements $\frac{\Gamma}{\rho^{j}\tau^{4k} }$.
The $\Pi_0(ko_{C_2})[\tau^4]$-module structure on 
$\bigoplus_{k\leq 0} \Pi_{4k-5}(ko_{C_2})$ 
is then governed by the 
orders of these elements, together with the relations 
\[
\alpha \cdot \frac{\Gamma}{\tau^{4k}} =
-8 \frac{\Gamma}{\rho^4 \tau^{4k}}
\]
and
\[
\beta\cdot \frac{\Gamma}{\tau^{4k}} = 16\frac{\Gamma}{\rho^8 \tau^{4k} }.
\]
The first relation follows from the calculation
\[
\alpha \cdot \frac{\Gamma}{\tau^{4k}} = 
\rho\alpha \cdot \frac{\Gamma}{\rho\tau^{4k}} = 
\eta^3 \cdot \frac{\Gamma}{\rho\tau^{4k}} =
(\eta \rho)^3 \cdot \frac{\Gamma}{\rho^4\tau^{4k}} = (-2)^3 \cdot \frac{\Gamma}{\rho^4\tau^{4k}} = 
-8 \frac{\Gamma}{\rho^4 \tau^{4k} }.
\] 
The second relation follows from a similar argument, 
using that $\rho^3 \beta = \eta \alpha$.

\begin{prop} 
\label{prop:theta-torsion}
The order of $\frac{\Gamma}{\tau^{4k} \rho^{j}}$ is $2^{\varphi(j)+1}$, where $\varphi(j)$ is the number of positive integers $0<i\leq j$ such that $i\equiv 0,1,2$ or $4 \pmod8$.
\end{prop}

\begin{pf} 
Since $h_0+\rho h_1$ detects the element $2$, the result is represented by
the chart on page~\pageref{E-dummy}, in stems zero to sixteen. As the top edge of the 
region is $(8,4)$-periodic, this gives the result in higher stems as well.
\end{pf}

\begin{rmk}
Proposition \ref{prop:theta-torsion} is an
independent verification of a well-known calculation.
We follow the argument given in \cite{D}*{Appendix~B}. 

Let $\R^{q,q}$ be the antipodal $C_2$-representation on $\R^q$.
Consider the cofiber sequence
\[ S(q,q) \rtarr D(q,q) \rtarr S^{q,q},\]
where $S(q,q)\subset D(q,q) \subset \R^{q,q}$ are the unit sphere and unit disk respectively. Since $D(q,q)$ is equivariantly contractible, this gives the exact sequence 
\[ \pi_{m,0}(ko_{C_2}) \ltarr \pi_{m+q,q}(ko_{C_2}) \ltarr ko_{C_2}^{-m-1,0}(S(q,q)) \ltarr \pi_{m+1,0}(ko_{C_2}).\]
If $m\leq -2$, the outer groups vanish. Moreover, $C_2$ acts freely on $S(q,q)$, and the orbit space is $S(q,q)/C_2 \iso \R\bP^{q-1}$. It follows (\cite{Alaska}*{Section~XIV.1}) that 
\[
ko_{C_2}^{-m-1}(S(q,q)) \iso ko^{-m-1}(\R\bP^{q-1})
\]
when $m \leq -2$ and $q \geq 1$.
In particular,
\[
\pi_{j,j+5}(ko_{C_2}) \iso ko^4(\R\bP^{j+4}),
\]
and the latter groups are known (cf \cite{DMEuler}*{Section~2}) 
to be cyclic of order $\varphi(j)$.
\end{rmk}

Having described all of the Milnor-Witt stems as
$\Pi_0(ko_{C_2})[\tau^4]$-modules, it remains only to 
understand products of the various
$\Pi_0(ko_{C_2})[\tau^4]$-module generators.  

\begin{prop}
In the homotopy groups of $ko_{C_2}$, we have the relations:
\begin{enumerate}
\item $(\tau^2\omega)^2 = 2\omega \cdot \tau^4$.
\item $\tau^2\omega\cdot \tau^2\alpha = \tau^4 \cdot \omega\alpha$.
\item $(\tau^2\alpha)^2 = 2 \tau^4 \cdot \omega \beta$.
\end{enumerate}
\end{prop}

\begin{pf}
The first relation is part of the definition of $\tau^2 \omega$.

For the second relation, use the definitions of $\tau^2 \alpha$
and of $\tau^2 \omega$ 
to see that
\[
2 \cdot \tau^2 \omega \cdot \tau^2 \alpha =
(\tau^2 \omega)^2 \alpha = 2 \tau^4 \cdot \omega \alpha.
\]
The group $\pi_{4,0} (ko_{C_2})$ has no 2-torsion,
so it follows that 
$\tau^2 \omega \cdot \tau^2 \alpha$ equals
$\tau^4 \cdot \omega \alpha$.

The proof of the third relation is similar.  Use the 
definitions of $\tau^2 \alpha$ and $\beta$ and part (2) to see that
\[
2 (\tau^2 \alpha)^2 = \tau^2 \omega \cdot \tau^2 \alpha \cdot \alpha =
\tau^4 \cdot \omega \alpha^2 = 4 \tau^4 \cdot \omega \beta.
\]
The group $\pi_{8,4} (ko_{C_2})$ has no 2-torsion.
\end{pf}

\subsection{The homotopy ring of $k\R$}

We may similarly describe the homotopy of $k\R$. Since this has already appeared in the literature (cf. \cite[Section~11]{GrM}), we do not give complete details.

We use the forgetful exact sequence of Proposition~\ref{prop:underlying-rho} to define the homotopy classes listed in Table~\ref{tbl:HtpyGenkR}. In each case, the forgetful map is injective, and we stipulate that $\tau^4$ restricts to $1$, that $v_1$ and $\tau^{-4}v_1$ restrict to the Bott element, and that $\tau^2\omega$, $\tau^{-2}\omega$, and $\tau^{-4}\omega$ all restrict to $2$.

\begin{prop}\label{T4extnskR}
There are $\tau^4$-extensions
\[ \tau^4 \cdot \tau^{-2}\omega = \tau^2 \omega, \quad \tau^4 \cdot \tau^{-4}\omega = 2, \quad \tau^4 \cdot \tau^{-4}v_1 = v_1.\]
\end{prop}

\begin{pf}
These all follow from the definition of these classes using the 
forgetful exact sequence of Proposition~\ref{prop:underlying-rho}.
Since the forgetful map is a ring homomorphism, we get that
\[\iota^*(\tau^4\cdot \tau^{-2}\omega) = \iota^*(\tau^4)\cdot \iota^*(\tau^{-2}\omega) = 1\cdot 2 = 2.\]
Since the forgetful map is injective in this degree, we conclude that 
$\tau^4\cdot \tau^{-2}\omega = \tau^2\omega$.
The same argument handles the other relations just as well.
\end{pf}

In order to describe the Milnor-Witt $0$-stem of $k\R$, it is convenient to write $\alpha=\tau^{-2}\omega\,v_1^2$ and $\beta=\tau^{-4}v_1\cdot v_1^3$.

\begin{prop} The (2-completed) Milnor-Witt $0$-stem of $k\R$ is
\[ \Pi_0(k\R) \iso 
\Z_2[\rho,\alpha,\beta]/
(2\rho,
\rho \alpha , \rho^3 \beta, \alpha^2-4 \beta),\]
where the generators have degrees
$(-1,-1)$, $(4,4)$, and $(8,8)$ respectively.
These homotopy classes are detected by $\rho$, 
$\frac{\gamma}{\tau} v_1^2$, and 
$\frac{\gamma}{\rho^2\tau^2}v_1^3$ in the Adams spectral sequence.
\end{prop}

The other Milnor-Witt stems, aside from those in degree $-5-4k$, can all be described cleanly as ideals in $\Pi_0(k\R)$.
The $\tau^4$-periodicities asserted in the following results all hold already on the level of $\Ext$, 
except for the $\tau^4$-multiplications from $\Ext_{NC}$ to $\Ext_{\cE(1)}$. Those are handled by Proposition~\ref{T4extnskR}. 
We recommend the reader to consult the diagram on page \pageref{fig:MWkR-module}  in order to visualize the following results.

\begin{prop} The map $\Pi_{-4}(k\R)\xrtarr{\tau^4} \Pi_{0}(k\R)$ 
is a monomorphism and identifies $\Pi_{-4}(k\R)$ with the ideal generated by $2$, $\alpha$, and $\beta$.
If $k\neq -1$, then multiplication by $\tau^4$ is an isomorphism $\Pi_{4k}(k\R) \iso \Pi_{4(k+1)}(k\R)$.
\end{prop}

Thus the Milnor-Witt stems of degree $4k$ break up into two families, which are displayed as the first two rows of the diagram on page \pageref{fig:MWkR-module}.

\begin{prop} The map $\Pi_{-1}(k\R)\xrtarr{v_1} \Pi_0(k\R)$ 
is a monomorphism and identifies $\Pi_{-1}(k\R)$ with the ideal generated by $\alpha$ and $\beta$.
Multiplication by $\tau^4$ is a split epimorphism
\[ \frac{\F_2[\rho]}{\rho^\infty} \rtarr \Pi_{-5}(k\R) \xrtarr{\tau^4} \Pi_{-1}(k\R).\]
If $k\neq -1$, then multiplication by $\tau^4$ is an isomorphism $\Pi_{-1+4k}(k\R) \iso \Pi_{3+4k}(k\R)$.
\end{prop}

\begin{prop} The map $\Pi_{-2}(k\R)\xrtarr{v_1} \Pi_{-1}(k\R)$ 
is an isomorphism.
Multiplication by $\tau^4$ is an isomorphism $\Pi_{4k-2}(k\R)\iso \Pi_{4k+2}(k\R)$ for all $k\in \Z$.
\end{prop}

\begin{prop} The map $\Pi_{-3}(k\R)\xrtarr{v_1^3} \Pi_{0}(k\R)$ 
is a monomorphism and identifies $\Pi_{-3}(k\R)$ with the ideal generated by $\beta$.
Multiplication by $\tau^4$ is an isomorphism $\Pi_{4k-3}(k\R)\iso \Pi_{4k+1}(k\R)$ for all $k\in \Z$.
\end{prop}

\begin{table}
\captionof{table}{Notation for $\pi_{*,*}(k\R)$}
\label{tbl:HtpyGenkR}
\begin{center}
\begin{tabular}{LLLLL} 
\hline
mw & (s,w) & \text{element} & \text{detected by} & \text{definition} \\ \hline
0 &  (-1,-1) & \rho &  \rho & \\
1 &  (2,1) & v_1 & v_1 & \iota^*(v_1)=v_1 \\
4 & (0,-4) & \tau^4 &   \tau^4 &  \iota^*(\tau^4) = 1\\
2 & (0,-2) & \tau^2\omega & \tau^2 h_0 
  & \iota^*(\tau^2\omega) = 2 \\
-2 & (0,2) & \tau^{-2} \omega & \frac{\gamma}{\tau}
  & \iota^*(\tau^{-2}\omega) = 2 \\
-4 & (0,4) & \tau^{-4} \omega & \frac{\gamma}{\tau^3} 
  &  \iota^*(\tau^{-4}\omega) = 2 \\
-3 & (2,5) & \tau^{-4} v_1 & \frac{\gamma}{\rho^2\tau^2} 
  & \iota^*(\tau^{-4}v_1)=v_1 \\
-5 & (0,5) & \frac{\Gamma}{\tau^4} & \frac{\gamma}{\tau^4}
  & \\ 
\hline
\end{tabular}
\end{center}
\end{table}

Combining the information from Table~\ref{tbl:ExtA1toE1} and Table~\ref{tbl:ExtNCA1toE1} 
yields the induced homomorphism on homotopy groups as described in 
Table~\ref{tbl:kotokR}. Note that all values $c_*(x)$ are to be interpreted as correct modulo
higher powers of $2$.

\begin{table}[h]
\captionof{table}{The homomorphism $\pi_{*,*}(ko_{C_2}) \xrtarr{c_*} \pi_{*,*}(k\R)$, 
modulo higher powers of $2$
\label{tbl:kotokR}}
\begin{center}
\begin{tabular}{LLLL} 
\hline
mw & (s,w) & \text{$x\in\pi_{*,*}(ko_{C_2})$} & \text{$c_*x\in\pi_{*,*}(k\R)$}  \\ \hline
0 &  (-1,-1) & \rho &  \rho \\
0 &  (1,1) & \eta & 0 \\
0 & (4,4) & \alpha &  \tau^{-2}\omega  \cdot v_1^2 \\ 
0 & (0,0) & \omega & 2  \\
4 & (0,-4) & \tau^4 &   \tau^4  \\
0 & (8,8) & \beta &    \tau^{-4} v_1 \cdot v_1^3 \\ 
2 & (0,-2) & \tau^2\omega & \tau^2 \omega \\
-2 & (0,2) & \tau^{-2}\omega & \tau^{-2}\omega \\
-4 & (0,4) & \tau^{-4}\omega & \tau^{-4}\omega \\
-5 & (j,j+5) & \frac{\Gamma}{\rho^{j}\tau^4} & \frac{\Gamma}{\rho^{j}\tau^4}
 \\
1 & (1,0) & \tau\eta & \rho v_1 \\
2 & (4,2) & \tau^2 \alpha &  2v_1^2 \\ 
 \hline
\end{tabular}
\end{center}
\end{table}

\begin{rmk}
Note that the results of this section provide another means of 
demonstrating the $\tau^4$-periodicity in $ko_{C_2}$ established
in Section~\ref{sec:t4Per}.
More specifically, the $\tau^4$-extensions given in Proposition~\ref{T4extnskR},
together with the homomorphism $c_*$ as described in 
Table~\ref{tbl:kotokR}, imply the $\tau^4$-extensions given in 
Proposition~\ref{prop:PhiMult}.
\end{rmk}


\section{Charts}
\label{sctn:chart}

\subsection{Bockstein $E^+$ and $\Ext_{\cA^\R(1)}$ charts}

The charts on pages \pageref{E+start}--\pageref{E+end}
depict the Bockstein $E^+$
spectral sequence that converges to $\Ext_{\cA^\R(1)}$.
The details of this calculation are described in 
Section \ref{sec:ExtAR1}.

The $E_2^+$-page
is too complicated to present
conveniently in one chart, so this page is separated into two parts by 
Milnor-Witt stem modulo 2.
Similarly, the $E_3^+$-page is separated into four parts by
Milnor-Witt stem modulo 4.
The $E_4^+$-page in Milnor-Witt stems 0 or 1 modulo 4 is not shown,
since it is identical to the $E_3^+$-page in those Milnor-Witt stems.
The $E_4^+$-page in Milnor-Witt stems 3 modulo 4 is not shown because 
it is zero.

Here is a key for reading the Bockstein charts:
\begin{enumerate}
\item
Gray dots and green dots indicate groups as displayed on the charts.
\item
Horizontal lines indicate multiplications by $\rho$.
\item
Vertical lines indicate multiplications by $h_0$.
\item
Diagonal lines indicate multiplications by $h_1$.
\item
Horizontal arrows indicate infinite sequences of multiplications
by $\rho$.
\item
Vertical arrows indicate infinite sequences of multiplications
by $h_0$.
\item
Diagonal arrows indicate infinite sequences of multiplications
by $h_1$.
\end{enumerate}

Here is a key for the charts of $\Ext_{\cA^\R(1)}$:
\begin{enumerate}
\item 
Gray dots indicate copies of $\F_2[\tau^4]$ that arise 
from a copy of $\F_2[\tau^4]$ in
the $E_\infty^+$-page.
\item 
Green dots indicate copies of $\F_2[\tau^4]$ that arise
from a copy of $\F_2$ and a copy of $\F_2[\tau^4]$
in the $E_\infty^+$-page, connected by a 
$\tau^4$ extension that is hidden in the Bockstein spectral
sequence.
For example, the green dot at $(3,3)$ arises from
a hidden $\tau^4$ extension from $h_1^3$ to $\rho \cdot \tau^2 a$.
\item
Blue dots indicate copies of $\F_2[\tau^4]$ that arise
from two copies of $\F_2$ and one copy of $\F_2[\tau^4]$
in the $E_\infty^+$-page, connected by
$\tau^4$ extensions that are hidden in the Bockstein spectral
sequence.
For example, the blue dot at $(7,7)$ arises from
hidden $\tau^4$ extensions from $h_1^7$ to $\rho^4 h_1^3 b$,
and from $\rho^4 h_1^3 b$ to $\rho^5 \cdot \tau^2 a \cdot b$.
\item
Horizontal lines indicate multiplications by $\rho$.
\item
Vertical lines indicate multiplications by $h_0$.
\item
Diagonal lines indicate multiplications by $h_1$.
\item
Dashed lines indicate extensions that are hidden in the Bockstein
spectral sequence.
\item
Orange horizontal lines indicate
$\rho$ multiplications that equal $\tau^4$ times a generator.
For example, $\rho \cdot \tau^2 a$ equals $\tau^4 \cdot h_1^3$.
\item
Horizontal arrows indicate infinite sequences of multiplications
by $\rho$.
\item
Vertical arrows indicate infinite sequences of multiplications
by $h_0$.
\item
Diagonal arrows indicate infinite sequences of multiplications
by $h_1$.
\end{enumerate}

\subsection{Bockstein $E^-$ and $\Ext_{NC}$ charts for $\cA^{C_2}(1)$}

The charts on pages \pageref{E-start}--\pageref{E-end}
depict the Bockstein $E^-$
spectral sequence that converges to $\Ext_{NC}$.
The details of this calculation are described in 
Section \ref{sec:BockNegCone}.

The $E_2^-$-page
is too complicated to present
conveniently in one chart, so this page is separated into two parts by 
Milnor-Witt stem modulo 2.
Similarly, the $E_3^-$-page is separated into four parts by
Milnor-Witt stem modulo 4.
The $E_4^-$-page in Milnor-Witt stems 0 or 3 modulo 4 is not shown,
since it is identical to the $E_3^-$-page in those Milnor-Witt stems.
The $E_5^-$-page and $E_6^-$-page in
Milnor-Witt stems 1 or 2 modulo 4 is not shown, since it is identical
to the $E_4^-$-page in those Milnor-Witt stems.

Here is a key for reading the Bockstein charts:
\begin{enumerate}
\item
Gray dots and green dots indicate groups as displayed on the charts.
\item
Horizontal lines indicate multiplications by $\rho$.
\item
Vertical lines indicate multiplications by $h_0$.
\item
Diagonal lines indicate multiplications by $h_1$.
\item
Horizontal rightward arrows indicate infinite sequences of divisions
by $\rho$, i.e., infinitely $\rho$-divisible elements.
\item
Vertical arrows indicate infinite sequences of multiplications
by $h_0$.
\item
Diagonal arrows indicate infinite sequences of multiplications
by $h_1$.
\end{enumerate}

The structure of $\Ext_{NC}$
is too complicated to present
conveniently in one chart, so it is separated into parts by 
Milnor-Witt stem modulo 4.
Unfortunately, the part in positive Milnor-Witt stems 0 modulo 4 alone is
still too complicated to present conveniently in one chart.
Instead, we display
$\Ext_{C_2}$, including both $\Ext_{\cA^\R(1)}$ and
$\Ext_{NC}$, for the Milnor-Witt $0$-stem and the
Milnor-Witt $4$-stem.

Here is a key for the charts of $\Ext_{NC}$:
\begin{enumerate}
\item 
Gray dots indicate copies of $\F_2[\tau^4] / \tau^\infty$.
\item
Horizontal lines indicate multiplications by $\rho$.
\item
Vertical lines indicate multiplications by $h_0$.
\item
Diagonal lines indicate multiplications by $h_1$.
\item
Dashed lines indicate extensions that are hidden in the Bockstein
spectral sequence.
\item
Dashed lines of slope $-1$ indicate $\rho$ extensions that 
are hidden in the Adams spectral sequence.
\item
Horizontal rightward arrows indicate infinite sequences of divisions
by $\rho$, i.e., infinitely $\rho$-divisible elements.
\item
Vertical arrows indicate infinite sequences of multiplications
by $h_0$.
\item
Diagonal arrows indicate infinite sequences of multiplications
by $h_1$.
\end{enumerate}

\subsection{Bockstein and $\Ext$ charts for $\cE^{C_2}(1)$}

The charts on page \pageref{BockE1start}
depict the Bockstein $E^+$ and $E^-$
spectral sequences that converge to 
$\Ext_{\cE^\R(1)}$ and $\Ext_{\cE^{\R}(1)}(NC, \bM_2^\R)$,
respectively.
The details of this calculation are described in 
Remark~\ref{ExtE1} and Section~\ref{sec:BockE1NC}.
For legibility, we have split each of the $E^+_\infty$, $E^-_4$, 
and $\Ext_{NC}$ pages into a pair of charts, 
organized by families of $v_1$-multiples rather
than by Milnor-Witt stems.

Here is a key for reading the Bockstein and $\Ext_{NC}$ charts:
\begin{enumerate}
\item
Gray dots indicate groups as displayed on the charts.
\item
Horizontal lines indicate multiplications by $\rho$.
\item
Vertical lines indicate multiplications by $h_0$. Dashed vertical lines denote $h_0$-multiplications that are hidden in the Bockstein spectral sequence
\item
Horizontal rightward arrows indicate infinite sequences of divisions
by $\rho$, i.e., infinitely $\rho$-divisible elements.
\item
Vertical arrows indicate infinite sequences of multiplications
by $h_0$.
\end{enumerate}

\subsection{Milnor-Witt stems}

The diagrams on pages \pageref{fig:MW-module} and \pageref{fig:MWkR-module} depict the Milnor-Witt stems for $ko_{C_2}$ and $k\R$ in families as described in Section~\ref{sec:homotopy}. 

The top figure on page \pageref{fig:MW-module} represents the Milnor-Witt $4k$-stem, where $k\geq 0$. The middle three figures represent the $\tau^4$-periodic classes, as in Theorem~\ref{thm:PhiPerd}. The bottom figure represents the Milnor-Witt stem $\Pi_n$, where $n\equiv 3 \pmod{4}$ and $n\leq -5$.

Here is a key for reading the Milnor-Witt charts:
\begin{enumerate}
\item
Black dots indicate copies of $\F_2$.
\item
Hollow circles indicate copies of $\Z^2_2$.
\item
Circled numbers indicate cyclic groups of given order. 
For instance, the 1-stem of $\Pi_{-5}$ is $\Z/4$.
\item
Blue lines indicate multiplications by $\eta$.
\item
Red lines indicate multiplications by $\rho$.
\item
Curved green lines denote multiplications by $\alpha$.
\item
Lines labelled with numbers indicate that a multiplication
equals a multiple of an additive generator.
For example, 
$\alpha \cdot \eta^4$ equals $4 \eta \rho \beta$
in $\Pi_0$.
\end{enumerate}

For clarity, some $\alpha$ multiplications are not shown in the first and last
diagrams of page \pageref{fig:MW-module}.  For example, the
 $\alpha$ multiplication on $\eta$ is not shown in the first diagram.


\psset{linewidth=0.3mm}
\psset{unit=0.555cm}

\begin{landscape}

\label{E+start}
\centerline{Bockstein charts for $\cA^{\R}(1)$}

\begin{pspicture}(-1,-1)(17,9)
\psgrid[unit=2,gridcolor=gridline,subgriddiv=0,gridlabelcolor=white](0,0)(8,4)
\scriptsize
\rput(0,-1){0}
\rput(2,-1){2}
\rput(4,-1){4}
\rput(6,-1){6}
\rput(8,-1){8}
\rput(10,-1){10}
\rput(12,-1){12}
\rput(14,-1){14}
\rput(16,-1){16}

\rput(-1,0){0}
\rput(-1,2){2}
\rput(-1,4){4}
\rput(-1,6){6}
\rput(-1,8){8}

\uput[0](0,8.2){\textsc{Bockstein $E_1^+$-page}}
\uput{\labelrad}[-90](0,0){$1$}
\psline[linecolor=pcolor]{->}(0,0)(-0.70,0)
\psline[linecolor=hzerocolor]{->}(0,0)(0,0.70)
\psline[linecolor=honecolor](0,0)(1,1)
\pscircle*[linecolor=tauzerocolor](0,0){\cirrad}
\psline[linecolor=pcolor]{->}(1,1)(0.30,1)
\psline[linecolor=honecolor](1,1)(2,2)
\pscircle*[linecolor=tauzerocolor](1,1){\cirrad}
\psline[linecolor=pcolor]{->}(2,2)(1.30,2)
\psline[linecolor=tauonecolor](2,2)(3,3)
\pscircle*[linecolor=tauzerocolor](2,2){\cirrad}
\psline[linecolor=tauonecolor]{->}(3,3)(2.30,3)
\psline[linecolor=tauonecolor]{->}(3,3)(3.70,3.70)
\pscircle*[linecolor=tauonecolor](3,3){\cirrad}
\uput{\labelrad}[-90](4,3){$a$}
\psline[linecolor=pcolor]{->}(4,3)(3.30,3)
\psline[linecolor=hzerocolor]{->}(4,3)(4,3.70)
\pscircle*[linecolor=tauzerocolor](4,3){\cirrad}
\uput{\labelrad}[-90](8,4){$b$}
\psline[linecolor=pcolor]{->}(8,4)(7.30,4)
\psline[linecolor=hzerocolor]{->}(8,4)(8,4.70)
\psline[linecolor=honecolor](8,4)(9,5)
\pscircle*[linecolor=tauzerocolor](8,4){\cirrad}
\psline[linecolor=pcolor]{->}(9,5)(8.30,5)
\psline[linecolor=honecolor](9,5)(10,6)
\pscircle*[linecolor=tauzerocolor](9,5){\cirrad}
\psline[linecolor=pcolor]{->}(10,6)(9.30,6)
\psline[linecolor=tauonecolor](10,6)(11,7)
\pscircle*[linecolor=tauzerocolor](10,6){\cirrad}
\psline[linecolor=tauonecolor]{->}(11,7)(10.30,7)
\psline[linecolor=tauonecolor]{->}(11,7)(11.70,7.70)
\pscircle*[linecolor=tauonecolor](11,7){\cirrad}
\uput{\labelrad}[-90](12,7){$a b$}
\psline[linecolor=pcolor]{->}(12,7)(11.30,7)
\psline[linecolor=hzerocolor]{->}(12,7)(12,7.70)
\pscircle*[linecolor=tauzerocolor](12,7){\cirrad}
\uput{\labelrad}[-90](16,8){$b^{2}$}
\psline[linecolor=pcolor]{->}(16,8)(15.30,8)
\psline[linecolor=hzerocolor]{->}(16,8)(16,8.70)
\pscircle*[linecolor=tauzerocolor](16,8){\cirrad}
\uput{\labelrad}[0](1,7){$\F_2[\tau]$}
\pscircle*[linecolor=tauzerocolor](1,7){\cirrad}
\uput{\labelrad}[0](1,6.5){$\F_2$}
\pscircle*[linecolor=tauonecolor](1,6.5){\cirrad}
\end{pspicture}
\vfill

\begin{pspicture}(-1,-1)(17,9)
\psgrid[unit=2,gridcolor=gridline,subgriddiv=0,gridlabelcolor=white](0,0)(8,4)
\scriptsize
\rput(0,-1){0}
\rput(2,-1){2}
\rput(4,-1){4}
\rput(6,-1){6}
\rput(8,-1){8}
\rput(10,-1){10}
\rput(12,-1){12}
\rput(14,-1){14}
\rput(16,-1){16}

\rput(-1,0){0}
\rput(-1,2){2}
\rput(-1,4){4}
\rput(-1,6){6}
\rput(-1,8){8}

\uput[0](0,8.2){\textsc{Bockstein $E_2^+$-page, $mw \equiv 0 \pmod 2$}}
\uput{\labelrad}[-90](0,0){$1$}
\psline[linecolor=pcolor]{->}(0,0)(-0.70,0)
\psline[linecolor=hzerocolor](0,0)(0,1)
\psline[linecolor=honecolor](0,0)(1,1)
\pscircle*[linecolor=tauzerocolor](0,0){\cirrad}
\psline[linecolor=hzerocolor]{->}(0,1)(0,1.70)
\pscircle*[linecolor=tauzerocolor](0,1){\cirrad}
\psline[linecolor=pcolor]{->}(1,1)(0.30,1)
\psline[linecolor=honecolor](1,1)(2,2)
\pscircle*[linecolor=tauzerocolor](1,1){\cirrad}
\psline[linecolor=pcolor]{->}(2,2)(1.30,2)
\psline[linecolor=tauonecolor](2,2)(3,3)
\pscircle*[linecolor=tauzerocolor](2,2){\cirrad}
\psline[linecolor=tauonecolor]{->}(3,3)(2.30,3)
\psline[linecolor=tauonecolor]{->}(3,3)(3.70,3.70)
\pscircle*[linecolor=tauonecolor](3,3){\cirrad}
\uput{\labelrad}[-90](4,3){$a$}
\psline[linecolor=pcolor]{->}(4,3)(3.30,3)
\psline[linecolor=hzerocolor](4,3)(4,4)
\pscircle*[linecolor=tauzerocolor](4,3){\cirrad}
\psline[linecolor=hzerocolor]{->}(4,4)(4,4.70)
\pscircle*[linecolor=tauzerocolor](4,4){\cirrad}
\uput{\labelrad}[-90](8,4){$b$}
\psline[linecolor=pcolor]{->}(8,4)(7.30,4)
\psline[linecolor=hzerocolor](8,4)(8,5)
\psline[linecolor=honecolor](8,4)(9,5)
\pscircle*[linecolor=tauzerocolor](8,4){\cirrad}
\psline[linecolor=hzerocolor]{->}(8,5)(8,5.70)
\pscircle*[linecolor=tauzerocolor](8,5){\cirrad}
\psline[linecolor=pcolor]{->}(9,5)(8.30,5)
\psline[linecolor=honecolor](9,5)(10,6)
\pscircle*[linecolor=tauzerocolor](9,5){\cirrad}
\psline[linecolor=pcolor]{->}(10,6)(9.30,6)
\psline[linecolor=tauonecolor](10,6)(11,7)
\pscircle*[linecolor=tauzerocolor](10,6){\cirrad}
\psline[linecolor=tauonecolor]{->}(11,7)(10.30,7)
\psline[linecolor=tauonecolor]{->}(11,7)(11.70,7.70)
\pscircle*[linecolor=tauonecolor](11,7){\cirrad}
\uput{\labelrad}[-90](12,7){$a b$}
\psline[linecolor=pcolor]{->}(12,7)(11.30,7)
\psline[linecolor=hzerocolor](12,7)(12,8)
\pscircle*[linecolor=tauzerocolor](12,7){\cirrad}
\psline[linecolor=hzerocolor]{->}(12,8)(12,8.70)
\pscircle*[linecolor=tauzerocolor](12,8){\cirrad}
\uput{\labelrad}[-90](16,8){$b^{2}$}
\psline[linecolor=pcolor]{->}(16,8)(15.30,8)
\pscircle*[linecolor=tauzerocolor](16,8){\cirrad}
\uput{\labelrad}[0](1,7){$\F_2[\tau^2]$}
\pscircle*[linecolor=tauzerocolor](1,7){\cirrad}
\uput{\labelrad}[0](1,6.5){$\F_2$}
\pscircle*[linecolor=tauonecolor](1,6.5){\cirrad}
\end{pspicture}
\begin{pspicture}(-1,-1)(17,9)
\psgrid[unit=2,gridcolor=gridline,subgriddiv=0,gridlabelcolor=white](0,0)(8,4)
\scriptsize
\rput(0,-1){0}
\rput(2,-1){2}
\rput(4,-1){4}
\rput(6,-1){6}
\rput(8,-1){8}
\rput(10,-1){10}
\rput(12,-1){12}
\rput(14,-1){14}
\rput(16,-1){16}

\rput(-1,0){0}
\rput(-1,2){2}
\rput(-1,4){4}
\rput(-1,6){6}
\rput(-1,8){8}

\uput[0](0,8.2){\textsc{Bockstein $E_2^+$-page, $mw \equiv 1 \pmod 2$}}
\uput{\labelrad}[-90](1,1){$\tau h_1$}
\psline[linecolor=pcolor]{->}(1,1)(0.30,1)
\psline[linecolor=honecolor](1,1)(2,2)
\pscircle*[linecolor=tauzerocolor](1,1){\cirrad}
\psline[linecolor=pcolor]{->}(2,2)(1.30,2)
\pscircle*[linecolor=tauzerocolor](2,2){\cirrad}
\uput{\labelrad}[-90](9,5){$\tau h_1 b$}
\psline[linecolor=pcolor]{->}(9,5)(8.30,5)
\psline[linecolor=honecolor](9,5)(10,6)
\pscircle*[linecolor=tauzerocolor](9,5){\cirrad}
\psline[linecolor=pcolor]{->}(10,6)(9.30,6)
\pscircle*[linecolor=tauzerocolor](10,6){\cirrad}
\uput{\labelrad}[0](1,7){$\F_2[\tau^2]$}
\pscircle*[linecolor=tauzerocolor](1,7){\cirrad}
\end{pspicture}
\vfill

\newpage
\centerline{Bockstein charts for $\cA^{\R}(1)$}

\begin{pspicture}(-1,-1)(17,9)
\psgrid[unit=2,gridcolor=gridline,subgriddiv=0,gridlabelcolor=white](0,0)(8,4)
\scriptsize
\rput(0,-1){0}
\rput(2,-1){2}
\rput(4,-1){4}
\rput(6,-1){6}
\rput(8,-1){8}
\rput(10,-1){10}
\rput(12,-1){12}
\rput(14,-1){14}
\rput(16,-1){16}

\rput(-1,0){0}
\rput(-1,2){2}
\rput(-1,4){4}
\rput(-1,6){6}
\rput(-1,8){8}

\uput[0](0,8.2){\textsc{Bockstein $E_3^+ = E_\infty^+$-page, $mw \equiv 0 \pmod 4$}}
\uput{\labelrad}[-90](0,0){$1$}
\psline[linecolor=pcolor]{->}(0,0)(-0.70,0)
\psline[linecolor=hzerocolor](0,0)(0,1)
\psline[linecolor=honecolor](0,0)(1,1)
\pscircle*[linecolor=tauzerocolor](0,0){\cirrad}
\psline[linecolor=hzerocolor]{->}(0,1)(0,1.70)
\pscircle*[linecolor=tauzerocolor](0,1){\cirrad}
\psline[linecolor=pcolor]{->}(1,1)(0.30,1)
\psline[linecolor=honecolor](1,1)(2,2)
\pscircle*[linecolor=tauzerocolor](1,1){\cirrad}
\psline[linecolor=pcolor]{->}(2,2)(1.30,2)
\psline[linecolor=tauonecolor](2,2)(3,3)
\pscircle*[linecolor=tauzerocolor](2,2){\cirrad}
\psline[linecolor=tauonecolor]{->}(3,3)(2.30,3)
\psline[linecolor=tauonecolor]{->}(3,3)(3.70,3.70)
\pscircle*[linecolor=tauonecolor](3,3){\cirrad}
\uput{\labelrad}[-90](4,3){$\tau^{2} a$}
\psline[linecolor=pcolor]{->}(4,3)(3.30,3)
\psline[linecolor=hzerocolor](4,3)(4,4)
\pscircle*[linecolor=tauzerocolor](4,3){\cirrad}
\psline[linecolor=hzerocolor]{->}(4,4)(4,4.70)
\pscircle*[linecolor=tauzerocolor](4,4){\cirrad}
\uput{\labelrad}[-90](8,4){$b$}
\psline[linecolor=pcolor]{->}(8,4)(7.30,4)
\psline[linecolor=hzerocolor](8,4)(8,5)
\psline[linecolor=honecolor](8,4)(9,5)
\pscircle*[linecolor=tauzerocolor](8,4){\cirrad}
\psline[linecolor=hzerocolor]{->}(8,5)(8,5.70)
\pscircle*[linecolor=tauzerocolor](8,5){\cirrad}
\psline[linecolor=pcolor]{->}(9,5)(8.30,5)
\psline[linecolor=honecolor](9,5)(10,6)
\pscircle*[linecolor=tauzerocolor](9,5){\cirrad}
\psline[linecolor=pcolor]{->}(10,6)(9.30,6)
\psline[linecolor=tauonecolor](10,6)(11,7)
\pscircle*[linecolor=tauzerocolor](10,6){\cirrad}
\psline[linecolor=tauonecolor]{->}(11,7)(10.30,7)
\psline[linecolor=tauonecolor]{->}(11,7)(11.70,7.70)
\pscircle*[linecolor=tauonecolor](11,7){\cirrad}
\uput{\labelrad}[-90](12,7){$\tau^{2} a b$}
\psline[linecolor=pcolor]{->}(12,7)(11.30,7)
\psline[linecolor=hzerocolor](12,7)(12,8)
\pscircle*[linecolor=tauzerocolor](12,7){\cirrad}
\psline[linecolor=hzerocolor]{->}(12,8)(12,8.70)
\pscircle*[linecolor=tauzerocolor](12,8){\cirrad}
\uput{\labelrad}[-90](16,8){$b^{2}$}
\psline[linecolor=pcolor]{->}(16,8)(15.30,8)
\pscircle*[linecolor=tauzerocolor](16,8){\cirrad}
\uput{\labelrad}[0](1,6.5){$\F_2$}
\pscircle*[linecolor=tauonecolor](1,6.5){\cirrad}
\uput{\labelrad}[0](1,7){$\F_2[\tau^4]$}
\pscircle*[linecolor=tauzerocolor](1,7){\cirrad}
\end{pspicture}
\begin{pspicture}(-1,-1)(17,9)
\psgrid[unit=2,gridcolor=gridline,subgriddiv=0,gridlabelcolor=white](0,0)(8,4)
\scriptsize
\rput(0,-1){0}
\rput(2,-1){2}
\rput(4,-1){4}
\rput(6,-1){6}
\rput(8,-1){8}
\rput(10,-1){10}
\rput(12,-1){12}
\rput(14,-1){14}
\rput(16,-1){16}

\rput(-1,0){0}
\rput(-1,2){2}
\rput(-1,4){4}
\rput(-1,6){6}
\rput(-1,8){8}

\uput[0](0,8.2){\textsc{Bockstein $E_3^+ = E_\infty^+$-page, $mw \equiv 1 \pmod 4$}}
\uput{\labelrad}[-90](1,1){$\tau h_1$}
\psline[linecolor=pcolor](1,1)(0,1)
\psline[linecolor=honecolor](1,1)(2,2)
\pscircle*[linecolor=tauzerocolor](1,1){\cirrad}
\psline[linecolor=honecolor](0,1)(1,2)
\pscircle*[linecolor=tauzerocolor](0,1){\cirrad}
\psline[linecolor=pcolor](2,2)(1,2)
\pscircle*[linecolor=tauzerocolor](2,2){\cirrad}
\pscircle*[linecolor=tauzerocolor](1,2){\cirrad}
\uput{\labelrad}[-90](9,5){$\tau h_1 b$}
\psline[linecolor=pcolor](9,5)(8,5)
\psline[linecolor=honecolor](9,5)(10,6)
\pscircle*[linecolor=tauzerocolor](9,5){\cirrad}
\psline[linecolor=honecolor](8,5)(9,6)
\pscircle*[linecolor=tauzerocolor](8,5){\cirrad}
\psline[linecolor=pcolor](10,6)(9,6)
\pscircle*[linecolor=tauzerocolor](10,6){\cirrad}
\pscircle*[linecolor=tauzerocolor](9,6){\cirrad}
\uput{\labelrad}[0](1,7){$\F_2[\tau^4]$}
\pscircle*[linecolor=tauzerocolor](1,7){\cirrad}
\end{pspicture}
\vfill

\begin{pspicture}(-1,-1)(17,9)
\psgrid[unit=2,gridcolor=gridline,subgriddiv=0,gridlabelcolor=white](0,0)(8,4)
\scriptsize
\rput(0,-1){0}
\rput(2,-1){2}
\rput(4,-1){4}
\rput(6,-1){6}
\rput(8,-1){8}
\rput(10,-1){10}
\rput(12,-1){12}
\rput(14,-1){14}
\rput(16,-1){16}

\rput(-1,0){0}
\rput(-1,2){2}
\rput(-1,4){4}
\rput(-1,6){6}
\rput(-1,8){8}

\uput[0](0,8.2){\textsc{Bockstein $E_3^+$-page, $mw \equiv 2 \pmod 4$}}
\uput{\labelrad}[-90](0,1){$\tau^{2} h_0$}
\psline[linecolor=hzerocolor]{->}(0,1)(0,1.70)
\pscircle*[linecolor=tauzerocolor](0,1){\cirrad}
\uput{\labelrad}[-90](2,2){$\tau^{2} h_1^{2}$}
\psline[linecolor=pcolor](2,2)(1,2)
\pscircle*[linecolor=tauzerocolor](2,2){\cirrad}
\pscircle*[linecolor=tauzerocolor](1,2){\cirrad}
\uput{\labelrad}[-90](4,3){$a$}
\psline[linecolor=pcolor]{->}(4,3)(3.30,3)
\psline[linecolor=hzerocolor](4,3)(4,4)
\pscircle*[linecolor=tauzerocolor](4,3){\cirrad}
\psline[linecolor=hzerocolor]{->}(4,4)(4,4.70)
\pscircle*[linecolor=tauzerocolor](4,4){\cirrad}
\uput{\labelrad}[-90](8,5){$\tau^{2} h_0 b$}
\psline[linecolor=hzerocolor]{->}(8,5)(8,5.70)
\pscircle*[linecolor=tauzerocolor](8,5){\cirrad}
\uput{\labelrad}[-90](10,6){$\tau^{2} h_1^{2} b$}
\psline[linecolor=pcolor](10,6)(9,6)
\pscircle*[linecolor=tauzerocolor](10,6){\cirrad}
\pscircle*[linecolor=tauzerocolor](9,6){\cirrad}
\uput{\labelrad}[-90](12,7){$a b$}
\psline[linecolor=pcolor]{->}(12,7)(11.30,7)
\psline[linecolor=hzerocolor](12,7)(12,8)
\pscircle*[linecolor=tauzerocolor](12,7){\cirrad}
\psline[linecolor=hzerocolor]{->}(12,8)(12,8.70)
\pscircle*[linecolor=tauzerocolor](12,8){\cirrad}
\uput{\labelrad}[0](1,7){$\F_2[\tau^4]$}
\pscircle*[linecolor=tauzerocolor](1,7){\cirrad}
\end{pspicture}
\begin{pspicture}(-1,-1)(17,9)
\psgrid[unit=2,gridcolor=gridline,subgriddiv=0,gridlabelcolor=white](0,0)(8,4)
\scriptsize
\rput(0,-1){0}
\rput(2,-1){2}
\rput(4,-1){4}
\rput(6,-1){6}
\rput(8,-1){8}
\rput(10,-1){10}
\rput(12,-1){12}
\rput(14,-1){14}
\rput(16,-1){16}

\rput(-1,0){0}
\rput(-1,2){2}
\rput(-1,4){4}
\rput(-1,6){6}
\rput(-1,8){8}

\uput[0](0,8.2){\textsc{Bockstein $E_3^+$-page, $mw \equiv 3 \pmod 4$}}
\uput{\labelrad}[-90](2,2){$\tau^{3} h_1^{2}$}
\psline[linecolor=pcolor]{->}(2,2)(1.30,2)
\pscircle*[linecolor=tauzerocolor](2,2){\cirrad}
\uput{\labelrad}[-90](10,6){$\tau^{3} h_1^{2} b$}
\psline[linecolor=pcolor]{->}(10,6)(9.30,6)
\pscircle*[linecolor=tauzerocolor](10,6){\cirrad}
\uput{\labelrad}[0](1,7){$\F_2[\tau^4]$}
\pscircle*[linecolor=tauzerocolor](1,7){\cirrad}
\end{pspicture}
\vfill

\newpage
\centerline{Bockstein and $\Ext$ charts for $\cA^{\R}(1)$}

\begin{pspicture}(-1,-1)(17,9)
\psgrid[unit=2,gridcolor=gridline,subgriddiv=0,gridlabelcolor=white](0,0)(8,4)
\scriptsize
\rput(0,-1){0}
\rput(2,-1){2}
\rput(4,-1){4}
\rput(6,-1){6}
\rput(8,-1){8}
\rput(10,-1){10}
\rput(12,-1){12}
\rput(14,-1){14}
\rput(16,-1){16}

\rput(-1,0){0}
\rput(-1,2){2}
\rput(-1,4){4}
\rput(-1,6){6}
\rput(-1,8){8}

\uput[0](0,8.2){\textsc{Bockstein $E_4^+ = E_\infty^+$-page, $mw \equiv 2 \pmod 4$}}
\uput{\labelrad}[-90](0,1){$\tau^{2} h_0$}
\psline[linecolor=hzerocolor]{->}(0,1)(0,1.70)
\pscircle*[linecolor=tauzerocolor](0,1){\cirrad}
\pscircle*[linecolor=tauzerocolor](1,2){\cirrad}
\uput{\labelrad}[-90](2,2){$\tau^{2} h_1^{2}$}
\psline[linecolor=pcolor](2,2)(1,2)
\pscircle*[linecolor=tauzerocolor](2,2){\cirrad}
\pscircle*[linecolor=tauzerocolor](2,3){\cirrad}
\psline[linecolor=pcolor](3,3)(2,3)
\pscircle*[linecolor=tauzerocolor](3,3){\cirrad}
\uput{\labelrad}[-90](4,3){$a$}
\psline[linecolor=pcolor](4,3)(3,3)
\psline[linecolor=hzerocolor](4,3)(4,4)
\pscircle*[linecolor=tauzerocolor](4,3){\cirrad}
\psline[linecolor=hzerocolor]{->}(4,4)(4,4.70)
\pscircle*[linecolor=tauzerocolor](4,4){\cirrad}
\uput{\labelrad}[-90](8,5){$\tau^{2} h_0 b$}
\psline[linecolor=hzerocolor]{->}(8,5)(8,5.70)
\pscircle*[linecolor=tauzerocolor](8,5){\cirrad}
\pscircle*[linecolor=tauzerocolor](9,6){\cirrad}
\uput{\labelrad}[-90](10,6){$\tau^{2} h_1^{2} b$}
\psline[linecolor=pcolor](10,6)(9,6)
\pscircle*[linecolor=tauzerocolor](10,6){\cirrad}
\pscircle*[linecolor=tauzerocolor](10,7){\cirrad}
\psline[linecolor=pcolor](11,7)(10,7)
\pscircle*[linecolor=tauzerocolor](11,7){\cirrad}
\uput{\labelrad}[-90](12,7){$a b$}
\psline[linecolor=pcolor](12,7)(11,7)
\psline[linecolor=hzerocolor](12,7)(12,8)
\pscircle*[linecolor=tauzerocolor](12,7){\cirrad}
\psline[linecolor=hzerocolor]{->}(12,8)(12,8.70)
\pscircle*[linecolor=tauzerocolor](12,8){\cirrad}
\uput{\labelrad}[0](1,7){$\F_2[\tau^4]$}
\pscircle*[linecolor=tauzerocolor](1,7){\cirrad}
\end{pspicture}
\begin{pspicture}(-1,-1)(17,9)
\psgrid[unit=2,gridcolor=gridline,subgriddiv=0,gridlabelcolor=white](0,0)(8,4)
\scriptsize
\rput(0,-1){0}
\rput(2,-1){2}
\rput(4,-1){4}
\rput(6,-1){6}
\rput(8,-1){8}
\rput(10,-1){10}
\rput(12,-1){12}
\rput(14,-1){14}
\rput(16,-1){16}

\rput(-1,0){0}
\rput(-1,2){2}
\rput(-1,4){4}
\rput(-1,6){6}
\rput(-1,8){8}

\uput[0](0,8.2){\textsc{$\Ext_{\cA^{\R}(1)}$, $mw \equiv 0 \pmod 4$}}
\uput{\labelrad}[-90](0,0){$1$}
\psline[linecolor=pcolor]{->}(0,0)(-0.70,0)
\psline[linecolor=hzerocolor](0,0)(0,1)
\psline[linecolor=honecolor](0,0)(1,1)
\pscircle*[linecolor=tauzerocolor](0,0){\cirrad}
\psline[linecolor=hzerocolor]{->}(0,1)(0,1.70)
\pscircle*[linecolor=tauzerocolor](0,1){\cirrad}
\psline[linecolor=pcolor]{->}(1,1)(0.30,1)
\psline[linecolor=honecolor](1,1)(2,2)
\pscircle*[linecolor=tauzerocolor](1,1){\cirrad}
\psline[linecolor=pcolor]{->}(2,2)(1.30,2)
\psline[linecolor=hiddentaucolor](2,2)(3,3)
\pscircle*[linecolor=tauzerocolor](2,2){\cirrad}
\psline[linecolor=hiddentaucolor]{->}(3,3)(2.30,3)
\psline[linecolor=hiddentaucolor](3,3)(3.90,4.10)
\pscircle*[linecolor=hiddentaucolor](3,3){\cirrad}
\uput{\labelrad}[-90](4,3){$\tau^{2} a$}
\psline[linecolor=taufourextncolor](4,3)(3,3)
\psline[linecolor=hzerocolor](4,3)(4.10,3.90)
\psline[linecolor=honecolor,linestyle=dashed, dash=3pt 2pt](4,3)(5,4)
\pscircle*[linecolor=tauzerocolor](4,3){\cirrad}
\psline[linecolor=hzerocolor]{->}(4.10,3.90)(4.10,4.60)
\pscircle*[linecolor=tauzerocolor](4.10,3.90){\cirrad}
\psline[linecolor=hiddentaucolor]{->}(3.90,4.10)(3.20,4.10)
\psline[linecolor=hiddentaucolor](3.90,4.10)(5,5)
\pscircle*[linecolor=hiddentaucolor](3.90,4.10){\cirrad}
\psline[linecolor=taufourextncolor](5,4)(3.90,4.10)
\psline[linecolor=honecolor](5,4)(6,5)
\pscircle*[linecolor=tauzerocolor](5,4){\cirrad}
\psline[linecolor=pcolor](6,4)(5,4)
\psline[linecolor=honecolor](6,4)(7,5)
\pscircle*[linecolor=tauzerocolor](6,4){\cirrad}
\psline[linecolor=pcolor](7,4)(6,4)
\psline[linecolor=honecolor](7,4)(7.90,5.10)
\pscircle*[linecolor=tauzerocolor](7,4){\cirrad}
\uput{\labelrad}[-90](8,4){$b$}
\psline[linecolor=pcolor](8,4)(7,4)
\psline[linecolor=hzerocolor](8,4)(8.10,4.90)
\psline[linecolor=honecolor](8,4)(9,5)
\pscircle*[linecolor=tauzerocolor](8,4){\cirrad}
\psline[linecolor=hiddentaucolor]{->}(5,5)(4.30,5)
\psline[linecolor=hiddentaucolor](5,5)(6,6)
\pscircle*[linecolor=hiddentaucolor](5,5){\cirrad}
\psline[linecolor=taufourextncolor](6,5)(5,5)
\psline[linecolor=honecolor](6,5)(7,6)
\pscircle*[linecolor=tauzerocolor](6,5){\cirrad}
\psline[linecolor=pcolor](7,5)(6,5)
\psline[linecolor=honecolor](7,5)(8,6)
\pscircle*[linecolor=tauzerocolor](7,5){\cirrad}
\psline[linecolor=hzerocolor]{->}(8.10,4.90)(8.10,5.60)
\pscircle*[linecolor=tauzerocolor](8.10,4.90){\cirrad}
\psline[linecolor=pcolor](7.90,5.10)(7,5)
\psline[linecolor=honecolor](7.90,5.10)(9,6)
\pscircle*[linecolor=tauzerocolor](7.90,5.10){\cirrad}
\psline[linecolor=pcolor](9,5)(7.90,5.10)
\psline[linecolor=honecolor](9,5)(10,6)
\pscircle*[linecolor=tauzerocolor](9,5){\cirrad}
\psline[linecolor=hiddentaucolor]{->}(6,6)(5.30,6)
\psline[linecolor=hiddentwotaucolor](6,6)(7,7)
\pscircle*[linecolor=hiddentaucolor](6,6){\cirrad}
\psline[linecolor=taufourextncolor](7,6)(6,6)
\psline[linecolor=hiddentaucolor](7,6)(8,7)
\pscircle*[linecolor=tauzerocolor](7,6){\cirrad}
\psline[linecolor=pcolor](8,6)(7,6)
\psline[linecolor=hiddentaucolor](8,6)(9,7)
\pscircle*[linecolor=tauzerocolor](8,6){\cirrad}
\psline[linecolor=pcolor](9,6)(8,6)
\psline[linecolor=hiddentaucolor](9,6)(10,7)
\pscircle*[linecolor=tauzerocolor](9,6){\cirrad}
\psline[linecolor=pcolor](10,6)(9,6)
\psline[linecolor=hiddentaucolor](10,6)(11,7)
\pscircle*[linecolor=tauzerocolor](10,6){\cirrad}
\psline[linecolor=hiddentwotaucolor]{->}(7,7)(6.30,7)
\psline[linecolor=hiddentwotaucolor](7,7)(8,8)
\pscircle*[linecolor=hiddentwotaucolor](7,7){\cirrad}
\psline[linecolor=taufourextncolor](8,7)(7,7)
\psline[linecolor=hiddentaucolor](8,7)(9,8)
\pscircle*[linecolor=hiddentaucolor](8,7){\cirrad}
\psline[linecolor=hiddentaucolor](9,7)(8,7)
\psline[linecolor=hiddentaucolor](9,7)(10,8)
\pscircle*[linecolor=hiddentaucolor](9,7){\cirrad}
\psline[linecolor=hiddentaucolor](10,7)(9,7)
\psline[linecolor=hiddentaucolor](10,7)(11,8)
\pscircle*[linecolor=hiddentaucolor](10,7){\cirrad}
\psline[linecolor=hiddentaucolor](11,7)(10,7)
\psline[linecolor=hiddentaucolor](11,7)(11.90,8.10)
\pscircle*[linecolor=hiddentaucolor](11,7){\cirrad}
\uput{\labelrad}[-90](12,7){$\tau^{2} a b$}
\psline[linecolor=taufourextncolor](12,7)(11,7)
\psline[linecolor=hzerocolor](12,7)(12.10,7.90)
\psline[linecolor=honecolor,linestyle=dashed, dash=3pt 2pt](12,7)(13,8)
\pscircle*[linecolor=tauzerocolor](12,7){\cirrad}
\psline[linecolor=hiddentwotaucolor]{->}(8,8)(7.30,8)
\psline[linecolor=hiddentwotaucolor]{->}(8,8)(8.70,8.70)
\pscircle*[linecolor=hiddentwotaucolor](8,8){\cirrad}
\psline[linecolor=taufourextncolor](9,8)(8,8)
\psline[linecolor=hiddentaucolor]{->}(9,8)(9.70,8.70)
\pscircle*[linecolor=hiddentaucolor](9,8){\cirrad}
\psline[linecolor=hiddentaucolor](10,8)(9,8)
\psline[linecolor=hiddentaucolor]{->}(10,8)(10.70,8.70)
\pscircle*[linecolor=hiddentaucolor](10,8){\cirrad}
\psline[linecolor=hiddentaucolor](11,8)(10,8)
\psline[linecolor=hiddentaucolor]{->}(11,8)(11.70,8.70)
\pscircle*[linecolor=hiddentaucolor](11,8){\cirrad}
\psline[linecolor=hzerocolor]{->}(12.10,7.90)(12.10,8.60)
\pscircle*[linecolor=tauzerocolor](12.10,7.90){\cirrad}
\psline[linecolor=hiddentaucolor](11.90,8.10)(11,8)
\psline[linecolor=hiddentaucolor]{->}(11.90,8.10)(12.60,8.80)
\pscircle*[linecolor=hiddentaucolor](11.90,8.10){\cirrad}
\psline[linecolor=taufourextncolor](13,8)(11.90,8.10)
\psline[linecolor=honecolor]{->}(13,8)(13.70,8.70)
\pscircle*[linecolor=tauzerocolor](13,8){\cirrad}
\psline[linecolor=pcolor](14,8)(13,8)
\psline[linecolor=honecolor]{->}(14,8)(14.70,8.70)
\pscircle*[linecolor=tauzerocolor](14,8){\cirrad}
\psline[linecolor=pcolor](15,8)(14,8)
\psline[linecolor=honecolor]{->}(15,8)(15.70,8.70)
\pscircle*[linecolor=tauzerocolor](15,8){\cirrad}
\uput{\labelrad}[-90](16,8){$b^{2}$}
\psline[linecolor=pcolor](16,8)(15,8)
\psline[linecolor=honecolor]{->}(16,8)(16.70,8.70)
\pscircle*[linecolor=tauzerocolor](16,8){\cirrad}
\uput{\labelrad}[0](8,0.5){$\F_2[\tau^4], \text{two hidden } \tau^4 \text{ extensions}$}
\pscircle*[linecolor=hiddentwotaucolor](8,0.5){\cirrad}
\uput{\labelrad}[0](8,1){$\F_2[\tau^4], \text{hidden } \tau^4 \text{ extension}$}
\pscircle*[linecolor=hiddentaucolor](8,1){\cirrad}
\uput{\labelrad}[0](8,1.5){$\F_2[\tau^4]$}
\pscircle*[linecolor=tauzerocolor](8,1.5){\cirrad}
\end{pspicture}
\vfill

\begin{pspicture}(-1,-1)(17,9)
\psgrid[unit=2,gridcolor=gridline,subgriddiv=0,gridlabelcolor=white](0,0)(8,4)
\scriptsize
\rput(0,-1){0}
\rput(2,-1){2}
\rput(4,-1){4}
\rput(6,-1){6}
\rput(8,-1){8}
\rput(10,-1){10}
\rput(12,-1){12}
\rput(14,-1){14}
\rput(16,-1){16}

\rput(-1,0){0}
\rput(-1,2){2}
\rput(-1,4){4}
\rput(-1,6){6}
\rput(-1,8){8}

\uput[0](0,8.2){\textsc{$\Ext_{\cA^{\R}(1)}$, $mw \equiv 1 \pmod 4$}}
\psline[linecolor=honecolor](0,1)(1,2)
\pscircle*[linecolor=tauzerocolor](0,1){\cirrad}
\uput{\labelrad}[-90](1,1){$\tau h_1$}
\psline[linecolor=pcolor](1,1)(0,1)
\psline[linecolor=hzerocolor,linestyle=dashed, dash=3pt 2pt](1,1)(1,2)
\psline[linecolor=honecolor](1,1)(2,2)
\pscircle*[linecolor=tauzerocolor](1,1){\cirrad}
\pscircle*[linecolor=tauzerocolor](1,2){\cirrad}
\psline[linecolor=pcolor](2,2)(1,2)
\pscircle*[linecolor=tauzerocolor](2,2){\cirrad}
\psline[linecolor=honecolor](8,5)(9,6)
\pscircle*[linecolor=tauzerocolor](8,5){\cirrad}
\uput{\labelrad}[-90](9,5){$\tau h_1 b$}
\psline[linecolor=pcolor](9,5)(8,5)
\psline[linecolor=hzerocolor,linestyle=dashed, dash=3pt 2pt](9,5)(9,6)
\psline[linecolor=honecolor](9,5)(10,6)
\pscircle*[linecolor=tauzerocolor](9,5){\cirrad}
\pscircle*[linecolor=tauzerocolor](9,6){\cirrad}
\psline[linecolor=pcolor](10,6)(9,6)
\pscircle*[linecolor=tauzerocolor](10,6){\cirrad}
\uput{\labelrad}[0](1,7){$\F_2[\tau^4]$}
\pscircle*[linecolor=tauzerocolor](1,7){\cirrad}
\end{pspicture}
\begin{pspicture}(-1,-1)(17,9)
\psgrid[unit=2,gridcolor=gridline,subgriddiv=0,gridlabelcolor=white](0,0)(8,4)
\scriptsize
\rput(0,-1){0}
\rput(2,-1){2}
\rput(4,-1){4}
\rput(6,-1){6}
\rput(8,-1){8}
\rput(10,-1){10}
\rput(12,-1){12}
\rput(14,-1){14}
\rput(16,-1){16}

\rput(-1,0){0}
\rput(-1,2){2}
\rput(-1,4){4}
\rput(-1,6){6}
\rput(-1,8){8}

\uput[0](0,8.2){\textsc{$\Ext_{\cA^{\R}(1)}$, $mw \equiv 2 \pmod 4$}}
\uput{\labelrad}[-90](0,1){$\tau^{2} h_0$}
\psline[linecolor=hzerocolor]{->}(0,1)(0,1.70)
\psline[linecolor=honecolor,linestyle=dashed, dash=3pt 2pt](0,1)(1,2)
\pscircle*[linecolor=tauzerocolor](0,1){\cirrad}
\psline[linecolor=honecolor,linestyle=dashed, dash=3pt 2pt](1,2)(2,3)
\pscircle*[linecolor=tauzerocolor](1,2){\cirrad}
\uput{\labelrad}[-90](2,2){$\tau^{2} h_1^{2}$}
\psline[linecolor=pcolor](2,2)(1,2)
\psline[linecolor=hzerocolor,linestyle=dashed, dash=3pt 2pt](2,2)(2,3)
\psline[linecolor=honecolor,linestyle=dashed, dash=3pt 2pt](2,2)(3,3)
\pscircle*[linecolor=tauzerocolor](2,2){\cirrad}
\pscircle*[linecolor=tauzerocolor](2,3){\cirrad}
\psline[linecolor=pcolor](3,3)(2,3)
\pscircle*[linecolor=tauzerocolor](3,3){\cirrad}
\uput{\labelrad}[-90](4,3){$a$}
\psline[linecolor=pcolor](4,3)(3,3)
\psline[linecolor=hzerocolor](4,3)(4,4)
\pscircle*[linecolor=tauzerocolor](4,3){\cirrad}
\psline[linecolor=hzerocolor]{->}(4,4)(4,4.70)
\pscircle*[linecolor=tauzerocolor](4,4){\cirrad}
\uput{\labelrad}[-90](8,5){$\tau^{2} h_0 b$}
\psline[linecolor=hzerocolor]{->}(8,5)(8,5.70)
\psline[linecolor=honecolor,linestyle=dashed, dash=3pt 2pt](8,5)(9,6)
\pscircle*[linecolor=tauzerocolor](8,5){\cirrad}
\psline[linecolor=honecolor,linestyle=dashed, dash=3pt 2pt](9,6)(10,7)
\pscircle*[linecolor=tauzerocolor](9,6){\cirrad}
\uput{\labelrad}[-90](10,6){$\tau^{2} h_1^{2} b$}
\psline[linecolor=pcolor](10,6)(9,6)
\psline[linecolor=hzerocolor,linestyle=dashed, dash=3pt 2pt](10,6)(10,7)
\psline[linecolor=honecolor,linestyle=dashed, dash=3pt 2pt](10,6)(11,7)
\pscircle*[linecolor=tauzerocolor](10,6){\cirrad}
\pscircle*[linecolor=tauzerocolor](10,7){\cirrad}
\psline[linecolor=pcolor](11,7)(10,7)
\pscircle*[linecolor=tauzerocolor](11,7){\cirrad}
\uput{\labelrad}[-90](12,7){$a b$}
\psline[linecolor=pcolor](12,7)(11,7)
\psline[linecolor=hzerocolor](12,7)(12,8)
\pscircle*[linecolor=tauzerocolor](12,7){\cirrad}
\psline[linecolor=hzerocolor]{->}(12,8)(12,8.70)
\pscircle*[linecolor=tauzerocolor](12,8){\cirrad}
\uput{\labelrad}[0](1,7){$\F_2[\tau^4]$}
\pscircle*[linecolor=tauzerocolor](1,7){\cirrad}
\end{pspicture}
\vfill
\label{E+end}

\newpage
\centerline{Bockstein $E^{-}$ charts for $\cA^{C_2}(1)$}

\label{E-start}
\begin{pspicture}(-1,-1)(17,9)
\psgrid[unit=2,gridcolor=gridline,subgriddiv=0,gridlabelcolor=white](0,0)(8,4)
\scriptsize
\rput(0,-1){0}
\rput(2,-1){2}
\rput(4,-1){4}
\rput(6,-1){6}
\rput(8,-1){8}
\rput(10,-1){10}
\rput(12,-1){12}
\rput(14,-1){14}
\rput(16,-1){16}

\rput(-1,0){0}
\rput(-1,2){2}
\rput(-1,4){4}
\rput(-1,6){6}
\rput(-1,8){8}

\uput[0](0,8.2){\textsc{Bockstein $E_1^-$-page}}
\uput{\labelrad}[-90](0,0){$\frac{\gamma}{\tau}$}
\psline[linecolor=pcolor]{->}(0,0)(0.70,0)
\psline[linecolor=hzerocolor]{->}(0,0)(0,0.70)
\psline[linecolor=honecolor](0,0)(1,1)
\pscircle*[linecolor=tauzerocolor](0,0){\cirrad}
\psline[linecolor=pcolor]{->}(1,1)(1.70,1)
\psline[linecolor=honecolor](1,1)(2,2)
\pscircle*[linecolor=tauzerocolor](1,1){\cirrad}
\psline[linecolor=pcolor]{->}(2,2)(2.70,2)
\pscircle*[linecolor=tauzerocolor](2,2){\cirrad}
\uput{\labelrad}[-90](4,2){${Q} h_1^{3}$}
\psline[linecolor=tauonecolor]{->}(4,2)(4.70,2)
\psline[linecolor=tauonecolor](4,2)(4,3)
\psline[linecolor=tauonecolor]{->}(4,2)(4.70,2.70)
\pscircle*[linecolor=tauonecolor](4,2){\cirrad}
\uput{\labelrad}[150](4,3){$\frac{\gamma}{\tau} a$}
\psline[linecolor=pcolor]{->}(4,3)(4.70,3)
\psline[linecolor=hzerocolor]{->}(4,3)(4,3.70)
\pscircle*[linecolor=tauzerocolor](4,3){\cirrad}
\uput{\labelrad}[-90](8,4){$\frac{\gamma}{\tau} b$}
\psline[linecolor=pcolor]{->}(8,4)(8.70,4)
\psline[linecolor=hzerocolor]{->}(8,4)(8,4.70)
\psline[linecolor=honecolor](8,4)(9,5)
\pscircle*[linecolor=tauzerocolor](8,4){\cirrad}
\psline[linecolor=pcolor]{->}(9,5)(9.70,5)
\psline[linecolor=honecolor](9,5)(10,6)
\pscircle*[linecolor=tauzerocolor](9,5){\cirrad}
\psline[linecolor=pcolor]{->}(10,6)(10.70,6)
\pscircle*[linecolor=tauzerocolor](10,6){\cirrad}
\uput{\labelrad}[-90](12,6){${Q} h_1^{3} b$}
\psline[linecolor=tauonecolor]{->}(12,6)(12.70,6)
\psline[linecolor=tauonecolor](12,6)(12,7)
\psline[linecolor=tauonecolor]{->}(12,6)(12.70,6.70)
\pscircle*[linecolor=tauonecolor](12,6){\cirrad}
\uput{\labelrad}[150](12,7){$\frac{\gamma}{\tau} a b$}
\psline[linecolor=pcolor]{->}(12,7)(12.70,7)
\psline[linecolor=hzerocolor]{->}(12,7)(12,7.70)
\pscircle*[linecolor=tauzerocolor](12,7){\cirrad}
\uput{\labelrad}[-90](16,8){$\frac{\gamma}{\tau} b^{2}$}
\psline[linecolor=pcolor]{->}(16,8)(16.70,8)
\psline[linecolor=hzerocolor]{->}(16,8)(16,8.70)
\pscircle*[linecolor=tauzerocolor](16,8){\cirrad}
\uput{\labelrad}[0](1,7){$\F_2[\tau]/ \tau^\infty$}
\pscircle*[linecolor=tauzerocolor](1,7){\cirrad}
\uput{\labelrad}[0](1,6.5){$\F_2$}
\pscircle*[linecolor=tauonecolor](1,6.5){\cirrad}
\end{pspicture}
\vfill

\begin{pspicture}(-1,-1)(17,9)
\psgrid[unit=2,gridcolor=gridline,subgriddiv=0,gridlabelcolor=white](0,0)(8,4)
\scriptsize
\rput(0,-1){0}
\rput(2,-1){2}
\rput(4,-1){4}
\rput(6,-1){6}
\rput(8,-1){8}
\rput(10,-1){10}
\rput(12,-1){12}
\rput(14,-1){14}
\rput(16,-1){16}

\rput(-1,0){0}
\rput(-1,2){2}
\rput(-1,4){4}
\rput(-1,6){6}
\rput(-1,8){8}

\uput[0](0,8.2){\textsc{Bockstein $E_2^-$-page, $mw \equiv 0 \pmod 2$}}
\uput{\labelrad}[-90](0,0){$\frac{\gamma}{\tau}$}
\psline[linecolor=hzerocolor](0,0)(0,1)
\psline[linecolor=honecolor](0,0)(1,1)
\pscircle*[linecolor=tauzerocolor](0,0){\cirrad}
\psline[linecolor=hzerocolor]{->}(0,1)(0,1.70)
\pscircle*[linecolor=tauzerocolor](0,1){\cirrad}
\psline[linecolor=pcolor]{->}(1,1)(1.70,1)
\psline[linecolor=honecolor](1,1)(2,2)
\pscircle*[linecolor=tauzerocolor](1,1){\cirrad}
\psline[linecolor=pcolor]{->}(2,2)(2.70,2)
\pscircle*[linecolor=tauzerocolor](2,2){\cirrad}
\uput{\labelrad}[-90](4,2){${Q} h_1^{3}$}
\psline[linecolor=tauonecolor](4,2)(4,3)
\psline[linecolor=tauonecolor](4,2)(5,3)
\pscircle*[linecolor=tauonecolor](4,2){\cirrad}
\psline[linecolor=tauonecolor]{->}(5,3)(5.70,3)
\psline[linecolor=tauonecolor]{->}(5,3)(5.70,3.70)
\pscircle*[linecolor=tauonecolor](5,3){\cirrad}
\uput{\labelrad}[150](4,3){$\frac{\gamma}{\tau} a$}
\psline[linecolor=hzerocolor]{->}(4,3)(4,3.70)
\pscircle*[linecolor=tauzerocolor](4,3){\cirrad}
\uput{\labelrad}[-90](8,4){$\frac{\gamma}{\tau} b$}
\psline[linecolor=hzerocolor](8,4)(8,5)
\psline[linecolor=honecolor](8,4)(9,5)
\pscircle*[linecolor=tauzerocolor](8,4){\cirrad}
\psline[linecolor=hzerocolor]{->}(8,5)(8,5.70)
\pscircle*[linecolor=tauzerocolor](8,5){\cirrad}
\psline[linecolor=pcolor]{->}(9,5)(9.70,5)
\psline[linecolor=honecolor](9,5)(10,6)
\pscircle*[linecolor=tauzerocolor](9,5){\cirrad}
\psline[linecolor=pcolor]{->}(10,6)(10.70,6)
\pscircle*[linecolor=tauzerocolor](10,6){\cirrad}
\uput{\labelrad}[-90](12,6){${Q} h_1^{3} b$}
\psline[linecolor=tauonecolor](12,6)(12,7)
\psline[linecolor=tauonecolor](12,6)(13,7)
\pscircle*[linecolor=tauonecolor](12,6){\cirrad}
\psline[linecolor=tauonecolor]{->}(13,7)(13.70,7)
\psline[linecolor=tauonecolor]{->}(13,7)(13.70,7.70)
\pscircle*[linecolor=tauonecolor](13,7){\cirrad}
\uput{\labelrad}[150](12,7){$\frac{\gamma}{\tau} a b$}
\psline[linecolor=hzerocolor]{->}(12,7)(12,7.70)
\pscircle*[linecolor=tauzerocolor](12,7){\cirrad}
\uput{\labelrad}[-90](16,8){$\frac{\gamma}{\tau} b^{2}$}
\psline[linecolor=hzerocolor]{->}(16,8)(16,8.70)
\pscircle*[linecolor=tauzerocolor](16,8){\cirrad}
\uput{\labelrad}[0](1,7){$\F_2[\tau^2]/ \tau^\infty$}
\pscircle*[linecolor=tauzerocolor](1,7){\cirrad}
\uput{\labelrad}[0](1,6.5){$\F_2$}
\pscircle*[linecolor=tauonecolor](1,6.5){\cirrad}
\end{pspicture}
\begin{pspicture}(-1,-1)(17,9)
\psgrid[unit=2,gridcolor=gridline,subgriddiv=0,gridlabelcolor=white](0,0)(8,4)
\scriptsize
\rput(0,-1){0}
\rput(2,-1){2}
\rput(4,-1){4}
\rput(6,-1){6}
\rput(8,-1){8}
\rput(10,-1){10}
\rput(12,-1){12}
\rput(14,-1){14}
\rput(16,-1){16}

\rput(-1,0){0}
\rput(-1,2){2}
\rput(-1,4){4}
\rput(-1,6){6}
\rput(-1,8){8}

\uput[0](0,8.2){\textsc{Bockstein $E_2^-$-page, $mw \equiv 1 \pmod 2$}}
\uput{\labelrad}[-90](0,0){$\frac{\gamma}{\tau^{2}}$}
\psline[linecolor=pcolor]{->}(0,0)(0.70,0)
\psline[linecolor=honecolor](0,0)(1,1)
\pscircle*[linecolor=tauzerocolor](0,0){\cirrad}
\psline[linecolor=pcolor]{->}(1,1)(1.70,1)
\psline[linecolor=honecolor](1,1)(2,2)
\pscircle*[linecolor=tauzerocolor](1,1){\cirrad}
\psline[linecolor=pcolor]{->}(2,2)(2.70,2)
\pscircle*[linecolor=tauzerocolor](2,2){\cirrad}
\uput{\labelrad}[-90](4,3){$\frac{\gamma}{\tau^{4}} a$}
\psline[linecolor=pcolor]{->}(4,3)(4.70,3)
\pscircle*[linecolor=tauzerocolor](4,3){\cirrad}
\uput{\labelrad}[-90](8,4){$\frac{\gamma}{\tau^{2}} b$}
\psline[linecolor=pcolor]{->}(8,4)(8.70,4)
\psline[linecolor=honecolor](8,4)(9,5)
\pscircle*[linecolor=tauzerocolor](8,4){\cirrad}
\psline[linecolor=pcolor]{->}(9,5)(9.70,5)
\psline[linecolor=honecolor](9,5)(10,6)
\pscircle*[linecolor=tauzerocolor](9,5){\cirrad}
\psline[linecolor=pcolor]{->}(10,6)(10.70,6)
\pscircle*[linecolor=tauzerocolor](10,6){\cirrad}
\uput{\labelrad}[-90](12,7){$\frac{\gamma}{\tau^{4}} a b$}
\psline[linecolor=pcolor]{->}(12,7)(12.70,7)
\pscircle*[linecolor=tauzerocolor](12,7){\cirrad}
\uput{\labelrad}[-90](16,8){$\frac{\gamma}{\tau^{2}} b^{2}$}
\psline[linecolor=pcolor]{->}(16,8)(16.70,8)
\pscircle*[linecolor=tauzerocolor](16,8){\cirrad}
\uput{\labelrad}[0](1,7){$\F_2[\tau^2]/ \tau^\infty$}
\pscircle*[linecolor=tauzerocolor](1,7){\cirrad}
\end{pspicture}
\vfill

\newpage
\centerline{Bockstein $E^{-}$ charts for $\cA^{C_2}(1)$}

\begin{pspicture}(-1,-1)(17,9)
\psgrid[unit=2,gridcolor=gridline,subgriddiv=0,gridlabelcolor=white](0,0)(8,4)
\scriptsize
\rput(0,-1){0}
\rput(2,-1){2}
\rput(4,-1){4}
\rput(6,-1){6}
\rput(8,-1){8}
\rput(10,-1){10}
\rput(12,-1){12}
\rput(14,-1){14}
\rput(16,-1){16}

\rput(-1,0){0}
\rput(-1,2){2}
\rput(-1,4){4}
\rput(-1,6){6}
\rput(-1,8){8}

\uput[0](0,8.2){\textsc{Bockstein $E_3^-=E_4^-$-page, $mw \equiv 0 \pmod 4$}}
\uput{\labelrad}[-90](0,0){$\frac{\gamma}{\tau^{3}}$}
\psline[linecolor=hzerocolor]{->}(0,0)(0,0.70)
\pscircle*[linecolor=tauzerocolor](0,0){\cirrad}
\uput{\labelrad}[-90](4,2){${Q} h_1^{3}$}
\psline[linecolor=tauonecolor](4,2)(4,3)
\psline[linecolor=tauonecolor](4,2)(5,3)
\pscircle*[linecolor=tauonecolor](4,2){\cirrad}
\psline[linecolor=tauonecolor]{->}(5,3)(5.70,3)
\psline[linecolor=tauonecolor]{->}(5,3)(5.70,3.70)
\pscircle*[linecolor=tauonecolor](5,3){\cirrad}
\uput{\labelrad}[150](4,3){$\frac{\gamma}{\tau} a$}
\psline[linecolor=hzerocolor]{->}(4,3)(4,3.70)
\pscircle*[linecolor=tauzerocolor](4,3){\cirrad}
\uput{\labelrad}[-90](8,4){$\frac{\gamma}{\tau^{3}} b$}
\psline[linecolor=hzerocolor]{->}(8,4)(8,4.70)
\pscircle*[linecolor=tauzerocolor](8,4){\cirrad}
\uput{\labelrad}[-90](12,6){${Q} h_1^{3} b$}
\psline[linecolor=tauonecolor](12,6)(12,7)
\psline[linecolor=tauonecolor](12,6)(13,7)
\pscircle*[linecolor=tauonecolor](12,6){\cirrad}
\psline[linecolor=tauonecolor]{->}(13,7)(13.70,7)
\psline[linecolor=tauonecolor]{->}(13,7)(13.70,7.70)
\pscircle*[linecolor=tauonecolor](13,7){\cirrad}
\uput{\labelrad}[150](12,7){$\frac{\gamma}{\tau} a b$}
\psline[linecolor=hzerocolor]{->}(12,7)(12,7.70)
\pscircle*[linecolor=tauzerocolor](12,7){\cirrad}
\uput{\labelrad}[-90](16,8){$\frac{\gamma}{\tau^{3}} b^{2}$}
\psline[linecolor=hzerocolor]{->}(16,8)(16,8.70)
\pscircle*[linecolor=tauzerocolor](16,8){\cirrad}
\uput{\labelrad}[0](1,7){$\F_2[\tau^4]/ \tau^\infty$}
\pscircle*[linecolor=tauzerocolor](1,7){\cirrad}
\uput{\labelrad}[0](1,6.5){$\F_2$}
\pscircle*[linecolor=tauonecolor](1,6.5){\cirrad}
\end{pspicture}
\begin{pspicture}(-1,-1)(17,9)
\psgrid[unit=2,gridcolor=gridline,subgriddiv=0,gridlabelcolor=white](0,0)(8,4)
\scriptsize
\rput(0,-1){0}
\rput(2,-1){2}
\rput(4,-1){4}
\rput(6,-1){6}
\rput(8,-1){8}
\rput(10,-1){10}
\rput(12,-1){12}
\rput(14,-1){14}
\rput(16,-1){16}

\rput(-1,0){0}
\rput(-1,2){2}
\rput(-1,4){4}
\rput(-1,6){6}
\rput(-1,8){8}

\uput[0](0,8.2){\textsc{Bockstein $E_3^-$-page, $mw \equiv 1 \pmod 4$}}
\uput{\labelrad}[-90](0,0){$\frac{\gamma}{\tau^{2}}$}
\psline[linecolor=pcolor](0,0)(1,0)
\psline[linecolor=honecolor](0,0)(1,1)
\pscircle*[linecolor=tauzerocolor](0,0){\cirrad}
\psline[linecolor=honecolor](1,0)(2,1)
\pscircle*[linecolor=tauzerocolor](1,0){\cirrad}
\psline[linecolor=pcolor](1,1)(2,1)
\pscircle*[linecolor=tauzerocolor](1,1){\cirrad}
\pscircle*[linecolor=tauzerocolor](2,1){\cirrad}
\uput{\labelrad}[-90](4,3){$\frac{\gamma}{\tau^{4}} a$}
\psline[linecolor=pcolor]{->}(4,3)(4.70,3)
\pscircle*[linecolor=tauzerocolor](4,3){\cirrad}
\uput{\labelrad}[-90](8,4){$\frac{\gamma}{\tau^{2}} b$}
\psline[linecolor=pcolor](8,4)(9,4)
\psline[linecolor=honecolor](8,4)(9,5)
\pscircle*[linecolor=tauzerocolor](8,4){\cirrad}
\psline[linecolor=honecolor](9,4)(10,5)
\pscircle*[linecolor=tauzerocolor](9,4){\cirrad}
\psline[linecolor=pcolor](9,5)(10,5)
\pscircle*[linecolor=tauzerocolor](9,5){\cirrad}
\pscircle*[linecolor=tauzerocolor](10,5){\cirrad}
\uput{\labelrad}[-90](12,7){$\frac{\gamma}{\tau^{4}} a b$}
\psline[linecolor=pcolor]{->}(12,7)(12.70,7)
\pscircle*[linecolor=tauzerocolor](12,7){\cirrad}
\uput{\labelrad}[-90](16,8){$\frac{\gamma}{\tau^{2}} b^{2}$}
\pscircle*[linecolor=tauzerocolor](16,8){\cirrad}
\uput{\labelrad}[0](1,7){$\F_2[\tau^4]/ \tau^\infty$}
\pscircle*[linecolor=tauzerocolor](1,7){\cirrad}
\end{pspicture}
\vfill

\begin{pspicture}(-1,-1)(17,9)
\psgrid[unit=2,gridcolor=gridline,subgriddiv=0,gridlabelcolor=white](0,0)(8,4)
\scriptsize
\rput(0,-1){0}
\rput(2,-1){2}
\rput(4,-1){4}
\rput(6,-1){6}
\rput(8,-1){8}
\rput(10,-1){10}
\rput(12,-1){12}
\rput(14,-1){14}
\rput(16,-1){16}

\rput(-1,0){0}
\rput(-1,2){2}
\rput(-1,4){4}
\rput(-1,6){6}
\rput(-1,8){8}

\uput[0](0,8.2){\textsc{Bockstein $E_3^-$-page, $mw \equiv 2 \pmod 4$}}
\uput{\labelrad}[-90](0,0){$\frac{\gamma}{\tau}$}
\psline[linecolor=hzerocolor](0,0)(0,1)
\psline[linecolor=honecolor](0,0)(1,1)
\pscircle*[linecolor=tauzerocolor](0,0){\cirrad}
\psline[linecolor=hzerocolor]{->}(0,1)(0,1.70)
\pscircle*[linecolor=tauzerocolor](0,1){\cirrad}
\psline[linecolor=pcolor](1,1)(2,1)
\psline[linecolor=honecolor](1,1)(2,2)
\pscircle*[linecolor=tauzerocolor](1,1){\cirrad}
\psline[linecolor=honecolor](2,1)(3,2)
\pscircle*[linecolor=tauzerocolor](2,1){\cirrad}
\psline[linecolor=pcolor](2,2)(3,2)
\pscircle*[linecolor=tauzerocolor](2,2){\cirrad}
\psline[linecolor=pcolor]{->}(3,2)(3.70,2)
\pscircle*[linecolor=tauzerocolor](3,2){\cirrad}
\uput{\labelrad}[150](4,3){$\frac{\gamma}{\tau^{3}} a$}
\psline[linecolor=hzerocolor]{->}(4,3)(4,3.70)
\pscircle*[linecolor=tauzerocolor](4,3){\cirrad}
\uput{\labelrad}[-90](8,4){$\frac{\gamma}{\tau} b$}
\psline[linecolor=hzerocolor](8,4)(8,5)
\psline[linecolor=honecolor](8,4)(9,5)
\pscircle*[linecolor=tauzerocolor](8,4){\cirrad}
\psline[linecolor=hzerocolor]{->}(8,5)(8,5.70)
\pscircle*[linecolor=tauzerocolor](8,5){\cirrad}
\psline[linecolor=pcolor](9,5)(10,5)
\psline[linecolor=honecolor](9,5)(10,6)
\pscircle*[linecolor=tauzerocolor](9,5){\cirrad}
\psline[linecolor=honecolor](10,5)(11,6)
\pscircle*[linecolor=tauzerocolor](10,5){\cirrad}
\psline[linecolor=pcolor](10,6)(11,6)
\pscircle*[linecolor=tauzerocolor](10,6){\cirrad}
\psline[linecolor=pcolor]{->}(11,6)(11.70,6)
\pscircle*[linecolor=tauzerocolor](11,6){\cirrad}
\uput{\labelrad}[150](12,7){$\frac{\gamma}{\tau^{3}} a b$}
\psline[linecolor=hzerocolor]{->}(12,7)(12,7.70)
\pscircle*[linecolor=tauzerocolor](12,7){\cirrad}
\uput{\labelrad}[-90](16,8){$\frac{\gamma}{\tau} b^{2}$}
\psline[linecolor=hzerocolor]{->}(16,8)(16,8.70)
\pscircle*[linecolor=tauzerocolor](16,8){\cirrad}
\uput{\labelrad}[0](1,7){$\F_2[\tau^4]/ \tau^\infty$}
\pscircle*[linecolor=tauzerocolor](1,7){\cirrad}
\end{pspicture}
\begin{pspicture}(-1,-1)(17,9)
\psgrid[unit=2,gridcolor=gridline,subgriddiv=0,gridlabelcolor=white](0,0)(8,4)
\scriptsize
\rput(0,-1){0}
\rput(2,-1){2}
\rput(4,-1){4}
\rput(6,-1){6}
\rput(8,-1){8}
\rput(10,-1){10}
\rput(12,-1){12}
\rput(14,-1){14}
\rput(16,-1){16}

\rput(-1,0){0}
\rput(-1,2){2}
\rput(-1,4){4}
\rput(-1,6){6}
\rput(-1,8){8}

\uput[0](0,8.2){\textsc{Bockstein $E_3^-=E_4^-$-page, $mw \equiv 3 \pmod 4$}}
\uput{\labelrad}[-90](0,0){$\frac{\gamma}{\tau^{4}}$}
\psline[linecolor=pcolor]{->}(0,0)(0.70,0)
\psline[linecolor=honecolor](0,0)(1,1)
\pscircle*[linecolor=tauzerocolor](0,0){\cirrad}
\psline[linecolor=pcolor]{->}(1,1)(1.70,1)
\psline[linecolor=honecolor](1,1)(2,2)
\pscircle*[linecolor=tauzerocolor](1,1){\cirrad}
\psline[linecolor=pcolor]{->}(2,2)(2.70,2)
\pscircle*[linecolor=tauzerocolor](2,2){\cirrad}
\uput{\labelrad}[-90](4,3){$\frac{\gamma}{\tau^{6}} a$}
\psline[linecolor=pcolor]{->}(4,3)(4.70,3)
\pscircle*[linecolor=tauzerocolor](4,3){\cirrad}
\uput{\labelrad}[-90](8,4){$\frac{\gamma}{\tau^{4}} b$}
\psline[linecolor=pcolor]{->}(8,4)(8.70,4)
\psline[linecolor=honecolor](8,4)(9,5)
\pscircle*[linecolor=tauzerocolor](8,4){\cirrad}
\psline[linecolor=pcolor]{->}(9,5)(9.70,5)
\psline[linecolor=honecolor](9,5)(10,6)
\pscircle*[linecolor=tauzerocolor](9,5){\cirrad}
\psline[linecolor=pcolor]{->}(10,6)(10.70,6)
\pscircle*[linecolor=tauzerocolor](10,6){\cirrad}
\uput{\labelrad}[-90](12,7){$\frac{\gamma}{\tau^{6}} a b$}
\psline[linecolor=pcolor]{->}(12,7)(12.70,7)
\pscircle*[linecolor=tauzerocolor](12,7){\cirrad}
\uput{\labelrad}[-90](16,8){$\frac{\gamma}{\tau^{4}} b^{2}$}
\psline[linecolor=pcolor]{->}(16,8)(16.70,8)
\pscircle*[linecolor=tauzerocolor](16,8){\cirrad}
\uput{\labelrad}[0](1,7){$\F_2[\tau^4]/ \tau^\infty$}
\pscircle*[linecolor=tauzerocolor](1,7){\cirrad}
\end{pspicture}
\vfill

\newpage
\centerline{Bockstein $E^{-}$ charts for $\cA^{C_2}(1)$}

\begin{pspicture}(-1,-1)(17,9)
\psgrid[unit=2,gridcolor=gridline,subgriddiv=0,gridlabelcolor=white](0,0)(8,4)
\scriptsize
\rput(0,-1){0}
\rput(2,-1){2}
\rput(4,-1){4}
\rput(6,-1){6}
\rput(8,-1){8}
\rput(10,-1){10}
\rput(12,-1){12}
\rput(14,-1){14}
\rput(16,-1){16}

\rput(-1,0){0}
\rput(-1,2){2}
\rput(-1,4){4}
\rput(-1,6){6}
\rput(-1,8){8}

\uput[0](0,8.2){\textsc{Bockstein $E_4^- = E_\infty^-$-page, $mw \equiv 1 \pmod 4$}}
\uput{\labelrad}[-90](0,0){$\frac{\gamma}{\tau^{2}}$}
\psline[linecolor=pcolor](0,0)(1,0)
\psline[linecolor=honecolor](0,0)(1,1)
\pscircle*[linecolor=tauzerocolor](0,0){\cirrad}
\psline[linecolor=honecolor](1,0)(2,1)
\pscircle*[linecolor=tauzerocolor](1,0){\cirrad}
\psline[linecolor=pcolor](1,1)(2,1)
\pscircle*[linecolor=tauzerocolor](1,1){\cirrad}
\pscircle*[linecolor=tauzerocolor](2,1){\cirrad}
\uput{\labelrad}[-90](8,4){$\frac{\gamma}{\tau^{2}} b$}
\psline[linecolor=pcolor](8,4)(9,4)
\psline[linecolor=honecolor](8,4)(9,5)
\pscircle*[linecolor=tauzerocolor](8,4){\cirrad}
\psline[linecolor=honecolor](9,4)(10,5)
\pscircle*[linecolor=tauzerocolor](9,4){\cirrad}
\psline[linecolor=pcolor](9,5)(10,5)
\pscircle*[linecolor=tauzerocolor](9,5){\cirrad}
\pscircle*[linecolor=tauzerocolor](10,5){\cirrad}
\uput{\labelrad}[-90](16,8){$\frac{\gamma}{\tau^{2}} b^{2}$}
\pscircle*[linecolor=tauzerocolor](16,8){\cirrad}
\uput{\labelrad}[0](1,7){$\F_2[\tau^4]/ \tau^\infty$}
\pscircle*[linecolor=tauzerocolor](1,7){\cirrad}
\end{pspicture}
\begin{pspicture}(-1,-1)(17,9)
\psgrid[unit=2,gridcolor=gridline,subgriddiv=0,gridlabelcolor=white](0,0)(8,4)
\scriptsize
\rput(0,-1){0}
\rput(2,-1){2}
\rput(4,-1){4}
\rput(6,-1){6}
\rput(8,-1){8}
\rput(10,-1){10}
\rput(12,-1){12}
\rput(14,-1){14}
\rput(16,-1){16}

\rput(-1,0){0}
\rput(-1,2){2}
\rput(-1,4){4}
\rput(-1,6){6}
\rput(-1,8){8}

\uput[0](0,8.2){\textsc{Bockstein $E_4^- = E_\infty^-$-page, $mw \equiv 2 \pmod 4$}}
\uput{\labelrad}[-90](0,0){$\frac{\gamma}{\tau}$}
\psline[linecolor=hzerocolor](0,0)(0,1)
\psline[linecolor=honecolor](0,0)(1,1)
\pscircle*[linecolor=tauzerocolor](0,0){\cirrad}
\psline[linecolor=hzerocolor]{->}(0,1)(0,1.70)
\pscircle*[linecolor=tauzerocolor](0,1){\cirrad}
\psline[linecolor=pcolor](1,1)(2,1)
\psline[linecolor=honecolor](1,1)(2,2)
\pscircle*[linecolor=tauzerocolor](1,1){\cirrad}
\psline[linecolor=honecolor](2,1)(3,2)
\pscircle*[linecolor=tauzerocolor](2,1){\cirrad}
\psline[linecolor=pcolor](2,2)(3,2)
\pscircle*[linecolor=tauzerocolor](2,2){\cirrad}
\psline[linecolor=pcolor](3,2)(4,2)
\pscircle*[linecolor=tauzerocolor](3,2){\cirrad}
\pscircle*[linecolor=tauzerocolor](4,2){\cirrad}
\uput{\labelrad}[150](4,3){$\frac{\gamma}{\tau^{3}} a$}
\psline[linecolor=hzerocolor]{->}(4,3)(4,3.70)
\pscircle*[linecolor=tauzerocolor](4,3){\cirrad}
\uput{\labelrad}[-90](8,4){$\frac{\gamma}{\tau} b$}
\psline[linecolor=hzerocolor](8,4)(8,5)
\psline[linecolor=honecolor](8,4)(9,5)
\pscircle*[linecolor=tauzerocolor](8,4){\cirrad}
\psline[linecolor=hzerocolor]{->}(8,5)(8,5.70)
\pscircle*[linecolor=tauzerocolor](8,5){\cirrad}
\psline[linecolor=pcolor](9,5)(10,5)
\psline[linecolor=honecolor](9,5)(10,6)
\pscircle*[linecolor=tauzerocolor](9,5){\cirrad}
\psline[linecolor=honecolor](10,5)(11,6)
\pscircle*[linecolor=tauzerocolor](10,5){\cirrad}
\psline[linecolor=pcolor](10,6)(11,6)
\pscircle*[linecolor=tauzerocolor](10,6){\cirrad}
\psline[linecolor=pcolor](11,6)(12,6)
\pscircle*[linecolor=tauzerocolor](11,6){\cirrad}
\pscircle*[linecolor=tauzerocolor](12,6){\cirrad}
\uput{\labelrad}[150](12,7){$\frac{\gamma}{\tau^{3}} a b$}
\psline[linecolor=hzerocolor]{->}(12,7)(12,7.70)
\pscircle*[linecolor=tauzerocolor](12,7){\cirrad}
\uput{\labelrad}[-90](16,8){$\frac{\gamma}{\tau} b^{2}$}
\psline[linecolor=hzerocolor]{->}(16,8)(16,8.70)
\pscircle*[linecolor=tauzerocolor](16,8){\cirrad}
\uput{\labelrad}[0](1,7){$\F_2[\tau^4]/ \tau^\infty$}
\pscircle*[linecolor=tauzerocolor](1,7){\cirrad}
\end{pspicture}
\vfill

\begin{pspicture}(-1,-1)(17,9)
\psgrid[unit=2,gridcolor=gridline,subgriddiv=0,gridlabelcolor=white](0,0)(8,4)
\scriptsize
\rput(0,-1){0}
\rput(2,-1){2}
\rput(4,-1){4}
\rput(6,-1){6}
\rput(8,-1){8}
\rput(10,-1){10}
\rput(12,-1){12}
\rput(14,-1){14}
\rput(16,-1){16}

\rput(-1,0){0}
\rput(-1,2){2}
\rput(-1,4){4}
\rput(-1,6){6}
\rput(-1,8){8}

\uput[0](0,8.2){\textsc{Bockstein $E_5^-$-page, $mw \equiv 0 \pmod 4$}}
\uput{\labelrad}[-90](0,0){$\frac{\gamma}{\tau^{3}}$}
\psline[linecolor=hzerocolor]{->}(0,0)(0,0.70)
\pscircle*[linecolor=tauzerocolor](0,0){\cirrad}
\uput{\labelrad}[0](1,6.5){$\F_2$}
\pscircle*[linecolor=tauonecolor](1,6.5){\cirrad}
\uput{\labelrad}[0](1,7){$\F_2[\tau^4]/ \tau^\infty$}
\pscircle*[linecolor=tauzerocolor](1,7){\cirrad}
\uput{\labelrad}[-90](4,2){${Q} h_1^{3}$}
\psline[linecolor=tauonecolor](4,2)(4,3)
\psline[linecolor=tauonecolor](4,2)(5,3)
\pscircle*[linecolor=tauonecolor](4,2){\cirrad}
\uput{\labelrad}[150](4,3){$\frac{\gamma}{\tau} a$}
\psline[linecolor=hzerocolor]{->}(4,3)(4,3.70)
\pscircle*[linecolor=tauzerocolor](4,3){\cirrad}
\psline[linecolor=tauonecolor](5,3)(6,3)
\psline[linecolor=tauonecolor](5,3)(6,4)
\pscircle*[linecolor=tauonecolor](5,3){\cirrad}
\psline[linecolor=tauonecolor](6,3)(7,3)
\psline[linecolor=tauonecolor](6,3)(7,4)
\pscircle*[linecolor=tauonecolor](6,3){\cirrad}
\psline[linecolor=tauonecolor](6,4)(7,4)
\psline[linecolor=tauonecolor](6,4)(7,5)
\pscircle*[linecolor=tauonecolor](6,4){\cirrad}
\psline[linecolor=tauonecolor](7,3)(8,3)
\psline[linecolor=tauonecolor](7,3)(7.90,4.10)
\pscircle*[linecolor=tauonecolor](7,3){\cirrad}
\psline[linecolor=tauonecolor](7,4)(7.90,4.10)
\psline[linecolor=tauonecolor](7,4)(8,5)
\pscircle*[linecolor=tauonecolor](7,4){\cirrad}
\psline[linecolor=tauonecolor](7,5)(8,5)
\psline[linecolor=tauonecolor](7,5)(8,6)
\pscircle*[linecolor=tauonecolor](7,5){\cirrad}
\psline[linecolor=tauonecolor](8,3)(9,4)
\pscircle*[linecolor=tauonecolor](8,3){\cirrad}
\uput{\labelrad}[-90](8.10,3.90){$\frac{\gamma}{\tau^{3}} b$}
\psline[linecolor=hzerocolor]{->}(8.10,3.90)(8.10,4.60)
\pscircle*[linecolor=tauzerocolor](8.10,3.90){\cirrad}
\psline[linecolor=tauonecolor](7.90,4.10)(9,4)
\psline[linecolor=tauonecolor](7.90,4.10)(9,5)
\pscircle*[linecolor=tauonecolor](7.90,4.10){\cirrad}
\psline[linecolor=tauonecolor](8,5)(9,5)
\psline[linecolor=tauonecolor](8,5)(9,6)
\pscircle*[linecolor=tauonecolor](8,5){\cirrad}
\psline[linecolor=tauonecolor](8,6)(9,6)
\psline[linecolor=tauonecolor]{->}(8,6)(8.70,6.70)
\pscircle*[linecolor=tauonecolor](8,6){\cirrad}
\psline[linecolor=tauonecolor](9,4)(10,5)
\pscircle*[linecolor=tauonecolor](9,4){\cirrad}
\psline[linecolor=tauonecolor](9,5)(10,5)
\psline[linecolor=tauonecolor](9,5)(10,6)
\pscircle*[linecolor=tauonecolor](9,5){\cirrad}
\psline[linecolor=tauonecolor](9,6)(10,6)
\psline[linecolor=tauonecolor]{->}(9,6)(9.70,6.70)
\pscircle*[linecolor=tauonecolor](9,6){\cirrad}
\psline[linecolor=tauonecolor](10,5)(11,6)
\pscircle*[linecolor=tauonecolor](10,5){\cirrad}
\psline[linecolor=tauonecolor](10,6)(11,6)
\psline[linecolor=tauonecolor]{->}(10,6)(10.70,6.70)
\pscircle*[linecolor=tauonecolor](10,6){\cirrad}
\psline[linecolor=tauonecolor]{->}(11,6)(11.70,6)
\psline[linecolor=tauonecolor]{->}(11,6)(11.70,6.70)
\pscircle*[linecolor=tauonecolor](11,6){\cirrad}
\uput{\labelrad}[-90](12,6){${Q} h_1^{3} b$}
\psline[linecolor=tauonecolor](12,6)(12,7)
\psline[linecolor=tauonecolor](12,6)(13,7)
\pscircle*[linecolor=tauonecolor](12,6){\cirrad}
\uput{\labelrad}[150](12,7){$\frac{\gamma}{\tau} a b$}
\psline[linecolor=hzerocolor]{->}(12,7)(12,7.70)
\pscircle*[linecolor=tauzerocolor](12,7){\cirrad}
\psline[linecolor=tauonecolor](13,7)(14,7)
\psline[linecolor=tauonecolor](13,7)(14,8)
\pscircle*[linecolor=tauonecolor](13,7){\cirrad}
\psline[linecolor=tauonecolor](14,7)(15,7)
\psline[linecolor=tauonecolor](14,7)(15,8)
\pscircle*[linecolor=tauonecolor](14,7){\cirrad}
\psline[linecolor=tauonecolor](15,7)(16,7)
\psline[linecolor=tauonecolor](15,7)(15.90,8.10)
\pscircle*[linecolor=tauonecolor](15,7){\cirrad}
\pscircle*[linecolor=tauonecolor](16,7){\cirrad}
\psline[linecolor=tauonecolor](14,8)(15,8)
\psline[linecolor=tauonecolor]{->}(14,8)(14.70,8.70)
\pscircle*[linecolor=tauonecolor](14,8){\cirrad}
\psline[linecolor=tauonecolor](15,8)(15.90,8.10)
\psline[linecolor=tauonecolor]{->}(15,8)(15.70,8.70)
\pscircle*[linecolor=tauonecolor](15,8){\cirrad}
\psline[linecolor=tauonecolor]{->}(15.90,8.10)(16.60,8.80)
\pscircle*[linecolor=tauonecolor](15.90,8.10){\cirrad}
\uput{\labelrad}[-90](16.10,7.90){$\frac{\gamma}{\tau^{3}} b^{2}$}
\psline[linecolor=hzerocolor]{->}(16.10,7.90)(16.10,8.60)
\pscircle*[linecolor=tauzerocolor](16.10,7.90){\cirrad}
\end{pspicture}
\begin{pspicture}(-1,-1)(17,9)
\psgrid[unit=2,gridcolor=gridline,subgriddiv=0,gridlabelcolor=white](0,0)(8,4)
\scriptsize
\rput(0,-1){0}
\rput(2,-1){2}
\rput(4,-1){4}
\rput(6,-1){6}
\rput(8,-1){8}
\rput(10,-1){10}
\rput(12,-1){12}
\rput(14,-1){14}
\rput(16,-1){16}

\rput(-1,0){0}
\rput(-1,2){2}
\rput(-1,4){4}
\rput(-1,6){6}
\rput(-1,8){8}

\uput[0](0,8.2){\textsc{Bockstein $E_5^-$-page, $mw \equiv 3 \pmod 4$}}
\uput{\labelrad}[-90](0,0){$\frac{\gamma}{\tau^{4}}$}
\psline[linecolor=pcolor]{->}(0,0)(0.70,0)
\psline[linecolor=honecolor](0,0)(1,1)
\pscircle*[linecolor=tauzerocolor](0,0){\cirrad}
\psline[linecolor=pcolor]{->}(1,1)(1.70,1)
\psline[linecolor=honecolor](1,1)(2,2)
\pscircle*[linecolor=tauzerocolor](1,1){\cirrad}
\uput{\labelrad}[0](1,7){$\F_2[\tau^4]/ \tau^\infty$}
\pscircle*[linecolor=tauzerocolor](1,7){\cirrad}
\psline[linecolor=pcolor]{->}(2,2)(2.70,2)
\pscircle*[linecolor=tauzerocolor](2,2){\cirrad}
\uput{\labelrad}[-90](4,3){$\frac{\gamma}{\tau^{6}} a$}
\psline[linecolor=pcolor]{->}(4,3)(4.70,3)
\pscircle*[linecolor=tauzerocolor](4,3){\cirrad}
\uput{\labelrad}[-90](8,4){$\frac{\gamma}{\tau^{8}} b$}
\psline[linecolor=pcolor]{->}(8,4)(8.70,4)
\psline[linecolor=honecolor](8,4)(9,5)
\pscircle*[linecolor=tauzerocolor](8,4){\cirrad}
\psline[linecolor=pcolor]{->}(9,5)(9.70,5)
\psline[linecolor=honecolor](9,5)(10,6)
\pscircle*[linecolor=tauzerocolor](9,5){\cirrad}
\psline[linecolor=pcolor]{->}(10,6)(10.70,6)
\pscircle*[linecolor=tauzerocolor](10,6){\cirrad}
\uput{\labelrad}[-90](12,7){$\frac{\gamma}{\tau^{6}} a b$}
\psline[linecolor=pcolor]{->}(12,7)(12.70,7)
\pscircle*[linecolor=tauzerocolor](12,7){\cirrad}
\uput{\labelrad}[-90](16,8){$\frac{\gamma}{\tau^{8}} b^{2}$}
\psline[linecolor=pcolor]{->}(16,8)(16.70,8)
\pscircle*[linecolor=tauzerocolor](16,8){\cirrad}
\end{pspicture}
\vfill

\newpage
\centerline{Bockstein $E^{-}$ charts for $\cA^{C_2}(1)$}

\begin{pspicture}(-1,-1)(17,9)
\psgrid[unit=2,gridcolor=gridline,subgriddiv=0,gridlabelcolor=white](0,0)(8,4)
\scriptsize
\rput(0,-1){0}
\rput(2,-1){2}
\rput(4,-1){4}
\rput(6,-1){6}
\rput(8,-1){8}
\rput(10,-1){10}
\rput(12,-1){12}
\rput(14,-1){14}
\rput(16,-1){16}

\rput(-1,0){0}
\rput(-1,2){2}
\rput(-1,4){4}
\rput(-1,6){6}
\rput(-1,8){8}

\uput[0](0,8.2){\textsc{Bockstein $E_6^-$-page, $mw \equiv 0 \pmod 4$}}
\uput{\labelrad}[-90](0,0){$\frac{\gamma}{\tau^{3}}$}
\psline[linecolor=hzerocolor]{->}(0,0)(0,0.70)
\pscircle*[linecolor=tauzerocolor](0,0){\cirrad}
\uput{\labelrad}[0](1,6.5){$\F_2$}
\pscircle*[linecolor=tauonecolor](1,6.5){\cirrad}
\uput{\labelrad}[0](1,7){$\F_2[\tau^4]/ \tau^\infty$}
\pscircle*[linecolor=tauzerocolor](1,7){\cirrad}
\uput{\labelrad}[-90](4,2){${Q} h_1^{3}$}
\psline[linecolor=tauonecolor](4,2)(4,3)
\psline[linecolor=tauonecolor](4,2)(5,3)
\pscircle*[linecolor=tauonecolor](4,2){\cirrad}
\uput{\labelrad}[150](4,3){$\frac{\gamma}{\tau} a$}
\psline[linecolor=hzerocolor]{->}(4,3)(4,3.70)
\pscircle*[linecolor=tauzerocolor](4,3){\cirrad}
\psline[linecolor=tauonecolor](5,3)(6,3)
\psline[linecolor=tauonecolor](5,3)(6,4)
\pscircle*[linecolor=tauonecolor](5,3){\cirrad}
\psline[linecolor=tauonecolor](6,3)(7,3)
\psline[linecolor=tauonecolor](6,3)(7,4)
\pscircle*[linecolor=tauonecolor](6,3){\cirrad}
\psline[linecolor=tauonecolor](6,4)(7,4)
\psline[linecolor=tauonecolor](6,4)(7,5)
\pscircle*[linecolor=tauonecolor](6,4){\cirrad}
\psline[linecolor=tauonecolor](7,3)(8,3)
\psline[linecolor=tauonecolor](7,3)(7.90,4.10)
\pscircle*[linecolor=tauonecolor](7,3){\cirrad}
\psline[linecolor=tauonecolor](7,4)(7.90,4.10)
\psline[linecolor=tauonecolor](7,4)(8,5)
\pscircle*[linecolor=tauonecolor](7,4){\cirrad}
\psline[linecolor=tauonecolor](7,5)(8,5)
\psline[linecolor=tauonecolor](7,5)(8,6)
\pscircle*[linecolor=tauonecolor](7,5){\cirrad}
\psline[linecolor=tauonecolor](8,3)(9,4)
\pscircle*[linecolor=tauonecolor](8,3){\cirrad}
\uput{\labelrad}[-90](8.10,3.90){$\frac{\gamma}{\tau^{3}} b$}
\psline[linecolor=hzerocolor]{->}(8.10,3.90)(8.10,4.60)
\pscircle*[linecolor=tauzerocolor](8.10,3.90){\cirrad}
\psline[linecolor=tauonecolor](7.90,4.10)(9,4)
\psline[linecolor=tauonecolor](7.90,4.10)(9,5)
\pscircle*[linecolor=tauonecolor](7.90,4.10){\cirrad}
\psline[linecolor=tauonecolor](8,5)(9,5)
\psline[linecolor=tauonecolor](8,5)(9,6)
\pscircle*[linecolor=tauonecolor](8,5){\cirrad}
\psline[linecolor=tauonecolor](8,6)(9,6)
\psline[linecolor=tauonecolor](8,6)(9,7)
\pscircle*[linecolor=tauonecolor](8,6){\cirrad}
\psline[linecolor=tauonecolor](9,4)(10,5)
\pscircle*[linecolor=tauonecolor](9,4){\cirrad}
\psline[linecolor=tauonecolor](9,5)(10,5)
\psline[linecolor=tauonecolor](9,5)(10,6)
\pscircle*[linecolor=tauonecolor](9,5){\cirrad}
\psline[linecolor=tauonecolor](9,6)(10,6)
\psline[linecolor=tauonecolor](9,6)(10,7)
\pscircle*[linecolor=tauonecolor](9,6){\cirrad}
\psline[linecolor=tauonecolor](9,7)(10,7)
\psline[linecolor=tauonecolor]{->}(9,7)(9.70,7.70)
\pscircle*[linecolor=tauonecolor](9,7){\cirrad}
\psline[linecolor=tauonecolor](10,5)(11,6)
\pscircle*[linecolor=tauonecolor](10,5){\cirrad}
\psline[linecolor=tauonecolor](10,6)(11,6)
\psline[linecolor=tauonecolor](10,6)(11,7)
\pscircle*[linecolor=tauonecolor](10,6){\cirrad}
\psline[linecolor=tauonecolor](10,7)(11,7)
\psline[linecolor=tauonecolor]{->}(10,7)(10.70,7.70)
\pscircle*[linecolor=tauonecolor](10,7){\cirrad}
\psline[linecolor=tauonecolor](11,6)(11.90,6.10)
\psline[linecolor=tauonecolor](11,6)(11.90,7.10)
\pscircle*[linecolor=tauonecolor](11,6){\cirrad}
\psline[linecolor=tauonecolor](11,7)(11.90,7.10)
\psline[linecolor=tauonecolor]{->}(11,7)(11.70,7.70)
\pscircle*[linecolor=tauonecolor](11,7){\cirrad}
\psline[linecolor=tauonecolor](11.90,6.10)(12.90,7.10)
\pscircle*[linecolor=tauonecolor](11.90,6.10){\cirrad}
\uput{\labelrad}[-90](12.10,5.90){${Q} h_1^{3} b$}
\psline[linecolor=tauonecolor](12.10,5.90)(12.10,6.90)
\psline[linecolor=tauonecolor](12.10,5.90)(13.10,6.90)
\pscircle*[linecolor=tauonecolor](12.10,5.90){\cirrad}
\psline[linecolor=tauonecolor](11.90,7.10)(12.90,7.10)
\psline[linecolor=tauonecolor]{->}(11.90,7.10)(12.60,7.80)
\pscircle*[linecolor=tauonecolor](11.90,7.10){\cirrad}
\uput{\labelrad}[-150](12.10,6.90){$\frac{\gamma}{\tau} a b$}
\psline[linecolor=hzerocolor]{->}(12.10,6.90)(12.10,7.60)
\pscircle*[linecolor=tauzerocolor](12.10,6.90){\cirrad}
\psline[linecolor=tauonecolor]{->}(12.90,7.10)(13.60,7.10)
\psline[linecolor=tauonecolor]{->}(12.90,7.10)(13.60,7.80)
\pscircle*[linecolor=tauonecolor](12.90,7.10){\cirrad}
\psline[linecolor=tauonecolor](13.10,6.90)(14,7)
\psline[linecolor=tauonecolor](13.10,6.90)(14,8)
\pscircle*[linecolor=tauonecolor](13.10,6.90){\cirrad}
\psline[linecolor=tauonecolor](14,7)(15,7)
\psline[linecolor=tauonecolor](14,7)(15,8)
\pscircle*[linecolor=tauonecolor](14,7){\cirrad}
\psline[linecolor=tauonecolor](14,8)(15,8)
\psline[linecolor=tauonecolor]{->}(14,8)(14.70,8.70)
\pscircle*[linecolor=tauonecolor](14,8){\cirrad}
\psline[linecolor=tauonecolor](15,7)(16,7)
\psline[linecolor=tauonecolor](15,7)(15.90,8.10)
\pscircle*[linecolor=tauonecolor](15,7){\cirrad}
\psline[linecolor=tauonecolor](15,8)(15.90,8.10)
\psline[linecolor=tauonecolor]{->}(15,8)(15.70,8.70)
\pscircle*[linecolor=tauonecolor](15,8){\cirrad}
\pscircle*[linecolor=tauonecolor](16,7){\cirrad}
\psline[linecolor=tauonecolor]{->}(15.90,8.10)(16.60,8.80)
\pscircle*[linecolor=tauonecolor](15.90,8.10){\cirrad}
\uput{\labelrad}[-90](16.10,7.90){$\frac{\gamma}{\tau^{3}} b^{2}$}
\psline[linecolor=hzerocolor]{->}(16.10,7.90)(16.10,8.60)
\pscircle*[linecolor=tauzerocolor](16.10,7.90){\cirrad}
\end{pspicture}
\begin{pspicture}(-1,-1)(17,9)
\psgrid[unit=2,gridcolor=gridline,subgriddiv=0,gridlabelcolor=white](0,0)(8,4)
\scriptsize
\rput(0,-1){0}
\rput(2,-1){2}
\rput(4,-1){4}
\rput(6,-1){6}
\rput(8,-1){8}
\rput(10,-1){10}
\rput(12,-1){12}
\rput(14,-1){14}
\rput(16,-1){16}

\rput(-1,0){0}
\rput(-1,2){2}
\rput(-1,4){4}
\rput(-1,6){6}
\rput(-1,8){8}

\uput[0](0,8.2){\textsc{Bockstein $E_6^-$-page, $mw \equiv 3 \pmod 4$}}
\uput{\labelrad}[-90](0,0){$\frac{\gamma}{\tau^{4}}$}
\psline[linecolor=pcolor]{->}(0,0)(0.70,0)
\psline[linecolor=honecolor](0,0)(1,1)
\pscircle*[linecolor=tauzerocolor](0,0){\cirrad}
\psline[linecolor=pcolor]{->}(1,1)(1.70,1)
\psline[linecolor=honecolor](1,1)(2,2)
\pscircle*[linecolor=tauzerocolor](1,1){\cirrad}
\uput{\labelrad}[0](1,7){$\F_2[\tau^4]/ \tau^\infty$}
\pscircle*[linecolor=tauzerocolor](1,7){\cirrad}
\psline[linecolor=pcolor]{->}(2,2)(2.70,2)
\pscircle*[linecolor=tauzerocolor](2,2){\cirrad}
\uput{\labelrad}[-90](4,3){$\frac{\gamma}{\tau^{6}} a$}
\psline[linecolor=pcolor]{->}(4,3)(4.70,3)
\pscircle*[linecolor=tauzerocolor](4,3){\cirrad}
\uput{\labelrad}[-90](8,4){$\frac{\gamma}{\tau^{8}} b$}
\psline[linecolor=pcolor]{->}(8,4)(8.70,4)
\psline[linecolor=honecolor](8,4)(9,5)
\pscircle*[linecolor=tauzerocolor](8,4){\cirrad}
\psline[linecolor=pcolor]{->}(9,5)(9.70,5)
\psline[linecolor=honecolor](9,5)(10,6)
\pscircle*[linecolor=tauzerocolor](9,5){\cirrad}
\psline[linecolor=pcolor]{->}(10,6)(10.70,6)
\pscircle*[linecolor=tauzerocolor](10,6){\cirrad}
\uput{\labelrad}[-90](12,7){$\frac{\gamma}{\tau^{10}} a b$}
\psline[linecolor=pcolor]{->}(12,7)(12.70,7)
\pscircle*[linecolor=tauzerocolor](12,7){\cirrad}
\uput{\labelrad}[-90](16,8){$\frac{\gamma}{\tau^{8}} b^{2}$}
\psline[linecolor=pcolor]{->}(16,8)(16.70,8)
\pscircle*[linecolor=tauzerocolor](16,8){\cirrad}
\end{pspicture}
\vfill
\label{E-end}

\newpage
\centerline{$\Ext_{NC}$ charts for $\cA^{C_2}(1)$}

\label{E-dummy}
\begin{pspicture}(-1,-1)(17,9)
\psgrid[unit=2,gridcolor=gridline,subgriddiv=0,gridlabelcolor=white](0,0)(8,4)
\scriptsize
\rput(0,-1){0}
\rput(2,-1){2}
\rput(4,-1){4}
\rput(6,-1){6}
\rput(8,-1){8}
\rput(10,-1){10}
\rput(12,-1){12}
\rput(14,-1){14}
\rput(16,-1){16}

\rput(-1,0){0}
\rput(-1,2){2}
\rput(-1,4){4}
\rput(-1,6){6}
\rput(-1,8){8}

\uput[0](0,8.2){\textsc{$\Ext_{NC}$, $mw \equiv 1 \pmod 4$}}
\uput{\labelrad}[-90](0,0){$\frac{\gamma}{\tau^{2}}$}
\psline[linecolor=pcolor](0,0)(1,0)
\psline[linecolor=honecolor](0,0)(1,1)
\pscircle*[linecolor=tauzerocolor](0,0){\cirrad}
\psline[linecolor=hzerocolor,linestyle=dashed, dash=3pt 2pt](1,0)(1,1)
\psline[linecolor=honecolor](1,0)(2,1)
\pscircle*[linecolor=tauzerocolor](1,0){\cirrad}
\psline[linecolor=pcolor](1,1)(2,1)
\pscircle*[linecolor=tauzerocolor](1,1){\cirrad}
\pscircle*[linecolor=tauzerocolor](2,1){\cirrad}
\uput{\labelrad}[-90](8,4){$\frac{\gamma}{\tau^{2}} b$}
\psline[linecolor=pcolor](8,4)(9,4)
\psline[linecolor=honecolor](8,4)(9,5)
\pscircle*[linecolor=tauzerocolor](8,4){\cirrad}
\psline[linecolor=hzerocolor,linestyle=dashed, dash=3pt 2pt](9,4)(9,5)
\psline[linecolor=honecolor](9,4)(10,5)
\pscircle*[linecolor=tauzerocolor](9,4){\cirrad}
\psline[linecolor=pcolor](9,5)(10,5)
\pscircle*[linecolor=tauzerocolor](9,5){\cirrad}
\pscircle*[linecolor=tauzerocolor](10,5){\cirrad}
\uput{\labelrad}[0](1,7){$\F_2[\tau^4]/ \tau^\infty$}
\pscircle*[linecolor=tauzerocolor](1,7){\cirrad}
\uput{\labelrad}[-90](16,8){$\frac{\gamma}{\tau^{2}} b^{2}$}
\pscircle*[linecolor=tauzerocolor](16,8){\cirrad}
\end{pspicture}
\vfill

\begin{pspicture}(-1,-1)(17,9)
\psgrid[unit=2,gridcolor=gridline,subgriddiv=0,gridlabelcolor=white](0,0)(8,4)
\scriptsize
\rput(0,-1){0}
\rput(2,-1){2}
\rput(4,-1){4}
\rput(6,-1){6}
\rput(8,-1){8}
\rput(10,-1){10}
\rput(12,-1){12}
\rput(14,-1){14}
\rput(16,-1){16}

\rput(-1,0){0}
\rput(-1,2){2}
\rput(-1,4){4}
\rput(-1,6){6}
\rput(-1,8){8}

\uput[0](0,8.2){\textsc{$\Ext_{NC}$, $mw \equiv 2 \pmod 4$}}
\uput{\labelrad}[-90](0,0){$\frac{\gamma}{\tau}$}
\psline[linecolor=hzerocolor](0,0)(0,1)
\psline[linecolor=honecolor](0,0)(1,1)
\pscircle*[linecolor=tauzerocolor](0,0){\cirrad}
\psline[linecolor=hzerocolor]{->}(0,1)(0,1.70)
\pscircle*[linecolor=tauzerocolor](0,1){\cirrad}
\psline[linecolor=pcolor](1,1)(2,1)
\psline[linecolor=honecolor](1,1)(2,2)
\pscircle*[linecolor=tauzerocolor](1,1){\cirrad}
\psline[linecolor=hzerocolor,linestyle=dashed, dash=3pt 2pt](2,1)(2,2)
\psline[linecolor=honecolor](2,1)(3,2)
\pscircle*[linecolor=tauzerocolor](2,1){\cirrad}
\psline[linecolor=pcolor](2,2)(3,2)
\pscircle*[linecolor=tauzerocolor](2,2){\cirrad}
\psline[linecolor=pcolor](3,2)(4,2)
\pscircle*[linecolor=tauzerocolor](3,2){\cirrad}
\psline[linecolor=hzerocolor,linestyle=dashed, dash=3pt 2pt](4,2)(4,3)
\pscircle*[linecolor=tauzerocolor](4,2){\cirrad}
\uput{\labelrad}[150](4,3){$\frac{\gamma}{\tau^{3}} a$}
\psline[linecolor=hzerocolor]{->}(4,3)(4,3.70)
\pscircle*[linecolor=tauzerocolor](4,3){\cirrad}
\uput{\labelrad}[-90](8,4){$\frac{\gamma}{\tau} b$}
\psline[linecolor=hzerocolor](8,4)(8,5)
\psline[linecolor=honecolor](8,4)(9,5)
\pscircle*[linecolor=tauzerocolor](8,4){\cirrad}
\psline[linecolor=hzerocolor]{->}(8,5)(8,5.70)
\pscircle*[linecolor=tauzerocolor](8,5){\cirrad}
\psline[linecolor=pcolor](9,5)(10,5)
\psline[linecolor=honecolor](9,5)(10,6)
\pscircle*[linecolor=tauzerocolor](9,5){\cirrad}
\psline[linecolor=hzerocolor,linestyle=dashed, dash=3pt 2pt](10,5)(10,6)
\psline[linecolor=honecolor](10,5)(11,6)
\pscircle*[linecolor=tauzerocolor](10,5){\cirrad}
\psline[linecolor=pcolor](10,6)(11,6)
\pscircle*[linecolor=tauzerocolor](10,6){\cirrad}
\psline[linecolor=pcolor](11,6)(12,6)
\pscircle*[linecolor=tauzerocolor](11,6){\cirrad}
\psline[linecolor=hzerocolor,linestyle=dashed, dash=3pt 2pt](12,6)(12,7)
\pscircle*[linecolor=tauzerocolor](12,6){\cirrad}
\uput{\labelrad}[0](1,7){$\F_2[\tau^4]/ \tau^\infty$}
\pscircle*[linecolor=tauzerocolor](1,7){\cirrad}
\uput{\labelrad}[150](12,7){$\frac{\gamma}{\tau^{3}} a b$}
\psline[linecolor=hzerocolor]{->}(12,7)(12,7.70)
\pscircle*[linecolor=tauzerocolor](12,7){\cirrad}
\uput{\labelrad}[-90](16,8){$\frac{\gamma}{\tau} b^{2}$}
\psline[linecolor=hzerocolor]{->}(16,8)(16,8.70)
\pscircle*[linecolor=tauzerocolor](16,8){\cirrad}
\end{pspicture}
\begin{pspicture}(-1,-1)(17,9)
\psgrid[unit=2,gridcolor=gridline,subgriddiv=0,gridlabelcolor=white](0,0)(8,4)
\scriptsize
\rput(0,-1){0}
\rput(2,-1){2}
\rput(4,-1){4}
\rput(6,-1){6}
\rput(8,-1){8}
\rput(10,-1){10}
\rput(12,-1){12}
\rput(14,-1){14}
\rput(16,-1){16}

\rput(-1,0){0}
\rput(-1,2){2}
\rput(-1,4){4}
\rput(-1,6){6}
\rput(-1,8){8}

\uput[0](0,8.2){\textsc{$\Ext_{NC}$, $mw \equiv 3 \pmod 4$}}
\uput{\labelrad}[-90](0,0){$\frac{\gamma}{\tau^{4}}$}
\psline[linecolor=pcolor](0,0)(1,0)
\psline[linecolor=honecolor](0,0)(1,1)
\pscircle*[linecolor=tauzerocolor](0,0){\cirrad}
\psline[linecolor=pcolor](1,0)(2,0)
\psline[linecolor=honecolor](1,0)(2,1)
\pscircle*[linecolor=tauzerocolor](1,0){\cirrad}
\psline[linecolor=pcolor](2,0)(3,0)
\psline[linecolor=honecolor](2,0)(3,1)
\pscircle*[linecolor=tauzerocolor](2,0){\cirrad}
\psline[linecolor=pcolor](3,0)(4,0)
\psline[linecolor=honecolor](3,0)(4,1)
\pscircle*[linecolor=tauzerocolor](3,0){\cirrad}
\psline[linecolor=pcolor](4,0)(5,0)
\psline[linecolor=honecolor](4,0)(5,1)
\pscircle*[linecolor=tauzerocolor](4,0){\cirrad}
\psline[linecolor=pcolor](5,0)(6,0)
\psline[linecolor=honecolor](5,0)(6,1)
\pscircle*[linecolor=tauzerocolor](5,0){\cirrad}
\psline[linecolor=pcolor](6,0)(7,0)
\psline[linecolor=honecolor](6,0)(7,1)
\pscircle*[linecolor=tauzerocolor](6,0){\cirrad}
\psline[linecolor=pcolor](7,0)(8,0)
\psline[linecolor=honecolor](7,0)(8,1)
\pscircle*[linecolor=tauzerocolor](7,0){\cirrad}
\psline[linecolor=pcolor]{->}(8,0)(8.70,0)
\psline[linecolor=honecolor](8,0)(9,1)
\pscircle*[linecolor=tauzerocolor](8,0){\cirrad}
\psline[linecolor=pcolor](1,1)(2,1)
\psline[linecolor=honecolor](1,1)(2,2)
\pscircle*[linecolor=tauzerocolor](1,1){\cirrad}
\psline[linecolor=pcolor](2,1)(3,1)
\psline[linecolor=honecolor](2,1)(3,2)
\pscircle*[linecolor=tauzerocolor](2,1){\cirrad}
\psline[linecolor=pcolor](3,1)(4,1)
\psline[linecolor=honecolor](3,1)(4,2)
\pscircle*[linecolor=tauzerocolor](3,1){\cirrad}
\psline[linecolor=pcolor](4,1)(5,1)
\psline[linecolor=honecolor](4,1)(5,2)
\pscircle*[linecolor=tauzerocolor](4,1){\cirrad}
\psline[linecolor=pcolor](5,1)(6,1)
\psline[linecolor=honecolor](5,1)(6,2)
\pscircle*[linecolor=tauzerocolor](5,1){\cirrad}
\psline[linecolor=pcolor](6,1)(7,1)
\psline[linecolor=honecolor](6,1)(7,2)
\pscircle*[linecolor=tauzerocolor](6,1){\cirrad}
\psline[linecolor=pcolor](7,1)(8,1)
\psline[linecolor=honecolor](7,1)(8,2)
\pscircle*[linecolor=tauzerocolor](7,1){\cirrad}
\psline[linecolor=pcolor](8,1)(9,1)
\psline[linecolor=honecolor](8,1)(9,2)
\pscircle*[linecolor=tauzerocolor](8,1){\cirrad}
\psline[linecolor=pcolor]{->}(9,1)(9.70,1)
\psline[linecolor=honecolor](9,1)(10,2)
\pscircle*[linecolor=tauzerocolor](9,1){\cirrad}
\psline[linecolor=pcolor](2,2)(3,2)
\pscircle*[linecolor=tauzerocolor](2,2){\cirrad}
\psline[linecolor=pcolor](3,2)(4,2)
\psline[linecolor=honecolor,linestyle=dashed, dash=3pt 2pt](3,2)(4,3)
\pscircle*[linecolor=tauzerocolor](3,2){\cirrad}
\psline[linecolor=pcolor](4,2)(5,2)
\psline[linecolor=honecolor,linestyle=dashed, dash=3pt 2pt](4,2)(5,3)
\pscircle*[linecolor=tauzerocolor](4,2){\cirrad}
\psline[linecolor=pcolor](5,2)(6,2)
\psline[linecolor=honecolor,linestyle=dashed, dash=3pt 2pt](5,2)(6,3)
\pscircle*[linecolor=tauzerocolor](5,2){\cirrad}
\psline[linecolor=pcolor](6,2)(7,2)
\psline[linecolor=honecolor,linestyle=dashed, dash=3pt 2pt](6,2)(7,3)
\pscircle*[linecolor=tauzerocolor](6,2){\cirrad}
\psline[linecolor=pcolor](7,2)(8,2)
\psline[linecolor=honecolor,linestyle=dashed, dash=3pt 2pt](7,2)(8,3)
\pscircle*[linecolor=tauzerocolor](7,2){\cirrad}
\psline[linecolor=pcolor](8,2)(9,2)
\psline[linecolor=honecolor,linestyle=dashed, dash=3pt 2pt](8,2)(9,3)
\pscircle*[linecolor=tauzerocolor](8,2){\cirrad}
\psline[linecolor=pcolor](9,2)(10,2)
\psline[linecolor=honecolor,linestyle=dashed, dash=3pt 2pt](9,2)(10,3)
\pscircle*[linecolor=tauzerocolor](9,2){\cirrad}
\psline[linecolor=pcolor]{->}(10,2)(10.70,2)
\psline[linecolor=honecolor,linestyle=dashed, dash=3pt 2pt](10,2)(11,3)
\pscircle*[linecolor=tauzerocolor](10,2){\cirrad}
\uput{\labelrad}[150](4,3){$\frac{\gamma}{\tau^{6}} a$}
\psline[linecolor=pcolor](4,3)(5,3)
\pscircle*[linecolor=tauzerocolor](4,3){\cirrad}
\psline[linecolor=pcolor](5,3)(6,3)
\pscircle*[linecolor=tauzerocolor](5,3){\cirrad}
\psline[linecolor=pcolor](6,3)(7,3)
\pscircle*[linecolor=tauzerocolor](6,3){\cirrad}
\psline[linecolor=pcolor](7,3)(8,3)
\psline[linecolor=honecolor,linestyle=dashed, dash=3pt 2pt](7,3)(8,4)
\pscircle*[linecolor=tauzerocolor](7,3){\cirrad}
\psline[linecolor=pcolor](8,3)(9,3)
\psline[linecolor=honecolor,linestyle=dashed, dash=3pt 2pt](8,3)(9,4)
\pscircle*[linecolor=tauzerocolor](8,3){\cirrad}
\psline[linecolor=pcolor](9,3)(10,3)
\psline[linecolor=honecolor,linestyle=dashed, dash=3pt 2pt](9,3)(10,4)
\pscircle*[linecolor=tauzerocolor](9,3){\cirrad}
\psline[linecolor=pcolor](10,3)(11,3)
\psline[linecolor=honecolor,linestyle=dashed, dash=3pt 2pt](10,3)(11,4)
\pscircle*[linecolor=tauzerocolor](10,3){\cirrad}
\psline[linecolor=pcolor]{->}(11,3)(11.70,3)
\psline[linecolor=honecolor,linestyle=dashed, dash=3pt 2pt](11,3)(12,4)
\pscircle*[linecolor=tauzerocolor](11,3){\cirrad}
\uput{\labelrad}[150](8,4){$\frac{\gamma}{\tau^{8}} b$}
\psline[linecolor=pcolor](8,4)(9,4)
\psline[linecolor=honecolor](8,4)(9,5)
\pscircle*[linecolor=tauzerocolor](8,4){\cirrad}
\psline[linecolor=pcolor](9,4)(10,4)
\psline[linecolor=honecolor](9,4)(10,5)
\pscircle*[linecolor=tauzerocolor](9,4){\cirrad}
\psline[linecolor=pcolor](10,4)(11,4)
\psline[linecolor=honecolor](10,4)(11,5)
\pscircle*[linecolor=tauzerocolor](10,4){\cirrad}
\psline[linecolor=pcolor](11,4)(12,4)
\psline[linecolor=honecolor](11,4)(12,5)
\pscircle*[linecolor=tauzerocolor](11,4){\cirrad}
\psline[linecolor=pcolor]{->}(12,4)(12.70,4)
\psline[linecolor=honecolor](12,4)(13,5)
\pscircle*[linecolor=tauzerocolor](12,4){\cirrad}
\psline[linecolor=pcolor](9,5)(10,5)
\psline[linecolor=honecolor](9,5)(10,6)
\pscircle*[linecolor=tauzerocolor](9,5){\cirrad}
\psline[linecolor=pcolor](10,5)(11,5)
\psline[linecolor=honecolor](10,5)(11,6)
\pscircle*[linecolor=tauzerocolor](10,5){\cirrad}
\psline[linecolor=pcolor](11,5)(12,5)
\psline[linecolor=honecolor](11,5)(12,6)
\pscircle*[linecolor=tauzerocolor](11,5){\cirrad}
\psline[linecolor=pcolor](12,5)(13,5)
\psline[linecolor=honecolor](12,5)(13,6)
\pscircle*[linecolor=tauzerocolor](12,5){\cirrad}
\psline[linecolor=pcolor]{->}(13,5)(13.70,5)
\psline[linecolor=honecolor](13,5)(14,6)
\pscircle*[linecolor=tauzerocolor](13,5){\cirrad}
\psline[linecolor=pcolor](10,6)(11,6)
\pscircle*[linecolor=tauzerocolor](10,6){\cirrad}
\psline[linecolor=pcolor](11,6)(12,6)
\psline[linecolor=honecolor,linestyle=dashed, dash=3pt 2pt](11,6)(12,7)
\pscircle*[linecolor=tauzerocolor](11,6){\cirrad}
\psline[linecolor=pcolor](12,6)(13,6)
\psline[linecolor=honecolor,linestyle=dashed, dash=3pt 2pt](12,6)(13,7)
\pscircle*[linecolor=tauzerocolor](12,6){\cirrad}
\psline[linecolor=pcolor](13,6)(14,6)
\psline[linecolor=honecolor,linestyle=dashed, dash=3pt 2pt](13,6)(14,7)
\pscircle*[linecolor=tauzerocolor](13,6){\cirrad}
\psline[linecolor=pcolor]{->}(14,6)(14.70,6)
\psline[linecolor=honecolor,linestyle=dashed, dash=3pt 2pt](14,6)(15,7)
\pscircle*[linecolor=tauzerocolor](14,6){\cirrad}
\uput{\labelrad}[0](1,7){$\F_2[\tau^4]/ \tau^\infty$}
\pscircle*[linecolor=tauzerocolor](1,7){\cirrad}
\uput{\labelrad}[150](12,7){$\frac{\gamma}{\tau^{10}} a b$}
\psline[linecolor=pcolor](12,7)(13,7)
\pscircle*[linecolor=tauzerocolor](12,7){\cirrad}
\psline[linecolor=pcolor](13,7)(14,7)
\pscircle*[linecolor=tauzerocolor](13,7){\cirrad}
\psline[linecolor=pcolor](14,7)(15,7)
\pscircle*[linecolor=tauzerocolor](14,7){\cirrad}
\psline[linecolor=pcolor]{->}(15,7)(15.70,7)
\psline[linecolor=honecolor,linestyle=dashed, dash=3pt 2pt](15,7)(16,8)
\pscircle*[linecolor=tauzerocolor](15,7){\cirrad}
\uput{\labelrad}[150](16,8){$\frac{\gamma}{\tau^{8}} b^{2}$}
\psline[linecolor=pcolor]{->}(16,8)(16.70,8)
\pscircle*[linecolor=tauzerocolor](16,8){\cirrad}
\end{pspicture}
\vfill

\newpage
\end{landscape}

\psset{unit=0.466cm}

\renewcommand{\cirrad}{0.13}

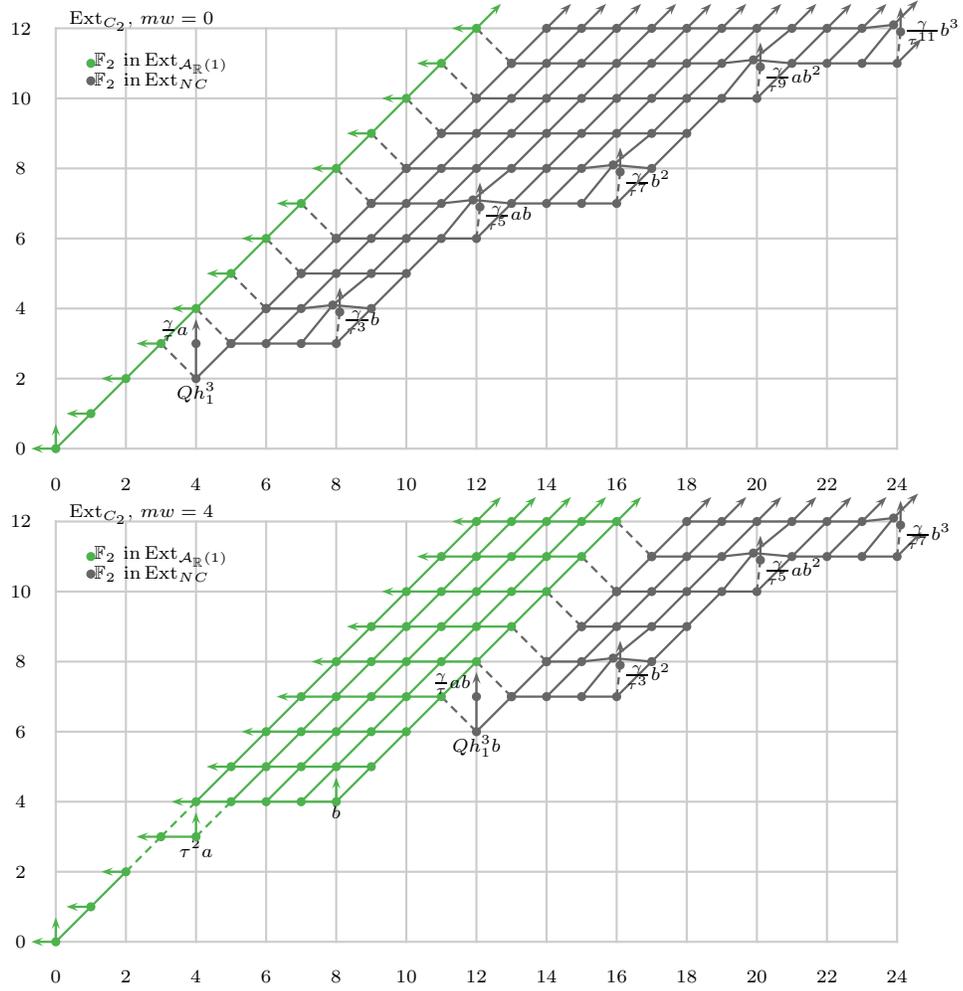
\centerline{$\Ext$ charts for $\cA^{C_2}(1)$ in $mw=0$ and $mw=4$}

\input{charts/ExtC2MW0chartdata.tex}

\input{charts/ExtC2MW4chartdata.tex}
\vfill

\newpage

\psset{unit=0.555cm}

\centerline{Bockstein charts for $\cE^{C_2}(1)$}
\label{BockE1start}

\renewcommand{\cirrad}{0.11}

\begin{center}
\begin{pspicture}(-1,-1)(9,7)
\psgrid[unit=2,gridcolor=gridline,subgriddiv=0,gridlabelcolor=white](0,0)(4,3)
\scriptsize
\rput(0,-1){0}
\rput(2,-1){2}
\rput(4,-1){4}
\rput(6,-1){6}
\rput(8,-1){8}

\rput(-1,0){0}
\rput(-1,2){2}
\rput(-1,4){4}
\rput(-1,6){6}

\uput[0](0,6.2){\textsc{Bockstein $E_1^+$-page}}
\uput{\labelrad}[-90](0,0){$1$}
\psline[linecolor=pcolor]{->}(0,0)(-0.70,0)
\psline[linecolor=hzerocolor](0,0)(0,1)
\pscircle*[linecolor=tauzerocolor](0,0){\cirrad}
\psline[linecolor=pcolor]{->}(0,1)(-0.70,1)
\psline[linecolor=hzerocolor]{->}(0,1)(0,1.70)
\pscircle*[linecolor=tauzerocolor](0,1){\cirrad}
\uput{\labelrad}[-90](2,1){$v_1$}
\psline[linecolor=pcolor]{->}(2,1)(1.30,1)
\psline[linecolor=hzerocolor](2,1)(2,2)
\pscircle*[linecolor=tauzerocolor](2,1){\cirrad}
\psline[linecolor=pcolor]{->}(2,2)(1.30,2)
\psline[linecolor=hzerocolor]{->}(2,2)(2,2.70)
\pscircle*[linecolor=tauzerocolor](2,2){\cirrad}
\uput{\labelrad}[-90](4,2){$v_1^{2}$}
\psline[linecolor=pcolor]{->}(4,2)(3.30,2)
\psline[linecolor=hzerocolor](4,2)(4,3)
\pscircle*[linecolor=tauzerocolor](4,2){\cirrad}
\psline[linecolor=pcolor]{->}(4,3)(3.30,3)
\psline[linecolor=hzerocolor]{->}(4,3)(4,3.70)
\pscircle*[linecolor=tauzerocolor](4,3){\cirrad}
\uput{\labelrad}[-90](6,3){$v_1^{3}$}
\psline[linecolor=pcolor]{->}(6,3)(5.30,3)
\psline[linecolor=hzerocolor](6,3)(6,4)
\pscircle*[linecolor=tauzerocolor](6,3){\cirrad}
\psline[linecolor=pcolor]{->}(6,4)(5.30,4)
\psline[linecolor=hzerocolor]{->}(6,4)(6,4.70)
\pscircle*[linecolor=tauzerocolor](6,4){\cirrad}
\uput{\labelrad}[-90](8,4){$v_1^{4}$}
\psline[linecolor=pcolor]{->}(8,4)(7.30,4)
\psline[linecolor=hzerocolor](8,4)(8,5)
\pscircle*[linecolor=tauzerocolor](8,4){\cirrad}
\psline[linecolor=pcolor]{->}(8,5)(7.30,5)
\psline[linecolor=hzerocolor]{->}(8,5)(8,5.70)
\pscircle*[linecolor=tauzerocolor](8,5){\cirrad}
\uput{\labelrad}[0](1,5){$\F_2[\tau]$}
\pscircle*[linecolor=tauzerocolor](1,5){\cirrad}
\end{pspicture}
\begin{pspicture}(-1,-1)(9,7)
\psgrid[unit=2,gridcolor=gridline,subgriddiv=0,gridlabelcolor=white](0,0)(4,3)
\scriptsize
\rput(0,-1){0}
\rput(2,-1){2}
\rput(4,-1){4}
\rput(6,-1){6}
\rput(8,-1){8}

\rput(-1,0){0}
\rput(-1,2){2}
\rput(-1,4){4}
\rput(-1,6){6}

\uput[0](0,6.2){\textsc{Bockstein $E_2^+$-page}}
\uput{\labelrad}[-90](0,0){$1$}
\psline[linecolor=pcolor]{->}(0,0)(-0.70,0)
\psline[linecolor=hzerocolor](0,0)(0,1)
\pscircle*[linecolor=tauzerocolor](0,0){\cirrad}
\psline[linecolor=hzerocolor]{->}(0,1)(0,1.70)
\pscircle*[linecolor=tauzerocolor](0,1){\cirrad}
\uput{\labelrad}[-90](2,1){$v_1$}
\psline[linecolor=pcolor]{->}(2,1)(1.30,1)
\psline[linecolor=hzerocolor](2,1)(2,2)
\pscircle*[linecolor=tauzerocolor](2,1){\cirrad}
\psline[linecolor=hzerocolor]{->}(2,2)(2,2.70)
\pscircle*[linecolor=tauzerocolor](2,2){\cirrad}
\uput{\labelrad}[-90](4,2){$v_1^{2}$}
\psline[linecolor=pcolor]{->}(4,2)(3.30,2)
\psline[linecolor=hzerocolor](4,2)(4,3)
\pscircle*[linecolor=tauzerocolor](4,2){\cirrad}
\psline[linecolor=hzerocolor]{->}(4,3)(4,3.70)
\pscircle*[linecolor=tauzerocolor](4,3){\cirrad}
\uput{\labelrad}[-90](6,3){$v_1^{3}$}
\psline[linecolor=pcolor]{->}(6,3)(5.30,3)
\psline[linecolor=hzerocolor](6,3)(6,4)
\pscircle*[linecolor=tauzerocolor](6,3){\cirrad}
\psline[linecolor=hzerocolor]{->}(6,4)(6,4.70)
\pscircle*[linecolor=tauzerocolor](6,4){\cirrad}
\uput{\labelrad}[-90](8,4){$v_1^{4}$}
\psline[linecolor=pcolor]{->}(8,4)(7.30,4)
\psline[linecolor=hzerocolor](8,4)(8,5)
\pscircle*[linecolor=tauzerocolor](8,4){\cirrad}
\psline[linecolor=hzerocolor]{->}(8,5)(8,5.70)
\pscircle*[linecolor=tauzerocolor](8,5){\cirrad}
\uput{\labelrad}[0](1,5){$\F_2[\tau^2]$}
\pscircle*[linecolor=tauzerocolor](1,5){\cirrad}
\end{pspicture}

\begin{pspicture}(-1,-1)(9,7)
\psgrid[unit=2,gridcolor=gridline,subgriddiv=0,gridlabelcolor=white](0,0)(4,3)
\scriptsize
\rput(0,-1){0}
\rput(2,-1){2}
\rput(4,-1){4}
\rput(6,-1){6}
\rput(8,-1){8}

\rput(-1,0){0}
\rput(-1,2){2}
\rput(-1,4){4}
\rput(-1,6){6}

\uput[0](0,6.2){\textsc{Bockstein $E_3^+ = E_\infty^+$-page, Part A}}
\uput{\labelrad}[-90](0.00,0.00){$1$}
\psline[linecolor=pcolor]{->}(0.00,0.00)(-0.70,0.00)
\psline[linecolor=hzerocolor](0.00,0.00)(-0.10,1.10)
\pscircle*[linecolor=tauzerocolor](0.00,0.00){\cirrad}
\psline[linecolor=hzerocolor]{->}(-0.10,1.10)(-0.10,1.80)
\pscircle*[linecolor=tauzerocolor](-0.10,1.10){\cirrad}
\pscircle*[linecolor=tauzerocolor](0.10,0.90){\cirrad}
\psline[linecolor=pcolor](1.00,1.00)(0.10,0.90)
\pscircle*[linecolor=tauzerocolor](1.00,1.00){\cirrad}
\uput{\labelrad}[0](1.00,5.00){$\F_2[\tau^4]$}
\pscircle*[linecolor=tauzerocolor](1.00,5.00){\cirrad}
\uput{\labelrad}[-90](2.00,1.00){$v_1$}
\psline[linecolor=pcolor](2.00,1.00)(1.00,1.00)
\psline[linecolor=hzerocolor](2.00,1.00)(1.90,2.10)
\pscircle*[linecolor=tauzerocolor](2.00,1.00){\cirrad}
\psline[linecolor=hzerocolor]{->}(1.90,2.10)(1.90,2.80)
\pscircle*[linecolor=tauzerocolor](1.90,2.10){\cirrad}
\pscircle*[linecolor=tauzerocolor](2.10,1.90){\cirrad}
\psline[linecolor=pcolor](3.00,2.00)(2.10,1.90)
\pscircle*[linecolor=tauzerocolor](3.00,2.00){\cirrad}
\uput{\labelrad}[-90](4.00,2.00){$v_1^{2}$}
\psline[linecolor=pcolor](4.00,2.00)(3.00,2.00)
\psline[linecolor=hzerocolor](4.00,2.00)(3.90,3.10)
\pscircle*[linecolor=tauzerocolor](4.00,2.00){\cirrad}
\psline[linecolor=hzerocolor]{->}(3.90,3.10)(3.90,3.80)
\pscircle*[linecolor=tauzerocolor](3.90,3.10){\cirrad}
\pscircle*[linecolor=tauzerocolor](4.10,2.90){\cirrad}
\psline[linecolor=pcolor](5.00,3.00)(4.10,2.90)
\pscircle*[linecolor=tauzerocolor](5.00,3.00){\cirrad}
\uput{\labelrad}[-90](6.00,3.00){$v_1^{3}$}
\psline[linecolor=pcolor](6.00,3.00)(5.00,3.00)
\psline[linecolor=hzerocolor](6.00,3.00)(5.90,4.10)
\pscircle*[linecolor=tauzerocolor](6.00,3.00){\cirrad}
\psline[linecolor=hzerocolor]{->}(5.90,4.10)(5.90,4.80)
\pscircle*[linecolor=tauzerocolor](5.90,4.10){\cirrad}
\pscircle*[linecolor=tauzerocolor](6.10,3.90){\cirrad}
\psline[linecolor=pcolor](7.00,4.00)(6.10,3.90)
\pscircle*[linecolor=tauzerocolor](7.00,4.00){\cirrad}
\uput{\labelrad}[-90](8.00,4.00){$v_1^{4}$}
\psline[linecolor=pcolor](8.00,4.00)(7.00,4.00)
\psline[linecolor=hzerocolor](8.00,4.00)(7.90,5.10)
\pscircle*[linecolor=tauzerocolor](8.00,4.00){\cirrad}
\psline[linecolor=hzerocolor]{->}(7.90,5.10)(7.90,5.80)
\pscircle*[linecolor=tauzerocolor](7.90,5.10){\cirrad}
\pscircle*[linecolor=tauzerocolor](8.10,4.90){\cirrad}
\end{pspicture}
\begin{pspicture}(-1,-1)(9,7)
\psgrid[unit=2,gridcolor=gridline,subgriddiv=0,gridlabelcolor=white](0,0)(4,3)
\scriptsize
\rput(0,-1){0}
\rput(2,-1){2}
\rput(4,-1){4}
\rput(6,-1){6}
\rput(8,-1){8}

\rput(-1,0){0}
\rput(-1,2){2}
\rput(-1,4){4}
\rput(-1,6){6}

\uput[0](0,6.2){\textsc{Bockstein $E_3^+ = E_\infty^+$-page, Part B}}
\uput{\labelrad}[-90](0.00,1.00){$\tau^{2} h_0$}
\psline[linecolor=hzerocolor]{->}(0.00,1.00)(0.00,1.70)
\pscircle*[linecolor=tauzerocolor](0.00,1.00){\cirrad}
\uput{\labelrad}[0](1.00,5.00){$\F_2[\tau^4]$}
\pscircle*[linecolor=tauzerocolor](1.00,5.00){\cirrad}
\uput{\labelrad}[-90](2.00,2.00){$\tau^{2} h_0 v_1$}
\psline[linecolor=hzerocolor]{->}(2.00,2.00)(2.00,2.70)
\pscircle*[linecolor=tauzerocolor](2.00,2.00){\cirrad}
\uput{\labelrad}[-90](4.00,3.00){$\tau^{2} h_0 v_1^{2}$}
\psline[linecolor=hzerocolor]{->}(4.00,3.00)(4.00,3.70)
\pscircle*[linecolor=tauzerocolor](4.00,3.00){\cirrad}
\uput{\labelrad}[-90](6.00,4.00){$\tau^{2} h_0 v_1^{3}$}
\psline[linecolor=hzerocolor]{->}(6.00,4.00)(6.00,4.70)
\pscircle*[linecolor=tauzerocolor](6.00,4.00){\cirrad}
\uput{\labelrad}[-90](8.00,5.00){$\tau^{2} h_0 v_1^{4}$}
\psline[linecolor=hzerocolor]{->}(8.00,5.00)(8.00,5.70)
\pscircle*[linecolor=tauzerocolor](8.00,5.00){\cirrad}
\end{pspicture}

\begin{pspicture}(-1,-1)(9,7)
\psgrid[unit=2,gridcolor=gridline,subgriddiv=0,gridlabelcolor=white](0,0)(4,3)
\scriptsize
\rput(0,-1.1){0}
\rput(2,-1.1){2}
\rput(4,-1.1){4}
\rput(6,-1.1){6}
\rput(8,-1.1){8}

\rput(-1,0){0}
\rput(-1,2){2}
\rput(-1,4){4}
\rput(-1,6){6}

\uput[0](0,6.2){\textsc{Bockstein $E_1^-$-page}}
\uput{\labelrad}[-90](0,0){$\frac{\gamma}{ \tau}$}
\psline[linecolor=pcolor]{->}(0,0)(0.70,0)
\psline[linecolor=hzerocolor](0,0)(0,1)
\pscircle*[linecolor=tauzerocolor](0,0){\cirrad}
\psline[linecolor=pcolor]{->}(0,1)(0.70,1)
\psline[linecolor=hzerocolor]{->}(0,1)(0,1.70)
\pscircle*[linecolor=tauzerocolor](0,1){\cirrad}
\uput{\labelrad}[-90](2,1){$\frac{\gamma}{ \tau} v_1$}
\psline[linecolor=pcolor]{->}(2,1)(2.70,1)
\psline[linecolor=hzerocolor](2,1)(2,2)
\pscircle*[linecolor=tauzerocolor](2,1){\cirrad}
\psline[linecolor=pcolor]{->}(2,2)(2.70,2)
\psline[linecolor=hzerocolor]{->}(2,2)(2,2.70)
\pscircle*[linecolor=tauzerocolor](2,2){\cirrad}
\uput{\labelrad}[-90](4,2){$\frac{\gamma}{ \tau} v_1^{2}$}
\psline[linecolor=pcolor]{->}(4,2)(4.70,2)
\psline[linecolor=hzerocolor](4,2)(4,3)
\pscircle*[linecolor=tauzerocolor](4,2){\cirrad}
\psline[linecolor=pcolor]{->}(4,3)(4.70,3)
\psline[linecolor=hzerocolor]{->}(4,3)(4,3.70)
\pscircle*[linecolor=tauzerocolor](4,3){\cirrad}
\uput{\labelrad}[-90](6,3){$\frac{\gamma}{ \tau} v_1^{3}$}
\psline[linecolor=pcolor]{->}(6,3)(6.70,3)
\psline[linecolor=hzerocolor](6,3)(6,4)
\pscircle*[linecolor=tauzerocolor](6,3){\cirrad}
\psline[linecolor=pcolor]{->}(6,4)(6.70,4)
\psline[linecolor=hzerocolor]{->}(6,4)(6,4.70)
\pscircle*[linecolor=tauzerocolor](6,4){\cirrad}
\uput{\labelrad}[-90](8,4){$\frac{\gamma}{ \tau} v_1^{4}$}
\psline[linecolor=pcolor]{->}(8,4)(8.70,4)
\psline[linecolor=hzerocolor](8,4)(8,5)
\pscircle*[linecolor=tauzerocolor](8,4){\cirrad}
\psline[linecolor=pcolor]{->}(8,5)(8.70,5)
\psline[linecolor=hzerocolor]{->}(8,5)(8,5.70)
\pscircle*[linecolor=tauzerocolor](8,5){\cirrad}
\uput{\labelrad}[0](1,5){$\F_2[\tau]/ \tau^\infty$}
\pscircle*[linecolor=tauzerocolor](1,5){\cirrad}
\end{pspicture}
\begin{pspicture}(-1,-1)(9,7)
\psgrid[unit=2,gridcolor=gridline,subgriddiv=0,gridlabelcolor=white](0,0)(4,3)
\scriptsize
\rput(0,-1.1){0}
\rput(2,-1.1){2}
\rput(4,-1.1){4}
\rput(6,-1.1){6}
\rput(8,-1.1){8}

\rput(-1,0){0}
\rput(-1,2){2}
\rput(-1,4){4}
\rput(-1,6){6}

\uput[0](0,6.2){\textsc{Bockstein $E_2^-$-page}}
\uput{\labelrad}[180](-0.10,0.10){$\frac{\gamma}{ \tau}$}
\psline[linecolor=hzerocolor](-0.10,0.10)(0.00,1.00)
\pscircle*[linecolor=tauzerocolor](-0.10,0.10){\cirrad}
\uput{\labelrad}[-90](0.10,-0.10){$\frac{\gamma}{ \tau^{2}}$}
\psline[linecolor=pcolor]{->}(0.10,-0.10)(0.80,-0.10)
\pscircle*[linecolor=tauzerocolor](0.10,-0.10){\cirrad}
\psline[linecolor=hzerocolor]{->}(0.00,1.00)(0.00,1.70)
\pscircle*[linecolor=tauzerocolor](0.00,1.00){\cirrad}
\uput{\labelrad}[0](1.00,5.00){$\F_2[\tau^2]/ \tau^\infty$}
\pscircle*[linecolor=tauzerocolor](1.00,5.00){\cirrad}
\uput{\labelrad}[180](1.90,1.10){$\frac{\gamma}{ \tau} v_1$}
\psline[linecolor=hzerocolor](1.90,1.10)(2.00,2.00)
\pscircle*[linecolor=tauzerocolor](1.90,1.10){\cirrad}
\uput{\labelrad}[-90](2.10,0.90){$\frac{\gamma}{ \tau^{2}} v_1$}
\psline[linecolor=pcolor]{->}(2.10,0.90)(2.80,0.90)
\pscircle*[linecolor=tauzerocolor](2.10,0.90){\cirrad}
\psline[linecolor=hzerocolor]{->}(2.00,2.00)(2.00,2.70)
\pscircle*[linecolor=tauzerocolor](2.00,2.00){\cirrad}
\uput{\labelrad}[180](3.90,2.10){$\frac{\gamma}{ \tau} v_1^{2}$}
\psline[linecolor=hzerocolor](3.90,2.10)(4.00,3.00)
\pscircle*[linecolor=tauzerocolor](3.90,2.10){\cirrad}
\uput{\labelrad}[-90](4.10,1.90){$\frac{\gamma}{ \tau^{2}} v_1^{2}$}
\psline[linecolor=pcolor]{->}(4.10,1.90)(4.80,1.90)
\pscircle*[linecolor=tauzerocolor](4.10,1.90){\cirrad}
\psline[linecolor=hzerocolor]{->}(4.00,3.00)(4.00,3.70)
\pscircle*[linecolor=tauzerocolor](4.00,3.00){\cirrad}
\uput{\labelrad}[180](5.90,3.10){$\frac{\gamma}{ \tau} v_1^{3}$}
\psline[linecolor=hzerocolor](5.90,3.10)(6.00,4.00)
\pscircle*[linecolor=tauzerocolor](5.90,3.10){\cirrad}
\uput{\labelrad}[-90](6.10,2.90){$\frac{\gamma}{ \tau^{2}} v_1^{3}$}
\psline[linecolor=pcolor]{->}(6.10,2.90)(6.80,2.90)
\pscircle*[linecolor=tauzerocolor](6.10,2.90){\cirrad}
\psline[linecolor=hzerocolor]{->}(6.00,4.00)(6.00,4.70)
\pscircle*[linecolor=tauzerocolor](6.00,4.00){\cirrad}
\uput{\labelrad}[180](7.90,4.10){$\frac{\gamma}{ \tau} v_1^{4}$}
\psline[linecolor=hzerocolor](7.90,4.10)(8.00,5.00)
\pscircle*[linecolor=tauzerocolor](7.90,4.10){\cirrad}
\uput{\labelrad}[-90](8.10,3.90){$\frac{\gamma}{ \tau^{2}} v_1^{4}$}
\psline[linecolor=pcolor]{->}(8.10,3.90)(8.80,3.90)
\pscircle*[linecolor=tauzerocolor](8.10,3.90){\cirrad}
\psline[linecolor=hzerocolor]{->}(8.00,5.00)(8.00,5.70)
\pscircle*[linecolor=tauzerocolor](8.00,5.00){\cirrad}
\end{pspicture}

\begin{pspicture}(-1,-1)(9,7)
\psgrid[unit=2,gridcolor=gridline,subgriddiv=0,gridlabelcolor=white](0,0)(4,3)
\scriptsize
\rput(0,-1.1){0}
\rput(2,-1.1){2}
\rput(4,-1.1){4}
\rput(6,-1.1){6}
\rput(8,-1.1){8}

\rput(-1,0){0}
\rput(-1,2){2}
\rput(-1,4){4}
\rput(-1,6){6}

\uput[0](0,6.2){\textsc{Bockstein $E_4^-$-page, Part A}}
\uput{\labelrad}[180](-0.10,0.10){$\frac{\gamma}{ \tau}$}
\psline[linecolor=hzerocolor](-0.10,0.10)(0.00,1.00)
\pscircle*[linecolor=tauzerocolor](-0.10,0.10){\cirrad}
\uput{\labelrad}[-90](0.10,-0.10){$\frac{\gamma}{ \tau^{4}}$}
\psline[linecolor=pcolor]{->}(0.10,-0.10)(0.80,-0.10)
\pscircle*[linecolor=tauzerocolor](0.10,-0.10){\cirrad}
\psline[linecolor=hzerocolor]{->}(0.00,1.00)(0.00,1.70)
\pscircle*[linecolor=tauzerocolor](0.00,1.00){\cirrad}
\uput{\labelrad}[0](1.00,5.00){$\F_2[\tau^4]/ \tau^\infty$}
\pscircle*[linecolor=tauzerocolor](1.00,5.00){\cirrad}
\uput{\labelrad}[-90](2.00,1.00){$\frac{\gamma}{ \tau} v_1$}
\psline[linecolor=hzerocolor](2.00,1.00)(2.00,2.00)
\pscircle*[linecolor=tauzerocolor](2.00,1.00){\cirrad}
\psline[linecolor=hzerocolor]{->}(2.00,2.00)(2.00,2.70)
\pscircle*[linecolor=tauzerocolor](2.00,2.00){\cirrad}
\uput{\labelrad}[-90](4.00,2.00){$\frac{\gamma}{ \tau} v_1^{2}$}
\psline[linecolor=hzerocolor](4.00,2.00)(4.00,3.00)
\pscircle*[linecolor=tauzerocolor](4.00,2.00){\cirrad}
\psline[linecolor=hzerocolor]{->}(4.00,3.00)(4.00,3.70)
\pscircle*[linecolor=tauzerocolor](4.00,3.00){\cirrad}
\uput{\labelrad}[-90](6.00,3.00){$\frac{\gamma}{ \tau} v_1^{3}$}
\psline[linecolor=hzerocolor](6.00,3.00)(6.00,4.00)
\pscircle*[linecolor=tauzerocolor](6.00,3.00){\cirrad}
\psline[linecolor=hzerocolor]{->}(6.00,4.00)(6.00,4.70)
\pscircle*[linecolor=tauzerocolor](6.00,4.00){\cirrad}
\uput{\labelrad}[-90](8.00,4.00){$\frac{\gamma}{ \tau} v_1^{4}$}
\psline[linecolor=hzerocolor](8.00,4.00)(8.00,5.00)
\pscircle*[linecolor=tauzerocolor](8.00,4.00){\cirrad}
\psline[linecolor=hzerocolor]{->}(8.00,5.00)(8.00,5.70)
\pscircle*[linecolor=tauzerocolor](8.00,5.00){\cirrad}
\end{pspicture}
\begin{pspicture}(-1,-1)(9,7)
\psgrid[unit=2,gridcolor=gridline,subgriddiv=0,gridlabelcolor=white](0,0)(4,3)
\scriptsize
\rput(0,-1.1){0}
\rput(2,-1.1){2}
\rput(4,-1){4}
\rput(6,-1){6}
\rput(8,-1){8}

\rput(-1,0){0}
\rput(-1,2){2}
\rput(-1,4){4}
\rput(-1,6){6}

\uput[0](0,6.2){\textsc{Bockstein $E_4^-$-page, Part B}}
\uput{\labelrad}[-90](0.10,-0.10){$\frac{\gamma}{ \tau^{2}}$}
\psline[linecolor=pcolor](0.10,-0.10)(1.00,0.00)
\pscircle*[linecolor=tauzerocolor](0.10,-0.10){\cirrad}
\uput{\labelrad}[120](-0.10,0.10){$\frac{\gamma}{ \tau^{3}}$}
\psline[linecolor=hzerocolor](-0.10,0.10)(0.00,1.00)
\pscircle*[linecolor=tauzerocolor](-0.10,0.10){\cirrad}
\psline[linecolor=hzerocolor]{->}(0.00,1.00)(0.00,1.70)
\pscircle*[linecolor=tauzerocolor](0.00,1.00){\cirrad}
\psline[linecolor=pcolor](1.00,0.00)(2.00,0.00)
\pscircle*[linecolor=tauzerocolor](1.00,0.00){\cirrad}
\uput{\labelrad}[0](1.00,5.00){$\F_2[\tau^4]/ \tau^\infty$}
\pscircle*[linecolor=tauzerocolor](1.00,5.00){\cirrad}
\pscircle*[linecolor=tauzerocolor](2.00,0.00){\cirrad}
\uput{\labelrad}[-60](2.18,0.95){$\frac{\gamma}{ \tau^{2}} v_1$}
\psline[linecolor=pcolor](2.10,0.90)(3.00,1.00)
\pscircle*[linecolor=tauzerocolor](2.10,0.90){\cirrad}
\uput{\labelrad}[180](1.90,1.10){$\frac{\gamma}{ \tau^{3}} v_1$}
\psline[linecolor=hzerocolor](1.90,1.10)(2.00,2.00)
\pscircle*[linecolor=tauzerocolor](1.90,1.10){\cirrad}
\psline[linecolor=hzerocolor]{->}(2.00,2.00)(2.00,2.70)
\pscircle*[linecolor=tauzerocolor](2.00,2.00){\cirrad}
\psline[linecolor=pcolor](3.00,1.00)(4.00,1.00)
\pscircle*[linecolor=tauzerocolor](3.00,1.00){\cirrad}
\pscircle*[linecolor=tauzerocolor](4.00,1.00){\cirrad}
\uput{\labelrad}[-60](4.18,1.95){$\frac{\gamma}{ \tau^{2}} v_1^{2}$}
\psline[linecolor=pcolor](4.10,1.90)(5.00,2.00)
\pscircle*[linecolor=tauzerocolor](4.10,1.90){\cirrad}
\uput{\labelrad}[180](3.90,2.10){$\frac{\gamma}{ \tau^{3}} v_1^{2}$}
\psline[linecolor=hzerocolor](3.90,2.10)(4.00,3.00)
\pscircle*[linecolor=tauzerocolor](3.90,2.10){\cirrad}
\psline[linecolor=hzerocolor]{->}(4.00,3.00)(4.00,3.70)
\pscircle*[linecolor=tauzerocolor](4.00,3.00){\cirrad}
\psline[linecolor=pcolor](5.00,2.00)(6.00,2.00)
\pscircle*[linecolor=tauzerocolor](5.00,2.00){\cirrad}
\pscircle*[linecolor=tauzerocolor](6.00,2.00){\cirrad}
\uput{\labelrad}[-60](6.18,2.95){$\frac{\gamma}{ \tau^{2}} v_1^{3}$}
\psline[linecolor=pcolor](6.10,2.90)(7.00,3.00)
\pscircle*[linecolor=tauzerocolor](6.10,2.90){\cirrad}
\uput{\labelrad}[180](5.90,3.10){$\frac{\gamma}{ \tau^{3}} v_1^{3}$}
\psline[linecolor=hzerocolor](5.90,3.10)(6.00,4.00)
\pscircle*[linecolor=tauzerocolor](5.90,3.10){\cirrad}
\psline[linecolor=hzerocolor]{->}(6.00,4.00)(6.00,4.70)
\pscircle*[linecolor=tauzerocolor](6.00,4.00){\cirrad}
\psline[linecolor=pcolor](7.00,3.00)(8.00,3.00)
\pscircle*[linecolor=tauzerocolor](7.00,3.00){\cirrad}
\pscircle*[linecolor=tauzerocolor](8.00,3.00){\cirrad}
\uput{\labelrad}[-60](8.18,3.95){$\frac{\gamma}{ \tau^{2}} v_1^{4}$}
\pscircle*[linecolor=tauzerocolor](8.10,3.90){\cirrad}
\uput{\labelrad}[180](7.90,4.10){$\frac{\gamma}{ \tau^{3}} v_1^{4}$}
\psline[linecolor=hzerocolor](7.90,4.10)(8.00,5.00)
\pscircle*[linecolor=tauzerocolor](7.90,4.10){\cirrad}
\psline[linecolor=hzerocolor]{->}(8.00,5.00)(8.00,5.70)
\pscircle*[linecolor=tauzerocolor](8.00,5.00){\cirrad}
\end{pspicture}
\end{center}

\newpage
\centerline{$\Ext_{NC}$ charts for $\cE^{C_2}(1)$}

\begin{center}
\begin{pspicture}(-1,-1)(9,7)
\psgrid[unit=2,gridcolor=gridline,subgriddiv=0,gridlabelcolor=white](0,0)(4,3)
\scriptsize
\rput(0,-1.1){0}
\rput(2,-1.1){2}
\rput(4,-1.1){4}
\rput(6,-1.1){6}
\rput(8,-1.1){8}

\rput(-1,0){0}
\rput(-1,2){2}
\rput(-1,4){4}
\rput(-1,6){6}

\uput[0](0,6.2){\textsc{$\Ext_{NC}$, Part A}}
\uput{\labelrad}[180](-0.10,0.10){$\frac{\gamma}{ \tau}$}
\psline[linecolor=hzerocolor](-0.10,0.10)(0.00,1.00)
\pscircle*[linecolor=tauzerocolor](-0.10,0.10){\cirrad}
\uput{\labelrad}[-90](0.10,-0.10){$\frac{\gamma}{ \tau^{4}}$}
\psline[linecolor=pcolor]{->}(0.10,-0.10)(0.80,-0.10)
\pscircle*[linecolor=tauzerocolor](0.10,-0.10){\cirrad}
\psline[linecolor=hzerocolor]{->}(0.00,1.00)(0.00,1.70)
\pscircle*[linecolor=tauzerocolor](0.00,1.00){\cirrad}
\uput{\labelrad}[0](1.00,5.00){$\F_2[\tau^4]/ \tau^\infty$}
\pscircle*[linecolor=tauzerocolor](1.00,5.00){\cirrad}
\uput{\labelrad}[-90](2.00,1.00){$\frac{\gamma}{ \tau} v_1$}
\psline[linecolor=hzerocolor](2.00,1.00)(2.00,2.00)
\pscircle*[linecolor=tauzerocolor](2.00,1.00){\cirrad}
\psline[linecolor=hzerocolor]{->}(2.00,2.00)(2.00,2.70)
\pscircle*[linecolor=tauzerocolor](2.00,2.00){\cirrad}
\uput{\labelrad}[-90](4.00,2.00){$\frac{\gamma}{ \tau} v_1^{2}$}
\psline[linecolor=hzerocolor](4.00,2.00)(4.00,3.00)
\pscircle*[linecolor=tauzerocolor](4.00,2.00){\cirrad}
\psline[linecolor=hzerocolor]{->}(4.00,3.00)(4.00,3.70)
\pscircle*[linecolor=tauzerocolor](4.00,3.00){\cirrad}
\uput{\labelrad}[-90](6.00,3.00){$\frac{\gamma}{ \tau} v_1^{3}$}
\psline[linecolor=hzerocolor](6.00,3.00)(6.00,4.00)
\pscircle*[linecolor=tauzerocolor](6.00,3.00){\cirrad}
\psline[linecolor=hzerocolor]{->}(6.00,4.00)(6.00,4.70)
\pscircle*[linecolor=tauzerocolor](6.00,4.00){\cirrad}
\uput{\labelrad}[-90](8.00,4.00){$\frac{\gamma}{ \tau} v_1^{4}$}
\psline[linecolor=hzerocolor](8.00,4.00)(8.00,5.00)
\pscircle*[linecolor=tauzerocolor](8.00,4.00){\cirrad}
\psline[linecolor=hzerocolor]{->}(8.00,5.00)(8.00,5.70)
\pscircle*[linecolor=tauzerocolor](8.00,5.00){\cirrad}
\end{pspicture}
\begin{pspicture}(-1,-1)(9,7)
\psgrid[unit=2,gridcolor=gridline,subgriddiv=0,gridlabelcolor=white](0,0)(4,3)
\scriptsize
\rput(0,-1.1){0}
\rput(2,-1.1){2}
\rput(4,-1.1){4}
\rput(6,-1.1){6}
\rput(8,-1.1){8}

\rput(-1,0){0}
\rput(-1,2){2}
\rput(-1,4){4}
\rput(-1,6){6}

\uput[0](0,6.2){\textsc{$\Ext_{NC}$, Part B}}
\uput{\labelrad}[-90](0.10,-0.10){$\frac{\gamma}{ \tau^{2}}$}
\psline[linecolor=pcolor](0.10,-0.10)(1.00,0.00)
\pscircle*[linecolor=tauzerocolor](0.10,-0.10){\cirrad}
\uput{\labelrad}[120](-0.10,0.10){$\frac{\gamma}{ \tau^{3}}$}
\psline[linecolor=hzerocolor](-0.10,0.10)(0.00,1.00)
\pscircle*[linecolor=tauzerocolor](-0.10,0.10){\cirrad}
\psline[linecolor=hzerocolor]{->}(0.00,1.00)(0.00,1.70)
\pscircle*[linecolor=tauzerocolor](0.00,1.00){\cirrad}
\psline[linecolor=pcolor](1.00,0.00)(2.00,0.00)
\pscircle*[linecolor=tauzerocolor](1.00,0.00){\cirrad}
\uput{\labelrad}[0](1.00,5.00){$\F_2[\tau^4]/ \tau^\infty$}
\pscircle*[linecolor=tauzerocolor](1.00,5.00){\cirrad}
\psline[linecolor=hzerocolor,linestyle=dashed, dash=3pt 2pt](2.00,0.00)(1.90,1.10)
\pscircle*[linecolor=tauzerocolor](2.00,0.00){\cirrad}
\uput{\labelrad}[-60](2.18,0.95){$\frac{\gamma}{ \tau^{2}} v_1$}
\psline[linecolor=pcolor](2.10,0.90)(3.00,1.00)
\pscircle*[linecolor=tauzerocolor](2.10,0.90){\cirrad}
\uput{\labelrad}[180](1.90,1.10){$\frac{\gamma}{ \tau^{3}} v_1$}
\psline[linecolor=hzerocolor](1.90,1.10)(2.00,2.00)
\pscircle*[linecolor=tauzerocolor](1.90,1.10){\cirrad}
\psline[linecolor=hzerocolor]{->}(2.00,2.00)(2.00,2.70)
\pscircle*[linecolor=tauzerocolor](2.00,2.00){\cirrad}
\psline[linecolor=pcolor](3.00,1.00)(4.00,1.00)
\pscircle*[linecolor=tauzerocolor](3.00,1.00){\cirrad}
\psline[linecolor=hzerocolor,linestyle=dashed, dash=3pt 2pt](4.00,1.00)(3.90,2.10)
\pscircle*[linecolor=tauzerocolor](4.00,1.00){\cirrad}
\uput{\labelrad}[-60](4.18,1.95){$\frac{\gamma}{ \tau^{2}} v_1^{2}$}
\psline[linecolor=pcolor](4.10,1.90)(5.00,2.00)
\pscircle*[linecolor=tauzerocolor](4.10,1.90){\cirrad}
\uput{\labelrad}[180](3.90,2.10){$\frac{\gamma}{ \tau^{3}} v_1^{2}$}
\psline[linecolor=hzerocolor](3.90,2.10)(4.00,3.00)
\pscircle*[linecolor=tauzerocolor](3.90,2.10){\cirrad}
\psline[linecolor=hzerocolor]{->}(4.00,3.00)(4.00,3.70)
\pscircle*[linecolor=tauzerocolor](4.00,3.00){\cirrad}
\psline[linecolor=pcolor](5.00,2.00)(6.00,2.00)
\pscircle*[linecolor=tauzerocolor](5.00,2.00){\cirrad}
\psline[linecolor=hzerocolor,linestyle=dashed, dash=3pt 2pt](6.00,2.00)(5.90,3.10)
\pscircle*[linecolor=tauzerocolor](6.00,2.00){\cirrad}
\uput{\labelrad}[-60](6.18,2.95){$\frac{\gamma}{ \tau^{2}} v_1^{3}$}
\psline[linecolor=pcolor](6.10,2.90)(7.00,3.00)
\pscircle*[linecolor=tauzerocolor](6.10,2.90){\cirrad}
\uput{\labelrad}[180](5.90,3.10){$\frac{\gamma}{ \tau^{3}} v_1^{3}$}
\psline[linecolor=hzerocolor](5.90,3.10)(6.00,4.00)
\pscircle*[linecolor=tauzerocolor](5.90,3.10){\cirrad}
\psline[linecolor=hzerocolor]{->}(6.00,4.00)(6.00,4.70)
\pscircle*[linecolor=tauzerocolor](6.00,4.00){\cirrad}
\psline[linecolor=pcolor](7.00,3.00)(8.00,3.00)
\pscircle*[linecolor=tauzerocolor](7.00,3.00){\cirrad}
\psline[linecolor=hzerocolor,linestyle=dashed, dash=3pt 2pt](8.00,3.00)(7.90,4.10)
\pscircle*[linecolor=tauzerocolor](8.00,3.00){\cirrad}
\uput{\labelrad}[-60](8.18,3.95){$\frac{\gamma}{ \tau^{2}} v_1^{4}$}
\pscircle*[linecolor=tauzerocolor](8.10,3.90){\cirrad}
\uput{\labelrad}[180](7.90,4.10){$\frac{\gamma}{ \tau^{3}} v_1^{4}$}
\psline[linecolor=hzerocolor](7.90,4.10)(8.00,5.00)
\pscircle*[linecolor=tauzerocolor](7.90,4.10){\cirrad}
\psline[linecolor=hzerocolor]{->}(8.00,5.00)(8.00,5.70)
\pscircle*[linecolor=tauzerocolor](8.00,5.00){\cirrad}
\end{pspicture}
\end{center}


\clearpage
\newpage

\begin{myfigure}
\label{fig:MW-module}

\begin{center}
Milnor-Witt modules for $ko_{C_2}$
\end{center}
\bigskip \bigskip

\begin{center}

\input{charts/HtpyMWchartdata.tex}
\end{center}

\end{myfigure}

\newpage

\begin{myfigure}
\label{fig:MWkR-module}

\begin{center}
Milnor-Witt modules for $k\R$
\end{center}
\bigskip \bigskip

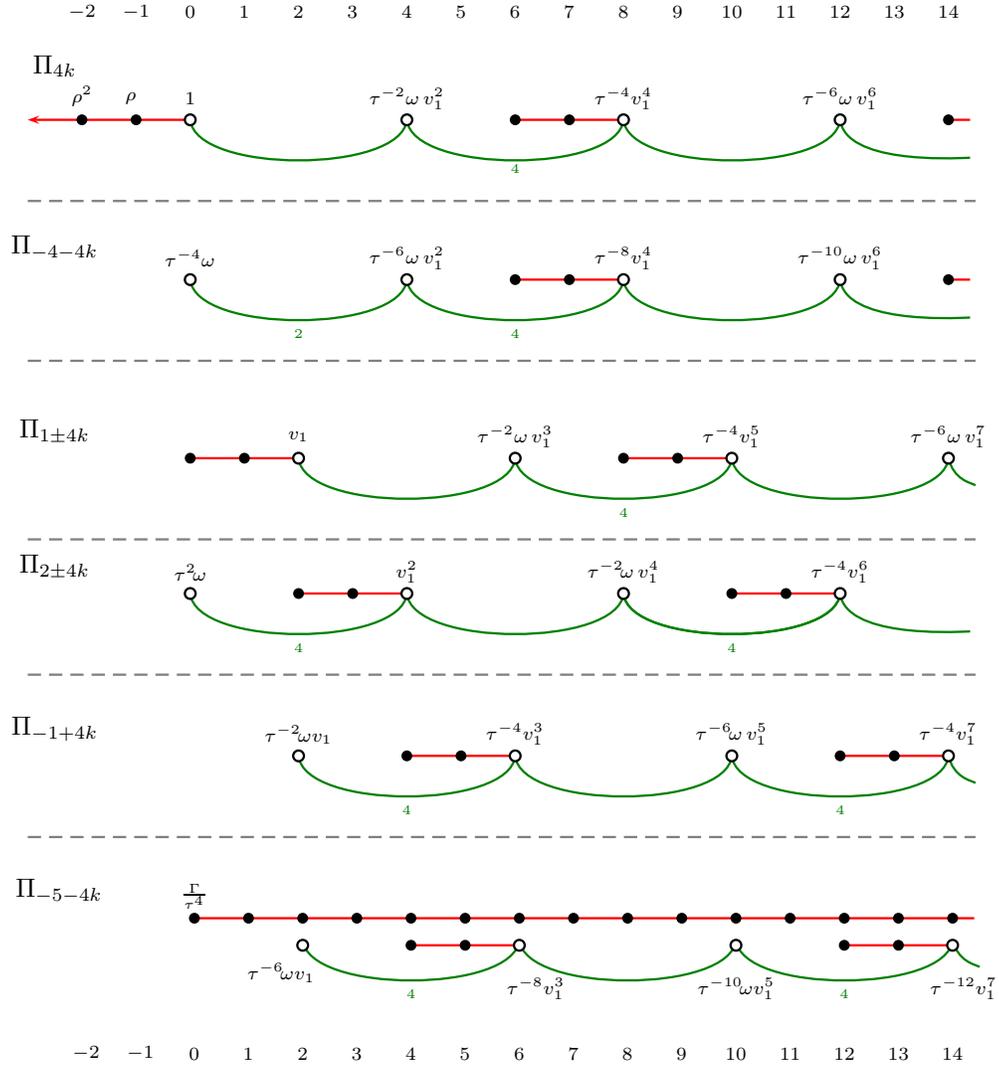
\begin{center}

\psset{unit=0.72cm}
\begin{pspicture}(-10,0)(7,1)

\rput(-1,0){
\begin{pspicture}(-3,-1.5)(15,3)

\rput(-2.5,1){$\Pi_{4k}$}

\scriptsize

\multido{\ii=-2+1}{17}
{\rput(\ii,2){$\ii$}}

\psline[linecolor=MWrhocolor](0,0)(-1,0)
\psline[linecolor=MWrhocolor](6,0)(8,0)
\psline[linecolor=MWrhocolor](14,0)(14.4,0)
\psline[linecolor=MWrhocolor]{->}(-1,-0)(-3,-0)

\psbezier[linecolor=MWalphacolor](0,0)(0.1,-1)(3.9,-1)(4,0)
\rput(4,0){\psbezier[linecolor=MWalphacolor](0,0)(0.1,-1)(3.9,-1)(4,0)}
\rput(8,0){\psbezier[linecolor=MWalphacolor](0,0)(0.1,-1)(3.9,-1)(4,0)}
\rput(12,0){\psbezier[linecolor=MWalphacolor](0,0)(0.1,-0.8)(2,-0.72)(2.4,-0.7)}

\pscircle[fillstyle=solid,fillcolor=white](0,0){\crad}
\pscircle[fillstyle=solid,fillcolor=white](4,0){\crad}
\pscircle[fillstyle=solid,fillcolor=white](8,0){\crad}
\pscircle[fillstyle=solid,fillcolor=white](12,0){\crad}
\multido{\ri=-1+-1}{2}
{\pscircle*(\ri,0){\smallcrad}}
\multido{\ri=6+1}{2}
{\pscircle*(\ri,0){\smallcrad}}
\pscircle*(14,0){\smallcrad}

\rput(-1.1,0.4){$\arho$}
\rput(-2,0.45){$\arho^2$}

\rput(0,0.4){$1$}
\rput(4,0.4){$\tau^{-2} \omega\, v_1^2$}
\rput(8,0.4){$\tau^{-4} v_1^4$}
\rput(12,0.4){$\tau^{-6}\omega\, v_1^6$}

\rput(6,-0.9){$\color{MWalphacolor}\scriptscriptstyle{4}$}

\psline[linecolor=gray,linestyle=dashed](-3,-1.5)(14.5,-1.5)

\end{pspicture} }


\rput(-1,-3.5){
\begin{pspicture}(-3,-1.1)(15,1.5)

\rput(-2.5,0.6){$\Pi_{-4-4k}$}

\scriptsize

\psline[linecolor=MWrhocolor](6,0)(8,0)
\psline[linecolor=MWrhocolor](14,0)(14.4,0)

\psbezier[linecolor=MWalphacolor](0,0)(0.1,-1)(3.9,-1)(4,0)
\rput(4,0){\psbezier[linecolor=MWalphacolor](0,0)(0.1,-1)(3.9,-1)(4,0)}
\rput(8,0){\psbezier[linecolor=MWalphacolor](0,0)(0.1,-1)(3.9,-1)(4,0)}
\rput(12,0){\psbezier[linecolor=MWalphacolor](0,0)(0.1,-0.8)(2,-0.72)(2.4,-0.7)}

\multido{\ri=0+4}{4}
{\pscircle[fillstyle=solid,fillcolor=white](\ri,0){\crad}}
\pscircle*(6,0){\smallcrad}
\pscircle*(7,0){\smallcrad}
\pscircle*(14,0){\smallcrad}

\rput(0,0.4){$\tau^{-4}\omega$}
\rput(4,0.4){$\tau^{-6}\omega\, v_1^2$}
\rput(8,0.4){$ \tau^{-8} v_1^4$}
\rput(12,0.4){$ \tau^{-10}\omega\, v_1^6$}

\rput(2,-1){$\color{MWalphacolor}\scriptscriptstyle{2}$}
\rput(6,-1){$\color{MWalphacolor}\scriptscriptstyle{4}$}

\psline[linecolor=gray,linestyle=dashed](-3,-1.5)(14.5,-1.5)

\end{pspicture} }


\rput(-1,-6.5){
\begin{pspicture}(-3,-1)(15,2)

\rput(-2.5,0.5){$\Pi_{1\pm4k}$}

\scriptsize

\psline[linecolor=MWrhocolor](2,0)(0,0)
\psline[linecolor=MWrhocolor](10,0)(8,0)

\rput(2,0){\psbezier[linecolor=MWalphacolor](0,0)(0.1,-1)(3.9,-1)(4,0)}
\rput(6,0){\psbezier[linecolor=MWalphacolor](0,0)(0.1,-1)(3.9,-1)(4,0)}
\rput(10,0){\psbezier[linecolor=MWalphacolor](0,0)(0.1,-1)(3.9,-1)(4,0)}
\rput(14,0){\psbezier[linecolor=MWalphacolor](0,0)(0.1,-0.4)(0.4,-0.45)(0.5,-0.5)}

\pscircle*(1,0){\smallcrad}
\pscircle[fillstyle=solid,fillcolor=white](2,0){\crad}
\pscircle*(0,0){\smallcrad}

\rput(2,0.4){$v_1$}

\pscircle[fillstyle=solid,fillcolor=white](6,0){\crad}
\rput(6,0.4){$ \tau^{-2}\omega\, v_1^3$}

\pscircle*(9,0){\smallcrad}
\pscircle[fillstyle=solid,fillcolor=white](10,0){\crad}
\pscircle*(8,0){\smallcrad}

\pscircle[fillstyle=solid,fillcolor=white](14,0){\crad}

\rput(10,0.4){$\tau^{-4} v_1^5$}
\rput(14,0.4){$\tau^{-6}\omega\, v_1^7$}

\rput(8,-1){$\color{MWalphacolor}\scriptscriptstyle{4}$}

\psline[linecolor=gray,linestyle=dashed](-3,-1.5)(14.5,-1.5)

\end{pspicture} }


\rput(-1,-9.5){
\begin{pspicture}(-3,-1.5)(15,1.5)

\rput(-2.5,0.5){$\Pi_{2\pm4k}$}

\scriptsize

\psline[linecolor=MWrhocolor](2,0)(4,0)

\psbezier[linecolor=MWalphacolor](0,0)(0.1,-1)(3.9,-1)(4,0)
\rput(4,0){\psbezier[linecolor=MWalphacolor](0,0)(0.1,-1)(3.9,-1)(4,0)}
\rput(8,0){\psbezier[linecolor=MWalphacolor](0,0)(0.1,-1)(3.9,-1)(4,0)}
\rput(12,0){\psbezier[linecolor=MWalphacolor](0,0)(0.1,-0.8)(2,-0.72)(2.4,-0.7)}

\pscircle[fillstyle=solid,fillcolor=white](0,0){\crad}
\pscircle*(2,0){\smallcrad}
\pscircle*(3,0){\smallcrad}
\pscircle[fillstyle=solid,fillcolor=white](4,0){\crad}

\rput(0,0.4){$\tau^2\!\omega$}
\rput(4,0.4){$v_1^2$}
\rput(2,-1){$\color{MWalphacolor}\scriptscriptstyle{4}$}

\psline[linecolor=MWrhocolor](10,0)(12,0)

\psbezier[linecolor=MWalphacolor](8,0)(8.1,-1)(11.9,-1)(12,0)

\pscircle[fillstyle=solid,fillcolor=white](8,0){\crad}
\pscircle*(10,0){\smallcrad}
\pscircle*(11,0){\smallcrad}
\pscircle[fillstyle=solid,fillcolor=white](12,0){\crad}

\rput(8,0.4){$\tau^{-2}\!\omega\, v_1^4$}
\rput(12,0.4){$\tau^{-4} v_1^6$}
\rput(10,-1){$\color{MWalphacolor}\scriptscriptstyle{4}$}

\psline[linecolor=gray,linestyle=dashed](-3,-1.5)(14.5,-1.5)

\end{pspicture} }


\rput(-1,-12.5){
\begin{pspicture}(-3,-1.5)(15,1.5)

\rput(-2.5,0.5){$\Pi_{-1 + 4k}$}

\scriptsize

\psline[linecolor=MWrhocolor](4,0)(6,0)

\rput(2,0){\psbezier[linecolor=MWalphacolor](0,0)(0.1,-1)(3.9,-1)(4,0)}
\rput(6,0){\psbezier[linecolor=MWalphacolor](0,0)(0.1,-1)(3.9,-1)(4,0)}
\rput(10,0){\psbezier[linecolor=MWalphacolor](0,0)(0.1,-1)(3.9,-1)(4,0)}
\rput(14,0){\psbezier[linecolor=MWalphacolor](0,0)(0.1,-0.4)(0.4,-0.45)(0.5,-0.5)}

\pscircle[fillstyle=solid,fillcolor=white](2,0){\crad}
\pscircle*(4,0){\smallcrad}
\pscircle*(5,0){\smallcrad}
\pscircle[fillstyle=solid,fillcolor=white](6,0){\crad}

\rput(2,0.4){$\tau^{-2}\!\omega v_1$}
\rput(6,0.4){$\tau^{-4}v_1^3$}
\rput(4,-1){$\color{MWalphacolor}\scriptscriptstyle{4}$}

\psline[linecolor=MWrhocolor](12,0)(14,0)

\pscircle[fillstyle=solid,fillcolor=white](10,0){\crad}
\pscircle*(12,0){\smallcrad}
\pscircle*(13,0){\smallcrad}
\pscircle[fillstyle=solid,fillcolor=white](14,0){\crad}

\rput(10,0.4){$\tau^{-6}\!\omega\,v_1^5$}
\rput(14,0.4){$\tau^{-4} v_1^7$}
\rput(12,-1){$\color{MWalphacolor}\scriptscriptstyle{4}$}

\psline[linecolor=gray,linestyle=dashed](-3,-1.5)(14.5,-1.5)

\end{pspicture} }


%
%
%
%
%


\rput(-1,-16){
\begin{pspicture}(-3,-2)(15,1)

\rput(-2.5,0.5){$\Pi_{-5-4k}$}

\scriptsize

\psline[linecolor=MWrhocolor](0,0)(14.4,0)

\rput(2,-0.4){\psbezier[linecolor=MWalphacolor](0,0)(0.1,-1)(3.9,-1)(4,0)}
\rput(6,-0.4){\psbezier[linecolor=MWalphacolor](0,0)(0.1,-1)(3.9,-1)(4,0)}
\rput(10,-0.4){\psbezier[linecolor=MWalphacolor](0,0)(0.1,-1)(3.9,-1)(4,0)}
\rput(14,-0.4){\psbezier[linecolor=MWalphacolor](0,0)(0.1,-0.4)(0.4,-0.45)(0.5,-0.5)}

\multido{\ri=0+1}{15}
{\pscircle*(\ri,0){\smallcrad}}

\psline[linecolor=MWrhocolor](4,-0.5)(6,-0.5)
\psline[linecolor=MWrhocolor](12,-0.5)(14,-0.5)

\pscircle[fillstyle=solid,fillcolor=white](2,-0.5){\crad}
\pscircle*(4,-0.5){\smallcrad}
\pscircle*(5,-0.5){\smallcrad}
\pscircle[fillstyle=solid,fillcolor=white](6,-0.5){\crad}
\pscircle[fillstyle=solid,fillcolor=white](10,-0.5){\crad}
\pscircle*(12,-0.5){\smallcrad}
\pscircle*(13,-0.5){\smallcrad}
\pscircle[fillstyle=solid,fillcolor=white](14,-0.5){\crad}

\rput(4,-1.4){$\color{MWalphacolor}\scriptscriptstyle{4}$}
\rput(12,-1.4){$\color{MWalphacolor}\scriptscriptstyle{4}$}

\rput(0,0.4){$\frac{\Gamma}{\tau^4}$}
\rput(1.6,-1){$\tau^{-6}\!\omega v_1$}
\rput(6.3,-1.3){$\tau^{-8} v_1^3$}
\rput(10,-1.3){$\tau^{-10}\!\omega v_1^5$}
\rput(14.2,-1.3){$\tau^{-12} v_1^7$}

\multido{\ii=-2+1}{17}
{\rput(\ii,-2.5){$\ii$}}

\end{pspicture}}

\end{pspicture}
\end{center}

\end{myfigure}

\clearpage
\newpage

\


\begin{thebibliography}{99}

\bibitem[A]{NthKind}
Adams, J. F. 
Operations of the $n$th kind in $K$-theory, and what we don't know about $\R\bP^\infty$. New Developments in Topology (Proc. Sympos. Algebraic Topology, Oxford, 1972), pp. 1--9. London Math. Soc. Lecture Note Ser., No. 11, Cambridge Univ. Press, London, 1974.

\bibitem[AI]{AI}
S.\ Araki, and K.\ Iriye. 
\emph{Equivariant stable homotopy groups of spheres with involutions I.}
Osaka J.\ Math.\ 19 (1982), 1--55. 


\bibitem[A]{At}
M.\ F.\ Atiyah. 
\emph{K-theory and reality.}
Quart.\ J.\ Math.\ Oxford Ser.\ (2) 17 (1966), 367--386.

\bibitem[Ba]{Ba}
C. Barwick.
Spectral Mackey functors and equivariant algebraic K-theory (I). 
Adv. Math. 304 (2017), 646--727.

\bibitem[B]{B}
G.\ E.\ Bredon,
\emph{Equivariant homotopy}, 
Proc.\ Conf.\ on Transformation Groups (New Orleans, 1967), 
Springer, 1968, 281--292. 

\bibitem[BG]{BG}
R.\ R.\ Bruner and J.\ P.\ C.\ Greenlees,
\emph{Connective real $K$-Theory of finite groups.}
Math.\ Surveys and Monographs 169, Amer.\ Math.\ Soc., 2010.

\bibitem[DM]{DMEuler}
D.\ Davis and M.\ Mahowald,
\emph{The Euler class for connective ko-theory and an application to immersions of quaternionic projective space.} 
Indiana Univ.\ Math.\ J.\ 28 (1979), 1025--1034.

\bibitem[DM2]{DM}
D. Davis and M. Mahowald.
{\it The spectrum $(P\smsh bo)_{-\infty}$.}
Math. Proc. Cambridge Philos. Soc. 96 (1984), no. 1, 85--93.

\bibitem[D]{D}
D.\ Dugger,
\emph{An Atiyah-Hirzebruch spectral sequence for KR-theory.}
 K-Theory 35 (2005), 213--256.

\bibitem[DI]{DImass}
D.\ Dugger and D.\ C.\ Isaksen, 
\emph{The motivic Adams spectral sequence.} 
Geom. Topol. 14 (2010), no. 2, 967--1014. 

\bibitem[DI2]{DIHopf}
D.\ Dugger and D.\ C.\ Isaksen, 
\emph{Motivic Hopf elements and relations.}
New York J.\ Math.\ 19 (2013), 823--871.

\bibitem[DI3]{MWLow}
D.\ Dugger and D.\ C.\ Isaksen,
\emph{Low dimensional Milnor-Witt stems over $\R$},
Ann.\ K-Theory 2 (2017), 175--210.

\bibitem[DI4]{C2RStems}
D.\ Dugger and D.\ C.\ Isaksen,
\emph{$\Z/2$-equivariant and $\R$-motivic stable stems},
 Proc.\ Amer.\ Math.\ Soc. 145 (2017), no. 8, 3617--3627.
 
\bibitem[G]{G2}
J.\ P.\ C.\ Greenlees,
{\it Four approaches to cohomology theories with reality.}
An alpine bouquet of algebraic topology, 139--156, 
Contemp. Math., 708, Amer. Math. Soc., 2018.

\bibitem[GrM]{GrM}
J. P. C. Greenlees and  Lennart Meier. 
\emph{Gorenstein duality for real spectra}. 
Algebr. Geom. Topol. 17 (2017), no. 6, 3547--3619.

\bibitem[GM]{GM}
B. Guillou and J. P. May.
Models of G-spectra as presheaves of spectra.
Available as arXiv preprint arXiv:1110.3571.

\bibitem[HO]{HO}
J. Heller and K. Ormsby, 
\emph{Galois equivariance and stable motivic homotopy theory. }
Trans. Amer. Math. Soc. 368 (2016), no. 11, 8047--8077.

\bibitem[H]{H}
M.\ A.\ Hill,
\emph{Ext and the motivic Steenrod algebra over $\R$. }
J.\ Pure Appl.\ Algebra 215 (2011), 715--727.

\bibitem[HHR]{HHR}
M. A. Hill, M. J. Hopkins, and D. C. Ravenel,
\emph{On the nonexistence of elements of Kervaire invariant one.} 
Ann. of Math. (2) 184 (2016), no. 1, 1--262.

\bibitem[HK]{HK}
P.\ Hu and I.\ Kriz,
\emph{Real-oriented homotopy theory and an analogue of the
Adams-Novikov spectral sequence.}
Topology 40 (2001), 317--399.

\bibitem[Ir]{Iriye}
K.\ Iriye, 
\emph{Equivariant stable homotopy groups of spheres with involutions II.}
Osaka J.\ Math.\ 19 (1982), 733--743.

\bibitem[I]{Isstems}
D. Isaksen,
\emph{Stable Stems},
Mem.\ Amer.\ Math.\ Soc., to appear. Available as arXiv preprint arXiv:1407.8418.

\bibitem[La]{Landweber}
P.\ S.\ Landweber,
\emph{On equivariant maps between spheres with involutions},
Ann.\ of Math.\ 89 (1969), 125--137.

\bibitem[L]{L}
L.\ G.\ Lewis, Jr.,
\emph{Change of universe functors in equivariant stable homotopy theory.}
Fund.\ Math.\ 148 (1995) 117--158.

\bibitem[M1]{matric}
J.\ P.\ May,
\emph{Matric Massey products.}
J.\ Algebra 12 (1969) 533--568.


\bibitem[M2]{Alaska}
J. P. May et al,
\emph{Equivariant homotopy and cohomology theory.}
CBMS Regional Conference Series in Mathematics 91, Amer.\ Math.\ Soc., 1996. 

\bibitem[Me]{Me}
M. Merling.
Equivariant algebraic K-theory of G-rings. 
Math. Z. 285 (2017), no. 3-4, 1205--1248. 

\bibitem[Mo]{Mo}
F.\ Morel,
\emph{On the motivic $\pi_0$ of the sphere spectrum}, 
Axiomatic, enriched and motivic homotopy theory, 
NATO Sci.\ Ser.\ II Math.\ Phys.\ Chem., vol. 131, Kluwer Acad.\ Publ., 2004.

\bibitem[Ra]{Rav}
D. C. Ravenel, 
\emph{Complex Cobordism and Stable Homotopy Groups of Spheres}, second ed., 
AMS Chelsea Publishing 347, Amer. Math. Soc., Providence, RI, 2004.

\bibitem[Ri]{Ricka}
N.\ Ricka.
\emph{Subalgebras of the $\Z/2$-equivariant Steenrod Algebra.}
Homology Homotopy Appl. 17 (2015), 281--305. 

\bibitem[SS]{SS}
S.\ Schwede and B.\ Shipley. 
\emph{Stable model categories are categories of modules.}
Topology 42 (2003), 103--153. 

\bibitem[Se]{Se}
G. Segal.
Equivariant stable homotopy theory. 
Actes du Congr\`es International des Math\'{e}maticiens (Nice, 1970), Tome 2, pp. 59--63, 1971.

\bibitem[T]{Toda}
H.\ Toda, 
\emph{Composition methods in homotopy groups of spheres}.
Annals of Mathematics Studies 49, Princeton University Press, 1962.

\bibitem[V1]{VoevHF2}
V.\ Voevodsky, 
\emph{Motivic cohomology with $\mathbb{Z}/2$-coefficients}. 
Publ.\ Math.\ Inst.\ Hautes \'Etudes Sci.\ 98 (2003), 59--104.

\bibitem[V2]{Voev}
V. Voevodsky.
{\it Reduced power operations in motivic cohomology.}
Publ.\ Math.\ Inst.\ Hautes \'Etudes Sci.\ 98 (2003), 1--57. 

\bibitem[W]{Weib}
C.\ A.\ Weibel,
\emph{An Introduction to Homological Algebra.}
Cambridge Studies in Advanced Mathematics 38, Cambridge University Press, 1994.


\end{thebibliography}
\end{document}